\documentclass[11pt,reqno]{amsart}

\usepackage{amsmath}
\usepackage{amsfonts}
\usepackage{amssymb}
\usepackage{amsthm,color}
\usepackage{amsmath,amsthm,amssymb}
\usepackage{graphicx,mathtools}
\usepackage{cite}
\usepackage{mathabx} 
\usepackage{bbm}

\newtheorem{Theorem}{Theorem}[section]
\newtheorem{Definition}[Theorem]{Definition}
\newtheorem{Proposition}[Theorem]{Proposition}
\newtheorem{Lemma}[Theorem]{Lemma}

\newtheorem{Remark}[Theorem]{Remark}

\def\r{\right}

\def\a{\alpha}

\def\r{\rho}

\def\e{\epsilon}
\def\la{\lambda}

\def\o{\omega}

\def\F{{\mathcal F}}

\def\H{{\mathcal H}}
\def\D{{\Delta}}

\def\({\left(}
\def\){\right)}

\def\t{\theta}
\def\g{\gamma}

\def\jpx{\langle x \rangle}
\def\jpy{\langle y \rangle}
\def\jpl{\langle \lambda x \rangle}

\def\IF{\mathcal{F}^{-1}}

\numberwithin{equation}{section}

\begin{document}

\title[Well-posedness of the Helmholtz equation]{Well-posedness of the Helmholtz equation with rough coefficients}

\author{Peijun Li}
\address{SKLMS, ICMSEC, Academy of Mathematics and Systems Science, Chinese Academy of Sciences, Beijing 100190, China}
\email{lipeijun@lsec.cc.ac.cn}  

\author{Yichun Zhu$^*$}
\address{Renmin University of China, Beijing 100872,
China} 
\email{steven00931002@hotmail.com}

\thanks{* Corresponding author}



\begin{abstract}
We establish the well-posedness of the Helmholtz equation with rough and compactly supported coefficients in $\mathbb R^d$, $d=2,3$, under sharp regularity assumptions. Using a paraproduct calculus in rescaled weighted Besov spaces, we rigorously define the product between the solution and the coefficient at the lowest regularity level without renormalization. We establish existence and uniqueness for the Helmholtz equation with Sommerfeld radiation condition. Moreover, we  obtain explicit wavenumber dependent resolvent estimates for Faddeev-type operators in a general $L^p$ setting, including an $L^2$ theory relevant to scattering amplitudes. The results provide a sharp analytic foundation for wave propagation and scattering in highly irregular media.
\end{abstract}

\maketitle

\section{Introduction}\label{sec:intro}

In this paper, we establish the well-posedness of the Helmholtz equation,
\begin{equation}\label{H}
\Delta u(x) + k^2 u(x) + V_k(x)u(x) = g(x),\quad x\in\mathbb{R}^d, 
\end{equation}
subject to the Sommerfeld radiation condition
\begin{equation}\label{SRC}
\lim_{|x| \to \infty} |x|^{(d-1)/2} \big( \partial_{|x|} u - {\rm i}k u \big) = 0,
\end{equation}
where $d=2,3$ and $k>0$ denotes the wavenumber. Throughout this paper, we assume that both the coefficient $V_k$ and the source term $g$ are supported in a ball $B(0,R)$ centered at the origin with radius $R$. Both $V_k$ and $g$ are allowed to be rough, and their precise regularity assumptions will be specified later. In particular, the coefficient $V_k$ may depend on the wavenumber $k$ and takes the form
\[
V_k(x) = k^2 \varepsilon(x) + {\rm i}k \sigma(x) + \varrho(x),
\]
where $\varepsilon$, $\sigma$, and $\varrho$ represent the dielectric permittivity contrast, conductivity or absorption, and potential, respectively. Throughout, we assume that $\sigma$ is nonnegative. We emphasize that the roughness of $V_k$ may arise from any of these components, even when the others are smooth. This formulation provides a unified framework for studying Helmholtz equations with rough material properties, absorption coefficients, or potentials.

The investigation of the Helmholtz equation with rough coefficients is motivated by the modeling and analysis of wave propagation in complex and random media, where material parameters may have limited regularity and exhibit rapidly spatial variations. In many realistic applications, including geophysical exploration, nondestructive testing, medical imaging, and acoustic or electromagnetic scattering in composite or heterogeneous materials, the underlying medium cannot be adequately characterized by smooth coefficients. Instead, the dielectric permittivity, conductivity, or potential may be discontinuous, highly oscillatory, or even distributional, reflecting sharp interfaces, fine-scale heterogeneities, or unresolved microstructures. Such features naturally give rise to Helmholtz equations with rough, and possibly wavenumber dependent, coefficients, and pose significant analytical challenges that are not covered by classical smooth coefficient theory.

The Helmholtz equation is closely related to the stationary Schr\"{o}dinger equation; see, for instance, \cite{DZ19}. The corresponding time-dependent Schr\"{o}dinger equation is given by
\[
{\rm i}\partial_t u(t,x) - \big( \Delta + V(x) \big) u(t,x) = F(t,x),
\]
where $V(x)$ denotes the potential and $\Delta + V$ is the Schr\"{o}dinger operator. The analysis of outgoing solutions to this evolution equation, as discussed in \cite[Section~2.1]{DZ19}, leads naturally to the study of the resolvent operator $R_V(k)$, defined as the inverse of $\Delta + k^2 + V$. The existence and boundedness of $R_V(k)$ constitute a fundamental problem in scattering theory and are closely connected to the well-posedness of the Helmholtz equation.

\subsection{Meromorphic continuation and limiting absorption principle}

The study of the free resolvent $(\Delta + k^2)^{-1}$, denoted by $R_0(k)$, is fundamental in the analysis of the Helmholtz equation for outgoing solution. Since the symbol $k^2 - |\xi|^2$ of the operator $\Delta + k^2$ vanishes on the sphere $\{\xi\in\mathbb{R}^d: |\xi| = k \}$ in the frequency domain, its inverse is not a well-defined distribution in general. As a consequence, the homogeneous Helmholtz equation admits infinitely many nontrivial solutions unless additional conditions are imposed. For instance, it is known that radial solutions of the homogeneous Helmholtz equation can be expressed as linear combinations of Bessel functions,
\begin{equation*}
u(x) = C_1 J_m(k|x|) + C_2 Y_m(k|x|),
\end{equation*}
where $J_m$ and $Y_m$ denote the Bessel functions of the first and second kinds with order $m$, respectively, and $C_1, C_2 \in \mathbb{C}$ are constants. The existence of such nontrivial solutions, including bounded or oscillatory ones, is commonly referred to as resonance. This inherent non-uniqueness shows that the Helmholtz equation is not well-posed unless an additional physical condition, such as the Sommerfeld radiation condition, is imposed to single out the physically relevant solution.

From a physical perspective, one expects a unique response of the system once an incident wave is prescribed. In realistic media, this uniqueness is enforced by the presence of dissipation or absorption mechanisms, such as material losses, friction, or radiative damping, which prevent the persistence of nonphysical standing waves. The limiting absorption principle provides a rigorous mathematical framework to incorporate this physical effect into the analysis of the Helmholtz equation. More precisely, the operator $\Delta + k^2$ is regularized by introducing a small attenuation term, leading to the perturbed operator $\Delta + k^2 + {\rm i}\tau$ with $\tau>0$. The free resolvent $R_0(k)$ is then defined as the limit of $(\Delta + k^2 + {\rm i}\tau)^{-1}$ as $\tau \to 0^+$ in an appropriate operator topology. This procedure selects the outgoing solution and restores uniqueness. For general potentials $V$, the resolvent $R_V(k)$ admits a meromorphic continuation in the complex plane, which plays a central role in scattering theory. In particular, the meromorphic extension for $\Im{k}>0$ is called the outgoing resolvent. One of its important applications is the analysis of the long-time behavior of wave equations, leading to resonance expansions of wave fields; see \cite{DZ19}. The meromorphic continuation of the resolvent has been established for bounded potentials in \cite{DZ19} and was later extended to certain classes of unbounded potentials in \cite{LYZ21}.

\subsection{Point spectrum of Schr\"{o}dinger operators}

The well-posedness of the Helmholtz equation subject to the Sommerfeld radiation condition is also closely related to the absence of point spectrum for the associated Schr\"{o}dinger operator $\Delta+V$. The spectral properties of Schr\"{o}dinger operators are known to depend on the behavior of the potential function $V$. The question of identifying classes of potentials for which the Schr\"{o}dinger operator admits no positive eigenvalues originates from physics \cite{S69}. In 1959, Kato proved that if $V$ is bounded and satisfies $\lim_{|x| \to \infty} |x|V(x)=0$, then $\Delta+V$ has no positive eigenvalues \cite{K59}. In the 1970s, Agmon developed a systematic theory in his seminal works \cite{A70,A75}, based on weighted $L^2$-spaces and resolvent estimates for the outgoing resolvent $R_V(k)$, which yield the absence of embedded eigenvalues and singular continuous spectrum. This approach unified and significantly extended earlier results. More recently, Ionescu and Jerison established the absence of positive eigenvalues for potentials that are merely locally integrable and satisfy suitable decay assumptions \cite{IJ02}. We refer to \cite[Section XIII]{RS72} for a comprehensive discussion of this topic.

The well-posedness of the Helmholtz equation can then be established by combining the Fredholm alternative with the unique continuation principle, which in turn follows from the absence of positive eigenvalues for the corresponding Schr\"{o}dinger operator; see \cite{CK98,P00,GS04,E11,DZ19}. The main analytical tools include Agmon-type estimates \cite{A75}, Sobolev inequalities \cite{KRS87,BSSY15}, and Carleman estimates; see, for instance, \cite[Lemma~31.4]{E11} and \cite[Theorem~2.3]{IJ02}. In particular, M. Goldberg and W. Schlag \cite[Theorem 2]{GS04} proved the well-posedness for compactly supported potentials in $L^{3/2+}$ in dimension three. More recently, the well-posedness of the Helmholtz equation in the high frequency regime has been studied for compactly supported potentials in $H^{\alpha}$ with any $\alpha>-\frac{1}{2}$; see \cite{LPS08,CHL19,LLM21}. These works rely on the contraction mapping principle, which yields well-posedness only for sufficiently large wavenumbers and typically requires the potential to be independent of the wavenumber.

\subsection{Main results}

The free resolvent operator $-(\Delta + k^2)^{-1}$ can be represented by its Green's function $G_k(x,y)$, which satisfies the Sommerfeld radiation condition; see, for example, \cite{CK98,DZ19}. Explicitly, the Green's function is given by
\begin{equation*}
G_k(x,y) = \frac{{\rm i}}{4} \left( \frac{k}{2\pi |x-y|} \right)^{\frac{d}{2} - 1}
H^{(1)}_{\frac{d}{2} - 1}\left( k |x-y| \right),
\end{equation*}
where $H^{(1)}_{\nu}$ denotes the Hankel function of the first kind of order $\nu$. Using the Green's function representation, the boundary value problem \eqref{H}--\eqref{SRC} can be equivalently reformulated as a Lippmann--Schwinger integral equation; see, for example, \cite{CK98,P00}. Specifically, the solution $u$ satisfies
\[
u(x) = \int_{\mathbb{R}^d} G_k(x-y)V_k(y) u(y)dy - \int_{\mathbb{R}^d} G_k(x-y) g(y) dy,
\]
where the integrals are understood in a suitable weak sense. Formally, this equation may be written as
\begin{equation*}
u + (\Delta + k^2)^{-1}(V_k u) = (\Delta + k^2)^{-1} g,
\end{equation*}
which highlights the role of the free resolvent in the analysis of the Helmholtz equation with rough coefficients.

We begin by briefly explaining the heuristic motivation behind the regularity assumptions in our analysis. Suppose that the coefficient $V_k$ belongs to the Besov space $B^{\alpha}_{\infty,\infty}$ and that the solution $u$ lies in $B^{\beta}_{p,q}$. By Bony’s paraproduct decomposition, the product $V_k u$ is well-defined only when $\alpha + \beta > 0$, and its regularity cannot exceed $\min\{\alpha,\beta\}$. On the other hand, the resolvent operator $(\Delta + k^2)^{-1}$ acts as a smoothing operator of order two. Consequently, one expects
\[
(\Delta + k^2)^{-1}(V_k u) \in B^{\min\{\alpha,\beta\}+2}_{p,q}.
\]
Since this term must have the same regularity as $u$ in the Lippmann--Schwinger formulation, a formal consistency argument yields
\[
\beta = \min\{\alpha,\beta\} + 2, \quad \alpha + \beta > 0.
\]
These relations indicate that the regularity of $V_k$ must satisfy $\alpha > -1$. This heuristic calculation suggests that existing results based on such arguments are not optimal and motivates the development of a refined analytical framework. 

Our first main results are stated in the following theorems.

\begin{Theorem}\label{thm:main}
Let $r\in(0,1)$, $p_0\in(1,d/(d-2+2r))$, $\e>0$, and $\eta>\max\{(d+1)/2-d/p_0',1/2\}$. Suppose that $V_k\in B^{-r+\e}_{p_0',p_0'}$ is compactly supported and that $g\in B^{r-2+d/p_0'}_{(2p_0)',(2p_0)'}(\jpx^\eta)$. Then the Helmholtz equation subject to the Sommerfeld radiation condition admits a unique solution $u\in B^r_{2p_0,2p_0}(\jpx^{-\eta})$.
\end{Theorem}

\begin{Theorem}\label{thm:main0}
Let $p_0\in [1,\infty]$, $\e>0$, $\eta >(d+1)/2$, and $r \in (r_0,2)$, where
\[
r_0=  (d+1)/2- (d-1)/\max\{p_0,p_0'\}.
\]
Suppose that $V_k \in B^{r-2+\e}_{\infty,\infty}$ is compactly supported and that $g\in B^{r-2}_{p_0,p_0}(\jpx^{\eta})$. Then the Helmholtz equation subject to the Sommerfeld radiation condition admits a unique solution in $B^{r}_{p_0,p_0}(\jpx^{-\eta})$.
\end{Theorem}

As $r\to 0^+$, the admissible range of $p_0$ in Theorem~\ref{thm:main} approaches $(1,d/(d-2))$, which is equivalent to $p_0'\in(d/2,\infty)$, with the convention that $d/(d-2)=\infty$ when $d=2$. In Theorem~\ref{thm:main0}, we establish the well-posedness of the boundary value problem \eqref{H}--\eqref{SRC} for potentials $V_k\in L^p$ with any $p>d/2$, which recovers the result of \cite[Theorem~1]{GS04}. On the other hand, choosing $p_0=2$ in Theorem~\ref{thm:main0} gives $r_0=1$. We obtain the well-posedness of \eqref{H}--\eqref{SRC} for potentials $V_k\in B^s_{\infty,\infty}$ with any $s>-1$, confirming the expectation stated before Theorem~\ref{thm:main}. It is worth emphasizing that, although our results extend the admissible Besov regularity to negative smoothness indices arbitrarily close to $-1$, the integrability exponents imposed on $V_k$ remain strictly greater than $d/2$. Thus, our results improve the regularity requirements without extending beyond the critical integrability threshold considered in \cite{IJ02}.

One of the main ideas in the proofs of Theorems~\ref{thm:main} and~\ref{thm:main0} is inspired by the pioneering work of \cite{SU87}, where the Faddeev operator was introduced and utilized in the construction of {complex geometric optic solutions}. In the present paper, we develop this idea further by considering a broader class of operators that includes both the Faddeev operator and the free resolvent operator $-(\Delta + k^2)^{-1}$. Specifically, for any $k\in(0,\infty)$ and $\gamma\in\mathbb{R}^d$, we study the family of Faddeev-type resolvent operators
\[
(\Delta + r_{k,\gamma}^2 - 2\gamma\cdot\nabla)^{-1}, \quad r_{k,\gamma}:=(k^2+|\gamma|^2)^{1/2}.
\]
A precise definition of these operators via the limiting absorption principle is given in Theorem~\ref{thm:F}.

Our second main results establish wavenumber explicit estimates for the Faddeev-type operators, which extend Agmon's classical resolvent estimates in several significant directions. We establish Sobolev-type estimates for the associated operators, including an $L^2$-based wavenumber explicit estimate and a more general $L^p$-based wavenumber explicit estimate. We first give an intuitive overview of the results, while the precise statements are provided in Theorems~\ref{thm:F}, ~\ref{Hsg}~ and~\ref{Hsg:c} respectively. Even though we assume $d=2,3$ in this work, it is worth mentioning that Theorems~\ref{thm:F}, ~\ref{Hsg}~ and~\ref{Hsg:c} holds for all $d \geq 2$. In the following, the operator $(\Delta + r_{k,\g}^2 - 2\gamma \cdot \nabla)^{-1}$ is understood as the limit of $(\Delta + r_{k,\g}^2 - 2\gamma \cdot \nabla+ {\rm i} \tau)^{-1}$ as $\tau\to 0^{\pm}$.

\begin{Theorem}\label{thm:main:0.1}
Let $r \in \mathbb{R}$, $\eta> \tfrac12$, and $s \in [0,2]$. Then the following estimate holds:
\begin{equation*}
\big\|(\Delta + r_{k,\g}^2 - 2 \gamma \cdot \nabla)^{-1} f\big\|_{H^{r+s}(\langle x\rangle^{-\eta})}
\lesssim \frac{(1+r_{k,\gamma})^{s}}{\min\{r_{k,\gamma},1\} r_{k,\gamma}} \|f\|_{H^{r}(\langle x\rangle^{\eta})},
\end{equation*}
where the implicit constant is independent of $k$ and $\gamma$.
\end{Theorem}

\begin{Theorem}\label{thm:main:0.2} 
Let $\a \in [0,2]$, $p_1,p_2,q \in [1,\infty]$, and $r \in \mathbb{R}$ satisfy $1/p_2 - 1/p_1 = \alpha/d$.  Then, for any $\eta > \max\{(d+1)/2 - \alpha,1/2\}$, the following estimate holds:
\begin{equation*}
\big\|(\Delta + r_{k,\g}^2 - 2\gamma \cdot \nabla)^{-1} f\big\|_{B^{r}_{p_1,q}(\langle x\rangle^{-\eta})}
\lesssim \max\{ r_{k,\gamma}^{\eta}, r_{k,\gamma}^{-2} \} \|f\|_{B^{r-2+\a}_{p_2,q}(\langle x\rangle^{\eta})},
\end{equation*}
where the implicit constant is independent of $k$ and $\gamma$.
\end{Theorem}

\begin{Theorem}\label{thm:main:0.3}
Let $\lambda\in(0,1)$, $\alpha \in [0,2]$, and $\beta \in [0,2]$.  Let $p_1,p_2,q \in [1,\infty]$ and $r \in \mathbb{R}$ satisfy $
1/p_2-1/p_1=\alpha/d.$ Then, for any $\eta >\max\{d/2-\alpha,0\}$, the following estimate holds:
\[
\|(\Delta + r_{k,\g}^2 - 2\gamma \cdot \nabla)^{-1} f\|_{B^{r}_{p_1,q}(\langle x\rangle^{-\eta})}
\lesssim  |\g|^{-1} r_{k,\g}^{1-\beta} \max\{1,r_{k,\g}^{\beta-2}\}
\|f\|_{B^{r-2+\alpha+\beta+\eta}_{p_2,q}(\langle x\rangle^{\eta})},
\]
where the implicit constant is independent of  $k$ and $\gamma$.
\end{Theorem}

Another key ingredient in the proofs of Theorems~\ref{thm:main} and~\ref{thm:main0} is a rescaling argument motivated by the improved integrability properties of the free Helmholtz resolvent obtained through the limiting absorption principle. The central question is whether the regularizing effect of the resolvent $-(\Delta+k^2)^{-1}$ can compensate for the growth in spatial volume induced by rescaling with a small parameter. To address this question rigorously, we introduce rescaled weighted Besov spaces and establish their scaling properties. This framework enables us to control the rescaled equation by balancing the regularizing effect of the resolvent against the volume growth associated with rescaling. We refer the reader to the proofs of Lemmas~\ref{W1}, \ref{W2}, and~\ref{W3} for further details.

The paper is organized as follows. In Section~\ref{sec:N}, we introduce the notation and functional framework used throughout the paper, including the rescaled weighted Besov spaces. In Section~\ref{sec:RF}, we study the regularized Faddeev-type operator $(\Delta+k^2+|\gamma|^2-2\gamma\cdot\nabla+{\rm i}\tau)^{-1}$, where $\tau\in\mathbb R\setminus\{0\}$ is a regularization parameter. In Section~\ref{sec:H}, we define the Faddeev-type operators $(\Delta+k^2+|\gamma|^2-2\gamma\cdot\nabla)^{-1}$ via the limiting absorption principle as $\tau\to0^\pm$ and establish explicit wavenumber-dependent estimates. In Section~\ref{sec:W}, we prove Theorems~\ref{thm:main} and~\ref{thm:main0} using the unique continuation principle and the Fredholm alternative. Finally, the proofs of auxiliary results, including the paraproduct estimates and scaling properties of the rescaled weighted Besov spaces, are collected in the appendices.

\section{Notations and preliminaries}\label{sec:N}

Let $\mathcal{S}(\mathbb{R}^d)$ denote the Schwartz space of rapidly decaying smooth functions on $\mathbb{R}^d$, and let $\mathcal{S}'(\mathbb{R}^d)$ be its dual space consisting of tempered distributions. We denote by $\mathcal{D}(\mathbb{R}^d)$ the space of smooth functions with compact support on $\mathbb{R}^d$, and by $\mathcal{D}'(\mathbb{R}^d)$ the corresponding space of distributions.

Let $\mathbb{N}_0 := \mathbb{Z}^+ \cup \{0\}$. Let $\omega \in \mathbb{S}^{d-1}$ denote a vector on the $(d-1)$-dimensional unit sphere, and write $\xi = r\omega$ for the spherical coordinates of $\xi \in \mathbb{R}^d$. Throughout the paper, we use the notation $(f)_\lambda(x) := f(\lambda x)$ for any $f \in \mathcal{S}'(\mathbb{R}^d)$ and the rescaling parameter $\lambda>0$, and we denote by $p'$ the H\"{o}lder conjugate of $p$ such that $1/p+1/p'=1$.

For any multi-index $\alpha = (\alpha_1,\alpha_2,\ldots,\alpha_d) \in \mathbb{N}_0^d$, we define the differential operator $D^{\alpha}$ by
\[
D^{\alpha} := \frac{\partial^{|\alpha|}}{\partial x_1^{\alpha_1}\partial x_2^{\alpha_2}\cdots\partial x_d^{\alpha_d}},
\]
where $|\alpha| := \alpha_1+\alpha_2+\cdots+\alpha_d$. The Fourier transform of a function $f \in \mathcal{S}(\mathbb{R}^d)$ is defined by
\[
\mathcal{F}(f)(\xi) = \hat{f}(\xi) := \int_{\mathbb{R}^d} e^{-2\pi {\rm i}\xi \cdot x} f(x){\rm d}x,
\]
and the inverse Fourier transform is given by
\[
\mathcal{F}^{-1}(f)(x) = \check{f}(x) := \int_{\mathbb{R}^d} e^{2\pi {\rm i}\xi \cdot x}f(\xi){\rm d}\xi.
\]

We denote by $\mathcal{L}(X,Y)$ the set of all bounded linear mappings from a Banach space $X$ into a Banach space $Y$. An operator $T \in \mathcal{L}(X,Y)$ is said to be the strong limit of a sequence of operators $\{T_n\}_{n\in\mathbb{N}} \subset \mathcal{L}(X,Y)$ if 
\[
\lim_{n \to \infty} \|T_n-T\|_{\mathcal{L}(X,Y)}=0.
\]
The operator $T$ is called the weak limit of $\{T_{n}\}_{n\in \mathbb{N}}$ if 
\[
\lim_{n \to \infty} \|T_n x-Tx\|_{Y}=0,\quad\forall\, x \in X.
\]

\subsection{Rescaled weighted Besov spaces}

A pair of smooth functions $({\chi}, {\varphi})$ is called admissible partition pair if the following properties are satisfied:
\begin{enumerate}
\item
$\chi(\xi)+\sum_{n\geq 0}\varphi(2^{-n} \xi)=1$;
\medskip
\item
$\text{supp}\,\varphi(2^{-i} \xi) \cap \text{supp}\,\varphi(2^{-j}\xi)=\emptyset $ for $|i-j|\geq 2$;
\medskip
\item
$\text{supp}\,\varphi(2^{-i} \xi) \cap \text{supp}\,\chi( \xi) = \emptyset$ for all $i \geq 1$.
\end{enumerate}
 According to \cite[Chapter 2]{BCD11}, there exists an admissible partition pair $({\chi}, {\varphi})$ such that 
 \begin{enumerate}
 \item
 $ \text{supp}\,{\chi} \subseteq B(0,4/3)=:\mathcal{B}\ \text{and}\ \chi(\xi)=1\ \text{for all}\ |\xi|\leq 3/4;$
 \medskip
 \item
 $\text{supp}\,{\varphi} \subseteq \{\xi:\ 3/4 \leq |\xi| \leq 8/3\}=:\mathcal{C}.$
 \end{enumerate}

For simplicity, we define $\varphi_n(\cdot) := \varphi(2^{-n}\cdot)$ for $n \ge 0$ and set $\varphi_{-1}(\cdot) := \chi(\cdot)$. Moreover, for $n \in \mathbb{Z}$, we define $\chi_n:= \sum_{j\leq n} \varphi_j= \chi(2^{-n-1} \cdot)$.
In particular, we have $\chi_{-1}(\cdot) = \chi(\cdot)$. Similar notation is adopted for any admissible partition pair.
 
We denote $\langle x\rangle := (1+|x|^2)^{1/2}$. Following \cite[Section~6.1.2]{T06}, a positive function $\rho$ is said to be admissible of type $W(\eta)$, $\eta\geq 0$, if the following conditions hold:
\begin{enumerate}
\item
for every multi-index $N_0 \in \mathbb{N}_0^d$, there exists a constant $C_{N_0}>0$ such that$
|D^{N_0}\rho(x)| \leq C_{N_0} \rho(x)$ for all $x \in \mathbb{R}^d$;
\medskip
\item
there exist constants $c>0$ and $\eta \ge 0$ such that $
0 < \rho(x) \leq c \rho(y) \langle x-y\rangle^{\eta}$ for all $x,y \in \mathbb{R}^d$.
\medskip
\end{enumerate}
It is straightforward to verify that $\langle x\rangle^{\eta}$ is admissible for all $\eta \geq 0$. Moreover, the following inequality holds:
\begin{equation}\label{eq:ad:1}
\langle x\rangle^{-\eta} \lesssim \langle y\rangle^{-\eta} \langle x-y\rangle^{\eta}, 
\quad \forall\, \eta \geq 0. 
\end{equation}
Hence, both $\langle x\rangle^{\eta}$ and $\langle x\rangle^{-\eta}$ are admissible weights of type $W(\eta)$.
 
For any admissible weight $\rho$ and $p \in [1,\infty]$, we denote by $L^p(\rho)$ the weighted Lebesgue space consisting of all measurable functions $f$ such that
\[
\|f\|_{L^p(\rho)} :=
\begin{cases}
\displaystyle
\left( \int_{\mathbb{R}^d} |f(x)|^p\rho(x)^p dx \right)^{1/p}, & 1 \leq p < \infty,\\[8pt]
\displaystyle
\operatorname*{ess\,sup}_{x \in \mathbb{R}^d} |f(x)|\rho(x), & p = \infty.
\end{cases}
\]

Next, we introduce the weighted Besov space $B^{r}_{p,q}(\rho)$. Let $f \in \mathcal{S}'(\mathbb{R}^d)$. We define the Littlewood--Paley operators $\Delta_j$ and $S_j$ by
\[
\Delta_j f := \mathcal{F}^{-1}\big(\varphi_j \hat{f}\big),
\quad
S_j f := \sum_{n=-1}^{j-1} \Delta_n f,
\quad j \ge -1.
\]
Let $r \in \mathbb{R}$ and $(p,q) \in [1,\infty]^2$. For any admissible weight $\rho$, the weighted Besov norm $\|\cdot\|_{B^{r}_{p,q}(\rho)}$ is defined by
\[
\|f\|_{B^{r}_{p,q}(\rho)} :=
\begin{cases}
\displaystyle
\left( \sum_{n \ge -1} 2^{nrq}\|\Delta_n f\|_{L^p(\rho)}^q \right)^{1/q}, & 1 \leq q < \infty,\\[8pt]
\displaystyle
\sup_{n \ge -1} 2^{nr}\|\Delta_n f\|_{L^p(\rho)}, & q = \infty.
\end{cases}
\]
In particular, when the weight is trivial, i.e., $\rho \equiv 1$, we write $B^{r}_{p,q} := B^{r}_{p,q}(\rho)$ and $\|f\|_{B^{r}_{p,q}} := \|f\|_{B^{r}_{p,q}(\rho)}$, recovering the classical Besov space.

The following lemma clarifies the relationship between weighted Besov spaces and the classical Besov spaces; see \cite[Theorem~6.5]{T06} for the proof.

\begin{Lemma}\label{Pre1}
Let $r \in \mathbb{R}$ and $p,q \in [1,\infty]$. For any admissible weight $\rho$, the following norm equivalence holds: 
\[
\|f\|_{B^{r}_{p,q}(\rho)} \sim \|\rho f\|_{B^{r}_{p,q}}.
\]
\end{Lemma}

By Lemma~\ref{Pre1} and \cite[Corollary~2.70]{BCD11}, the weighted Besov space $B^{r}_{p,q}(\rho)$ is independent of the particular choice of the partition of unity $(\chi,\varphi)$ for any admissible weight $\rho$. Moreover, weighted Besov spaces share many of the fundamental properties of the classical nonhomogeneous Besov spaces; we refer the reader to \cite[Chapter~2.7]{BCD11} for a comprehensive discussion. 

The following result follows directly from the definition of the weighted Besov space.
\begin{Lemma}\label{Pre00}
Let $\rho_1$ and $\rho_2$ be admissible weights with $\rho_1 \leq c\rho_2$ for some constant $c>0$. For any $p,q \in [1,\infty]$ and $r \in \mathbb{R}$, the following estimate holds:
\[
\|f\|_{B^{r}_{p,q}(\rho_1)} \leq c \|f\|_{B^{r}_{p,q}(\rho_2)}.
\]
\end{Lemma}

The following lemma follows directly from the monotonicity of $l^p$ spaces; see \cite[Section~2.7.1]{T83}.

\begin{Lemma}\label{Pre0}
Let $\rho$ be an admissible weight. The following embedding properties hold:
\begin{enumerate}
\item
For any $p \in [1,\infty]$, $1 \leq q_1 \leq q_2 \leq \infty$, and $r \in \mathbb{R}$, 
\[
\|f\|_{B^{r}_{p,q_2}(\rho)} \lesssim \|f\|_{B^{r}_{p,q_1}(\rho)}.
\]
\item
For any $p,q_1,q_2 \in [1,\infty]$ and $-\infty < r_1 < r_2 < \infty$, 
\[
\|f\|_{B^{r_1}_{p,q_1}(\rho)} \lesssim \|f\|_{B^{r_2}_{p, q_2}(\rho)}.
\]
\end{enumerate}
Here the implicit constants are independent of the admissible weight $\rho$.
\end{Lemma}

Furthermore, combining Lemma~\ref{Pre1} with the monotonicity of $l^p$ spaces (see \cite[Section~2.7.1]{T83}) yields the following embedding result for weighted Besov spaces.

\begin{Lemma}\label{Pre000}
For any $1 \leq p_1 \leq p_2 \leq \infty$, $q \in [1,\infty]$, and $0 < r_2 \leq r_1 < \infty$ satisfying $
r_1 - d/p_1 = r_2 -d/p_2$, the following continuous embedding holds: 
\[
\|f\|_{B^{r_2}_{p_2,q}(\rho)} \lesssim \|f\|_{B^{r_1}_{p_1,q}(\rho)}.
\]
\end{Lemma}

The following lifting property of weighted Besov spaces is proved in \cite[Theorem~6.1.3]{T06}.

\begin{Lemma}\label{lifting}
Let $r,s \in \mathbb{R}$ and $(p,q) \in [1,\infty]^2$. For any admissible weight $\rho$, the following norm equivalence holds: 
\[
\big\|(I-\Delta)^{s/2} f\big\|_{B^{r}_{p,q}(\rho)} \sim \|f\|_{B^{r+s}_{p,q}(\rho)},
\]
where $I$ denotes the identity operator.
\end{Lemma}

We next extend the definition of weighted Besov spaces and introduce the rescaled weighted Besov spaces. Let $r \in \mathbb{R}$ and $(p,q) \in [1,\infty]^2$. For any admissible weight $\rho$ and scaling parameter $\lambda>0$, we define the norm $\|\cdot\|_{B^{r}_{p,q}(\rho,\lambda)}$ by
\[
\|f\|_{B^{r}_{p,q}(\rho,\lambda)} :=
\begin{cases}
\displaystyle
\left( \sum_{n \geq -1} 2^{nrq} \|\Delta_n f\|_{L^p(\rho,\lambda)}^{q} \right)^{1/q}, & 1 \leq q < \infty,\\[8pt]
\displaystyle
\sup_{n \geq -1} 2^{nr}\|\Delta_n f\|_{L^p(\rho,\lambda)}, & q = \infty,
\end{cases}
\]
where $L^{p}(\rho,\lambda) := L^p((\rho)_\lambda)$ denotes the weighted Lebesgue space associated with the rescaled weight $(\rho)_\lambda(x) := \rho(\lambda x)$. It is immediate from the definition that
\[
\|f\|_{B^{r}_{p,q}(\rho,\lambda)} = \|f\|_{B^{r}_{p,q}((\rho)_\lambda)}.
\]

Although the nonhomogeneous Besov spaces do not enjoy the same scaling properties as their homogeneous counterparts (cf. \cite[Chapter~2.3]{BCD11}), the rescaled weighted Besov spaces introduced here still satisfy useful rescaling properties. The proofs are given in Appendix~\ref{p:R1}.

\begin{Lemma}\label{R2:n}
Let $r \in \mathbb{R}$ and $(p,q) \in [1,\infty]^2$. For any $\lambda \in (0,\infty)$, $\eta \in [0,\infty)$, and admissible weight $\rho$ of type $W(\eta)$, the following estimate holds:
\begin{equation*}
\begin{aligned}
\|(f)_\lambda\|_{B^{r}_{p,q}(\rho,\lambda)}\lesssim_{\,r,q,\eta}
\begin{cases}
 \lambda^{-d/p}\, C^{}_{\la,r,q} \|f\|_{B^{r}_{p,q}(\rho)},\quad&\la \in (0,1),\\
 \lambda^{-d/p}\, C^{}_{\la,r,q,\eta} \|f\|_{B^{r}_{p,q}(\rho)},\quad&\la \in [1,\infty),
 \end{cases}
\end{aligned}
\end{equation*}
where 
\[
\begin{aligned}
C^{}_{\la,r,q}:= 
\begin{cases}
\la^r,\quad &r<0,\\
(1+|\ln(\la)|)^{1/q'},\quad &r=0,\\
1,\quad &r>0,
\end{cases}
\quad
C^{}_{\la,r,q,\eta}:= 
\begin{cases}
\la^{\eta},\quad &r<\eta,\\
\la^{\eta}(1+\ln(\la))^{1/q},\quad &r=\eta,\\
\la^{r},\quad &r>\eta,
\end{cases}
\end{aligned}
\]
and the implicit constant is independent of $\lambda$.
\end{Lemma}

Lemma \ref{R2:n} implies that, under suitable conditions,
\[
\|(f)_\lambda\|_{B^{r}_{p,q}(\rho,\lambda)}
\lesssim_{\,r,q} \lambda^{r-d/p}\|f\|_{B^{r}_{p,q}(\rho)},
\]
where the exponent $r-d/p$ is commonly referred to as the differential dimension. This quantity plays an important role in embedding theorems; see Lemma~\ref{Pre000}.

\subsection{Interpolations}

Let $X$ and $Y$ be Banach spaces with $Y \subseteq X$. We begin by recalling the definition of real interpolation spaces. For $x \in X$, the $K$-functional associated with the pair $(X,Y)$ is defined by
\[
K(t,x;X,Y) := \inf_{\substack{x = a+b \\ a \in X,\, b \in Y}}
\big( \|a\|_X + t\|b\|_Y \big), \quad t>0.
\]

Let $L^p_{\ast}(0,\infty)$ denote the Lebesgue space on $(0,\infty)$ with respect to the measure $t^{-1} dt$. For any $\theta \in (0,1)$ and $p \in [1,\infty]$, the real interpolation space $(X,Y)_{\theta,p}$ is defined as the set of all $x \in X$ such that
\[
\|x\|_{(X,Y)_{\theta,p}} := \big\| t^{-\theta} K(t,x;X,Y) \big\|_{L^p_{\ast}(0,\infty)} < \infty.
\]
For further properties of real interpolation spaces, we refer the reader to \cite[Chapter~1]{L99}.

The following result on bounded linear operators between real interpolation spaces is fundamental; see \cite[Theorem~1.6]{L99}.

\begin{Lemma}\label{real:interpolation}
Let $(X_1,Y_1)$ and $(X_2,Y_2)$ be real interpolation couples. Suppose that $L \in \mathcal{L}(X_1,X_2) \cap \mathcal{L}(Y_1,Y_2)$. Then, for every $\theta \in (0,1)$ and $p \in [1,\infty]$, the operator $L$ extends to a bounded linear operator $
L \in \mathcal{L}\big((X_1,Y_1)_{\theta,p},\,(X_2,Y_2)_{\theta,p}\big)$, and the following estimate holds 
\[
\|L\|_{\mathcal{L}((X_1,Y_1)_{\theta,p},(X_2,Y_2)_{\theta,p})} \leq (\|L\|_{\mathcal{L}(X_1,X_2)})^{1-\theta} (\|L\|_{\mathcal{L}(Y_1,Y_2)})^\theta.
\]
\end{Lemma}

For any $r \in \mathbb{R}$ and $p \in [1,\infty)$, we define the Sobolev space $W^{r,p} := B^{r}_{p,p}$
and the weighted Sobolev space $W^{r,p}(\rho) := B^{r}_{p,p}(\rho).$ In particular, for $p=2$, we write $
H^{r} := B^{r}_{2,2}$ and $H^{r}(\rho) := B^{r}_{2,2}(\rho)$, which correspond to the $L^2$-based Sobolev space and its weighted counterpart, respectively. For more general results on the equivalence and interrelations between Besov spaces, H\"{o}lder--Zygmund spaces, Bessel-potential spaces, and Sobolev spaces, we refer the reader to \cite[Section~2.3.5]{T83}.

The following result on real interpolation between Besov spaces with different integrability exponents is proved in \cite[Section~2.4.1]{T78} and \cite[Section~2.4.3]{T83}.

\begin{Lemma}\label{I:Besov}
Let $r_1, r_2 \in \mathbb{R}$ and $p_1, p_2 \in [1,\infty)$. For any $\theta \in (0,1)$, define
\[
r := (1-\theta) r_1 + \theta r_2, \quad 1/p := (1-\theta)/p_1 + \theta/p_2.
\]
Then the following real interpolation identity holds:
\[
\big( B^{r_1}_{p_1,p_1}, B^{r_2}_{p_2,p_2} \big)_{\theta,p}
= B^{r}_{p,p}.
\]
\end{Lemma}

\subsection{Besov spaces on bounded domains}\label{sec:B:bd}

Let $\Omega \subset \mathbb{R}^d$ be a nonempty bounded domain, and let $p,q \in [1,\infty]$ and $r \in \mathbb{R}$. The Besov space $B^{r}_{p,q}(\Omega)$ is defined by
\[
B^{r}_{p,q}(\Omega) \coloneqq 
\bigl\{ f \in \mathcal{D}'(\Omega) : \exists\, g \in B^{r}_{p,q}(\mathbb{R}^d) \ \text{such that}\ f = g\big|_{\Omega} \bigr\}.
\]
The space $B^{r}_{p,q}(\Omega)$ is equipped with the quotient norm
\[
\|f\|_{B^{r}_{p,q}(\Omega)} \coloneqq 
\inf \bigl\{ \|g\|_{B^{r}_{p,q}(\mathbb{R}^d)} :
g \in B^{r}_{p,q}(\mathbb{R}^d),\ g\big|_{\Omega} = f \bigr\},
\]
which makes $B^{r}_{p,q}(\Omega)$ a Banach space.

Let $\partial\Omega$ be a compact $(d-1)$-dimensional Lipschitz manifold, and let $\{O_j\}_{j=1}^N$ be an open covering of $\partial\Omega$. We choose a finite family of Lipschitz charts $\kappa_j : U_j \subset \mathbb{R}^{d-1} \to O_j \subset \partial\Omega, j=1,\dots,N$, together with a smooth partition of unity $\{\phi_j\}_{j=1}^N$ subordinate to $\{O_j\}$, i.e., $\phi_j \in C^\infty(\partial\Omega)$, $\operatorname{supp}\phi_j \subset O_j$, and $
\sum_{j=1}^N \phi_j(x) = 1$ for all $x \in \partial\Omega.$

Let $p,q \in [1,\infty]$ and $r \in \mathbb{R}$. The Besov space on $\partial\Omega$ is defined by
\[
B^{r}_{p,q}(\partial\Omega)
:= \bigl\{ f \in \mathcal{D}'(\partial\Omega) \;:\;
(\phi_j f)\circ \kappa_j \in B^{r}_{p,q}(\mathbb{R}^{d-1})
\ \text{for each } j=1,\dots,N \bigr\},
\]
where $B^{r}_{p,q}(\mathbb{R}^{d-1})$ denotes the usual Besov space on $\mathbb{R}^{d-1}$. The norm on $B^{r}_{p,q}(\partial\Omega)$ is given by
\[
\|f\|_{B^{r}_{p,q}(\partial\Omega)}
:= \sum_{j=1}^N
\bigl\| (\phi_j f)\circ \kappa_j \bigr\|_{B^{r}_{p,q}(\mathbb{R}^{d-1})},
\]
which makes $B^{r}_{p,q}(\partial\Omega)$ a Banach space. Different choices of the atlas $\{\kappa_j\}$ and the partition of unity $\{\phi_j\}$ lead to equivalent norms. In particular, when $\partial\Omega$ is smooth, the resulting space is independent of these choices. For further details, we refer the reader to \cite[Section~3.2.2]{T83}.

We define the trace operator $Tu$ as the restriction of $u$ to the boundary $\partial \Omega$. The following lemma extends the trace operator to Sobolev spaces; see \cite[Section~3.3.3]{T83}. 

\begin{Lemma}\label{Trace:bd:1}
Let $p \in (1,\infty)$ and $r \in (1/p,\infty)$. Assume that $U \subset \mathbb{R}^d$ is a bounded domain with smooth boundary $\partial U$. Then the trace operator $T$ extends to a bounded linear operator satisfying
\[
\|Tu\|_{W^{r-1/p,p}(\partial U)} \lesssim \|u\|_{W^{r,p}(\mathbb{R}^d)}.
\]
\end{Lemma}

\section{The regularized Faddeev-type operator}\label{sec:RF}

For any $\gamma \in \mathbb{R}^d$ and $\tau \in \mathbb{R}$, we define the symbol
$m_{k,\gamma,\tau} : \mathbb{R}^d \to \mathbb{C}$ by
\begin{equation*}
m_{k,\gamma,\tau}(\xi)
:= |\xi|^2 - k^2 - |\gamma|^2 + 2{\rm i}\gamma \cdot \xi - {\rm i}\tau .
\end{equation*}
Let $\psi : \mathbb{R} \to [0,1]$ be a smooth function such that
$\operatorname{supp}\psi \subset B(0,4)$ and $\psi \equiv 1$ on $B(0,2)$, where $B(0,r)$ is the ball centered at $0$ with radius $r$. For any $r_0 \in \mathbb{R}$ and $\epsilon_0>0$, we define the rescaled cutoff
function $\psi(\cdot; r_0, \epsilon_0) : \mathbb{C} \to [0,1]$ by
\begin{equation*}
\psi(x; r_0,\epsilon_0) := \psi\big( |x-r_0|\epsilon_0^{-1} \big).
\end{equation*}

We now introduce the regularized Faddeev type operator $\mathcal{H}_{k,\gamma,\tau}(\cdot\,;\psi,\epsilon_0)$.

\begin{Definition}
Fix $s \in [0,2]$, $\epsilon_0 \in (0,\min\{r_{k,\gamma}/8, 1\})$, and $
\tau \in (-\epsilon_0^2,0)\cup(0,\epsilon_0^2)$.
We define the operator $
(I-\Delta)^{s/2}\mathcal{H}_{k,\gamma,\tau}(\,\cdot\,;\psi,\epsilon_0)
:\mathcal{S}(\mathbb{R}^d)\to \mathcal{S}'(\mathbb{R}^d)$
by duality as follows:
\begin{equation*}
\big\langle (I-\Delta)^{s/2}\mathcal{H}_{k,\gamma,\tau}(f;\psi,\epsilon_0),\, g \big\rangle
=
I^{(1,s)}_{k,\gamma,\epsilon_0,\tau}(f,g)
+
I^{(2,s)}_{k,\gamma,\epsilon_0,\tau}(f,g)
+
I^{(3,s)}_{k,\gamma,\epsilon_0,\tau}(f,g),
\end{equation*}
where
\[
\begin{aligned}
I^{(1,s)}_{k,\gamma,\epsilon_0,\tau}(f,g)
:=& \int_{0}^{\infty}\!\!\int_{\mathbb{S}^{d-1}}
\frac{(1+r^2)^{s/2}\hat f(\omega r)\bar{\hat g}(\omega r)}
{m_{k,\gamma,\tau}(\omega r)}
\bigl(1-\psi(r; r_{k,\gamma},\epsilon_0)\bigr) r^{d-1}d\omega dr,
\\[4pt]
I^{(2,s)}_{k,\gamma,\epsilon_0,\tau}(f,g)
:=&
\int_{0}^{\infty}\!\!\int_{\mathbb{S}^{d-1}}
\frac{(1+r^2)^{s/2}\,
\hat f(\omega r_{k,\gamma})\bar{\hat g}(\omega r_{k,\gamma})}
{m_{k,\gamma,\tau}(\omega r)}\psi(r; r_{k,\gamma},\epsilon_0) r^{d-1} d\omega dr,
\\[4pt]
I^{(3,s)}_{k,\gamma,\epsilon_0,\tau}(f,g)
:=&
\int_{0}^{\infty}\!\!\int_{\mathbb{S}^{d-1}}
\frac{(1+r^2)^{s/2}
\big(\hat f(\omega r) \bar{\hat g}(\omega r)-
\hat f(\omega r_{k,\gamma})\bar{\hat g}(\omega r_{k,\gamma})
\big)}{m_{k,\gamma,\tau}(\omega r)}\psi(r; r_{k,\gamma},\epsilon_0)r^{d-1}d\omega dr,
\end{aligned}
\]
for all $f,g \in \mathcal{S}(\mathbb{R}^d)$.
\end{Definition}

Once the integrals $I^{(j,s)}_{k,\gamma,\epsilon_0,\tau}(f,g)$, $j=1,2,3$, are well-defined, the above definition can be written in the Fourier multiplier form
\begin{equation}\label{r:F}
\big\langle (I-\Delta)^{s/2}\mathcal{H}_{k,\gamma,\tau}(f;\psi,\epsilon_0), g \big\rangle
=\int_{\mathbb{R}^d}
\frac{(1+|\xi|^2)^{s/2}}{m_{k,\gamma,\tau}(\xi)}\hat f(\xi)\bar{\hat g}(\xi)d\xi.
\end{equation}
Thus, the operator $\mathcal{H}_{k,\gamma,\tau}$ can be viewed as the Fourier multiplier associated with the symbol $m_{k,\gamma,\tau}$.

We next verify that the integrals $I^{(j,s)}_{k,\gamma,\epsilon_0,\tau}(f,g)$, $j=1,2,3$, are well-defined. We begin with the following result.

\begin{Lemma}\label{F1}
The integral $I^{(1,s)}_{k,\gamma,\epsilon_0,\tau}(f,g)$ is well-defined. Moreover, for any $f,g \in \mathcal{S}(\mathbb{R}^d)$, the following estimate holds:
\[
\big| I^{(1,s)}_{k,\gamma,\epsilon_0,\tau}(f,g) \big|\lesssim
\frac{(1+r_{k,\gamma})^{s}}{\epsilon_0 r_{k,\gamma}}\|f\|_{L^2(\mathbb{R}^d)}\|g\|_{L^2(\mathbb{R}^d)}.
\]
\end{Lemma}

\begin{proof}
We begin with the case $s=0$. Observe that
\begin{equation}\label{F1:1}
\big||\xi|^2 - r_{k,\gamma}^2\big|
= |r+r_{k,\gamma}||r-r_{k,\gamma}|
\gtrsim \epsilon_0 (r_{k,\gamma}+\varepsilon_0)
\quad \text{on } \operatorname{supp}\bigl(1-\psi(r; r_{k,\gamma},\varepsilon_0)\bigr).
\end{equation}
Hence, we obtain
\[
\big|I^{(1,0)}_{k,\gamma,\epsilon_0,\tau}(f,g)\big|
\lesssim \int_{0}^{\infty}\!\!\int_{\mathbb{S}^{d-1}}
\frac{|\hat f(\omega r)| |\bar{\hat g}(\omega r)|}
{\epsilon_0 (\epsilon_0 + r_{k,\gamma})} r^{d-1}d\omega dr.
\]
Applying H\"{o}lder's inequality together with Plancherel’s identity yields
\[
\big|I^{(1,0)}_{k,\gamma,\epsilon_0,\tau}(f,g)\big|
\lesssim \frac{1}{\epsilon_0 r_{k,\gamma}}
\|f\|_{L^2(\mathbb{R}^d)} \|g\|_{L^2(\mathbb{R}^d)}.
\]

We next consider the case $s=2$. Using \eqref{F1:1} once again, we have 
\[
\begin{aligned}
\frac{1+r^2}{|r^2-r_{k,\gamma}^2|} & =\frac{1+r_{k,\gamma}^2 + r^2-r_{k,\gamma}^2}{|r^2-r_{k,\gamma}^2|}  \lesssim 1 + \frac{1+ r_{k,\gamma}^2}{|r^2-r_{k,\gamma}^2|}\\
 & \lesssim 1 + \frac{1+ r_{k,\gamma}^2}{\epsilon_0 (r_{k,\gamma}+\epsilon_0)} \lesssim \frac{(1+ r_{k,\gamma})^2}{\epsilon_0 r_{k,\g}}\quad \text{on supp} \big(1- \psi(r; r_{k,\gamma},\epsilon_0)\big), 
\end{aligned}
\]
which implies
\[
|I^{(1,2)}_{k,\gamma,\epsilon_0,\tau}(f,g)|  \lesssim \int_0^\infty\!\! \int_{\mathbb{S}^{d-1}} \frac{(1+r^2)|\hat{f}(\omega r)||\bar{\hat{g}}(\omega r)|}{|r^2-r_{k,\gamma}^2|} r^{d-1} d\omega dr \lesssim \frac{(1+ r_{k,\gamma})^2}{\epsilon_0 r_{k,\gamma}}  \|f\|_{L^2(\mathbb R^d)} \|g\|_{L^2(\mathbb R^d)}.
\]

For any $s \in (0,2)$, we apply an interpolation argument between the cases $s=0$ and $s=2$ to obtain
\[
|I^{(1,s)}_{k,\gamma,\epsilon_0,\tau}(f,g)|  \lesssim \int_0^\infty \!\!\int_{\mathbb{S}^{d-1}} \frac{(1+r^2)^{s/2}}{|r^2-r_{k,\gamma}^2|^{s/2}} \frac{|\hat{f}(\omega r)||\bar{\hat{g}}(\omega r)|}{|r^2-r_{k,\gamma}^2|^{1-s/2}} r^{d-1} d\omega dr \lesssim \frac{(1+ r_{k,\gamma})^s}{\epsilon_0 r_{k,\gamma}}  \|f\|_{L^2(\mathbb R^d)} \|g\|_{L^2(\mathbb R^d)},
\]
which completes the proof. 
\end{proof}

\begin{Lemma}\label{F2}
The integral $I^{(2,s)}_{k,\gamma,\epsilon_0,\tau}(f,g)$ is well-defined. Moreover, for any $f,g \in \mathcal{S}(\mathbb{R}^d)$ and any $\eta > 1/2$, the following estimate holds:
\[
\big| I^{(2,s)}_{k,\gamma,\epsilon_0,\tau}(f,g) \big|\lesssim
\frac{(1+r_{k,\gamma})^{s}}{\epsilon_0 r_{k,\gamma}}\|f\|_{L^2(\langle x\rangle^{\eta})}
\|g\|_{L^2(\langle x\rangle^{\eta})}.
\]
\end{Lemma}

\begin{proof}
By Fubini's theorem, we obtain
\[
\begin{aligned}
|I^{(2,s)}_{k,\gamma,\epsilon_0,\tau}(f,g)| \lesssim  \int_{\mathbb{S}^{d-1}} \!\! \left| \int_{\mathbb{R}^+}\ \frac{(1+r^2)^{s/2}}{m_{k,\g,\tau}(\o r)} \psi(r; r_{k,\g},\e_0)r^{n-1} dr\right| \  |\hat{f}(\o r_{k,\g})\bar{\hat{g}}(\o r_{k,\g})| d\o.
\end{aligned}
\]
To proceed, we rewrite the symbol $m_{k,\gamma,\tau}$ in polar coordinates as
\begin{equation}\label{F2:m}
m_{k,\gamma,\tau}(r,\omega)= r^2- r_{k,\gamma}^2 + 2 {\rm i} r \gamma_{\omega} - {\rm i} \tau,\quad  \gamma_{\omega}= \gamma \cdot \omega.
\end{equation}

When $|\gamma_{\omega}| \ge \varepsilon_0/4$, we estimate the symbol $m_{k,\gamma,\tau}$ using its imaginary part. In this case,
\begin{equation}\label{F2:3}
|m_{k,\gamma,\tau}| \geq  \epsilon_0 r/2 - \tau \geq \epsilon_0 (r_{k,\gamma}-\epsilon_0)/2 - \tau \geq \epsilon_0 r_{k,\gamma}/8,
\end{equation}
where we have used the facts that
$r \ge r_{k,\gamma}-\epsilon_0$ on $\operatorname{supp}\psi(r; r_{k,\gamma},\epsilon_0)$, $\epsilon_0 \in (0,r_{k,\gamma}/4)$, and $\tau \in [0,\epsilon_0^2) \subset [0,\epsilon_0 r_{k,\gamma}/4)$. Therefore, we have
\[
\left| \int_{\mathbb{R}^+} \frac{(1+r^2)^{s/2}}{m_{k,\g,\tau}(\o r)} \psi(r; r_{k,\g},\e_0)r^{d-1} dr\right| \lesssim \frac{(1+ r_{k,\g})^{s} r_{k,\g}^{d-1}}{\e_0 r_{k,\g}} 2\e_0=(1+ r_{k,\g})^{s} r_{k,\g}^{d-2} \ .
\]

When $|\gamma_{\omega}| < \epsilon_0/4$, we extend the symbol $m_{k,\gamma,\tau}$ to the complex plane and deform the integration path accordingly. We introduce the contour $\Gamma$ as follows. If $\tau<0$, we define
\[
\Gamma
:=
\bigl\{ r = r_{k,\gamma} + \epsilon_0 e^{{\rm i}\theta} : \theta \in (0,\pi) \bigr\}\cup
\bigl\{ r \in [r_{k,\gamma}-\epsilon_0,  r_{k,\gamma}+\epsilon_0] \bigr\}
=: \Gamma_{\epsilon_0} \cup \Gamma_0.
\]
If $\tau>0$, we instead define
\[
\Gamma
:=
\bigl\{ r = r_{k,\gamma} + \epsilon_0 e^{{\rm i}\theta} : \theta \in (\pi,2\pi) \bigr\}\cup
\bigl\{ r \in [r_{k,\gamma}-\epsilon_0, r_{k,\gamma}+\epsilon_0] \bigr\}
=: \Gamma_{\epsilon_0} \cup \Gamma_0.
\]

Let $A_{\gamma}$ denote the subset of $\mathbb{S}^{d-1}$ defined by $A_{\gamma} := \bigl\{ \omega \in \mathbb{S}^{d-1} : |\gamma_{\omega}| < \epsilon_0/4 \bigr\}.$ To handle the possible poles enclosed by the contour $\Gamma$, we solve the equation $m_{k,\gamma,\tau}(r,\omega)=0$, which yields
\[
r_{1,2}= -{\rm i}\gamma_{\omega} \pm \bigl(r_{k,\gamma}^2 - \gamma_{\omega}^2 + {\rm i}\tau \bigr)^{1/2}.
\]
Here the complex square root is understood in the principal branch, i.e.,
\[
\bigl(r_{k,\gamma}^2 - \gamma_{\omega}^2 + {\rm i}\tau \bigr)^{1/2}
= \bigl( (r_{k,\gamma}^2 - \gamma_{\omega}^2)^2 + \tau^2 \bigr)^{1/4}
e^{\frac{{\rm i}}{2} \arctan\Bigl( \frac{\tau}{r_{k,\gamma}^2 - \gamma_{\omega}^2} \Bigr)}.
\]
Since $|\gamma_{\omega}| < \epsilon_0/4 < r_{k,\gamma}/16$ and $\tau \in (0,\epsilon_0^2)$, we have
\[
\bigl( (r_{k,\gamma}^2 - \gamma_{\omega}^2)^2 + \tau^2 \bigr)^{1/4}
\ge \frac{r_{k,\gamma}}{2},\quad \arctan\Bigl( \frac{\tau}{r_{k,\gamma}^2 - \gamma_{\omega}^2} \Bigr)
\leq \arctan \Bigl( \frac{1}{4} \Bigr),
\]
which in turn implies
\begin{equation}\label{F2:4}
\Re \bigl( r_{k,\gamma}^2 - \gamma_{\omega}^2 + {\rm i}\tau \bigr)^{1/2}
\ge \frac{r_{k,\gamma}}{4}.
\end{equation}

Here we only consider the case $\tau>0$, as the case $\tau<0$ can be treated in an analogous manner. On the semicircular part $\Gamma_{\varepsilon_0}$ of the contour, we decompose the symbol $m_{k,\gamma,\tau}$ as
\[
m_{k,\gamma,\tau}
= (r-r_{k,\gamma}+{\rm i}\gamma_{\omega})(r+r_{k,\gamma}-{\rm i}\gamma_{\omega})
+ 2(r-r_{k,\gamma}){\rm i}\gamma_{\omega}
-{\rm i}\tau - \gamma_{\omega}^2 .
\]
Since $|\gamma_{\omega}|<\epsilon_0/4$, we have
\[
\begin{aligned}
|r-r_{k,\gamma}+{\rm i}\gamma_{\omega}|
&= \big|\epsilon_0\cos\theta + {\rm i}(\gamma_{\omega}+\epsilon_0\sin\theta)\big| \\
&= \big(\epsilon_0^2 + 2\epsilon_0\gamma_{\omega}\sin\theta + \gamma_{\omega}^2\big)^{1/2}
\geq \big(\epsilon_0^2/2 + \gamma_{\omega}^2\big)^{1/2},
\end{aligned}
\]
where we have used that $\theta\in(\pi,2\pi)$ on $\Gamma_{\epsilon_0}$. Hence, we obtain
\begin{align}\label{F2:2}
|m_{k,\gamma,\tau}| &\geq \big|(r-r_{k,\gamma}+{\rm i}\gamma_{\omega})(r+r_{k,\gamma}-{\rm i}\gamma_{\omega})\big|
- |2(r-r_{k,\gamma})\gamma_{\omega}| - |\tau| - \gamma_{\omega}^2 \notag\\
&\geq \epsilon_0(r_{k,\gamma}-\epsilon_0) - \epsilon_0^2/2 - \epsilon_0^2 - \epsilon_0^2/16
\geq \epsilon_0 r_{k,\gamma}/4,
\end{align}
where we used the assumptions $\epsilon_0\in(0,r_{k,\gamma}/4)$ and $\tau\in(0,\epsilon_0^2)$.
This lower bound implies that
\[
\left|\int_{\Gamma_{\epsilon_0}} \frac{(1+r^2)^{s/2}}{m_{k,\gamma,\tau}(\omega r)} r^{d-1} dr\right|
\lesssim \frac{2\pi \epsilon_0 (1+r_{k,\gamma})^{s}}{\epsilon_0 r_{k,\gamma}}
r_{k,\gamma}^{d-1} \lesssim (1+r_{k,\gamma})^{s} r_{k,\gamma}^{d-2}.
\]
In particular, there are no poles of the integrand on $\Gamma_{\epsilon_0}$ for all
$\omega\in\mathbb{S}^{d-1}$ satisfying $|\gamma_{\omega}|<\epsilon_0/4$.

On the other hand, poles may potentially lie on the segment $\Gamma_0$. We first observe that, since the semicircular arc $\Gamma_{\epsilon_0}$ lies in the right half-plane and in the lower half of the complex plane, the only possible pole enclosed by the contour $\Gamma$ is
 \begin{equation}\label{F2:pole}
 r:= -{\rm i} \gamma_{\omega} + (r_{k,\gamma}^2 -\gamma_{\omega}^2 + {\rm i} \tau)^{1/2},\quad 0<\gamma_{\omega}<\epsilon_0/4.
 \end{equation}
 
We next examine the imaginary part of the candidate pole \eqref{F2:pole}. A direct computation shows that the imaginary and real parts of $(r_{k,\gamma}^2 -\gamma_{\omega}^2 + {\rm i} \tau)^{1/2}$ are given by
\begin{align*}
\Im(r_{k,\gamma}^2 -\gamma_{\omega}^2 + {\rm i} \tau)^{1/2}&=\left(\frac{-(r_{k,\gamma}^2 -\gamma_{\omega}^2) + \big((r_{k,\gamma}^2 -\gamma_{\omega}^2)^2+\tau^2\big)^{1/2}}{2}\right)^{1/2},\\
\Re(r_{k,\gamma}^2 -\gamma_{\omega}^2 + {\rm i} \tau)^{1/2}&=\left(\frac{(r_{k,\gamma}^2 -\gamma_{\omega}^2) + \big(r_{k,\gamma}^2 -\gamma_{\omega}^2)^2+\tau^2\big)^{1/2}}{2}\right)^{1/2}.
\end{align*}
Define the function $f(x):= x^2 - \big(\Im(r_{k,\gamma}^2 -x^2 +{\rm i} \tau)^{1/2}\big)^2$. A pole lies on $\Gamma_0$ only if $f(x)=0$. Differentiating $f$ yields
\[
f'(x)= x+ \frac{x(r_{k,\gamma}^2-x^2)}{\big((r_{k,\gamma}^2 -x^2)^2+\tau^2\big)^{1/2}}>0, \quad x\in (0,r_{k,\gamma}), 
\]
which shows that $f$ is strictly increasing on $(0,r_{k,\gamma})$.

The above analysis shows that there exists at most one value $\gamma_{\omega} \in (0,\epsilon_0/4)$ for which $\Im(r)=0$.
We denote this critical value by $\Theta_{\gamma}(\tau)$. Since $f(0)<0$ and $f$ is strictly increasing on $(0,r_{k,\gamma})$,
a pole lies on or inside the contour $\Gamma$ if and only if $\gamma_{\omega} \ge \Theta_{\gamma}(\tau)$. More precisely, the following two cases occur:
\begin{enumerate}
\item
If $\gamma_{\omega} < \Theta_{\gamma}(\tau)$, then $\Im(r)>0$ and no pole lies inside or on the contour $\Gamma$.
\item
If $\Theta_{\gamma}(\tau) \leq \gamma_{\omega} \leq \epsilon_0/4$, we show that the pole $r$ lies inside or on the contour $\Gamma$.
\end{enumerate}

By the preceding discussion, the condition $\gamma_{\omega} \geq \Theta_{\gamma}(\tau)$ implies that $\Im(r)\leq 0$. Therefore, it suffices to show that $|r-r_{k,\gamma}|<\epsilon_0$. Let $a= {r_{k,\gamma}^2 -\gamma_{\omega}^2}$. A direct computation yields
 \begin{align*}
 \Re(r_{k,\gamma}^2 -\gamma_{\omega}^2 +{\rm i} \tau)^{1/2} - r_{k,\gamma} &= \Re(a + {\rm i} \tau)^{1/2} - \sqrt{a} + \sqrt{a}- r_{k,\gamma}\\
 &= \frac{\frac{\sqrt{a^2+\tau^2}-a}{2}}{\big(\frac{a+\sqrt{a^2+\tau^2}}{2}\big)^{1/2} +\sqrt{a}} - \frac{\gamma_\omega^2}{\sqrt{a}+r_{k,\gamma}}\\
 &= \frac{\tau^2}{\big((2a+2\sqrt{a^2+\tau^2})^{1/2}+2\sqrt{a}\big)\big(a+\sqrt{a^2+\tau^2}\big) }- \frac{\gamma_\omega^2}{\sqrt{a}+r_{k,\gamma}}. 
\end{align*}
Since $|\gamma_{\omega}|\leq \epsilon_0/4$ and $\tau \in (0,\epsilon_0^2)$, we have $a \geq r_{k,\gamma}^2 - \epsilon_0^2/16 \geq 16\epsilon_0^2$, which implies
 \[
| \Re(r_{k,\g}^2 -\g_{\o}^2 + {\rm i} \tau)^{1/2} - r_{k,\gamma}| \leq \frac{\tau^2}{8a\sqrt{a}} + \frac{\gamma_\omega^2}{2\sqrt{a}} \leq \frac{\epsilon_0}{256}.
 \]
Moreover, since $\Theta_{\gamma}(\tau)\leq \gamma_{\omega}\leq \epsilon_0/4$, we have $\gamma_{\omega} - \Im(a + {\rm i}\tau)^{1/2} > 0$. Therefore, 
 \begin{align*}
 |r-r_{k,\gamma}| & \leq | \Re(r_{k,\gamma}^2 -\gamma_{\omega}^2 + {\rm i} \tau)^{1/2} - r_{k,\gamma}| + |\gamma_{\omega}- \Im(r_{k,\gamma}^2 -\gamma_{\omega}^2 + {\rm i} \tau)^{1/2}| \\
 & \leq | \Re(r_{k,\gamma}^2 -\gamma_{\omega}^2 + {\rm i} \tau)^{1/2} - r_{k,\gamma}| + \gamma_{\omega} <\epsilon_0/2, 
 \end{align*}
which proves the claim.
 
We define the sets $A_{\gamma,\tau}^+$ and $\partial^{\mathrm{in}} A_{\gamma,\tau}^+$ by 
\[
 A_{\gamma,\tau}^+:=\bigl\{ \omega \in \mathbb{S}^{d-1}: \Theta_{\gamma}(\tau) < \gamma \cdot \omega \leq \epsilon_0/4 \bigr\},\quad  \partial^{\rm in}A_{\gamma,\tau}^+:= \bigl\{ \omega \in \mathbb{S}^{d-1}: \Theta_{\gamma}(\tau) = \gamma \cdot \omega  \bigr\}.
 \]
Geometrically, the set $A_{\gamma,\tau}^+$ forms a spherical band on $\mathbb{S}^{d-1}$. From the preceding analysis, the union
$A_{\gamma,\tau}^+ \cup \partial^{\mathrm{in}} A_{\gamma,\tau}^+$ is precisely the subset of $A_{\gamma}$ consisting of all directions $\omega$ for which the contour $\Gamma$ encloses a pole or passes through one. In particular, $\partial^{\mathrm{in}} A_{\gamma,\tau}^+$ characterizes the set of $\omega \in A_{\gamma}$ for which a pole lies exactly on the contour $\Gamma$.

Since $\partial^{\mathrm{in}} A_{\gamma,\tau}^+$ is a set of Lebesgue measure zero on $\mathbb{S}^{d-1}$, the possible presence of poles on the contour $\Gamma$ does not affect the definition of $I^{(2,s)}_{k,\gamma,\epsilon_0,\tau}$. We apply Cauchy's integral theorem to account for possible poles lying either inside $\Gamma$ or on $\Gamma$, and obtain
\[
\bigg|\frac{1}{2\pi {\rm i}}\int_{\Gamma}
\frac{(1+r^2)^{s/2}}{m_{k,\gamma,\tau}(\omega r)} r^{d-1}dr\bigg|\lesssim
\frac{(1+r_1^2)^{s/2} r_1^{d-1}}{|r_1-r_2|}\lesssim
(1+r_{k,\gamma})^{s} r_{k,\gamma}^{d-2},
\]
where we used the bound $(1+r_1^2)^{s/2} r_1^{d-1} \lesssim (1+r_{k,\gamma})^{s+d-1}$,
the estimate \eqref{F2:4}, and the fact that
\[
|r_1-r_2|= 2(r_{k,\gamma}^2 - \gamma_{\omega}^2 + {\rm i}\tau)^{1/2}
\gtrsim \big|(r_{k,\gamma}^2 - \gamma_{\omega}^2 + {\rm i}\tau)^{1/2}\big|
\gtrsim r_{k,\gamma}.
\]
Since the support of $\psi(r; r_{k,\gamma},\epsilon_0)$ is contained in $\Gamma_0$, we have
\begin{align*}
\bigg|\int_{0}^{\infty}\frac{(1+r^2)^{s/2}}{m_{k,\gamma,\tau}(\omega r)}
\psi(r; r_{k,\gamma},\epsilon_0) r^{d-1}dr\bigg|
&=\bigg|\int_{\Gamma_0}\frac{(1+r^2)^{s/2}}{m_{k,\gamma,\tau}(\omega r)}
\psi(r; r_{k,\gamma},\epsilon_0)r^{d-1}dr\bigg| \\
&\lesssim \bigg| \int_{\Gamma_{\epsilon_0}} \frac{(1+r^2)^{s/2}}{|m_{k,\gamma,\tau}(\omega r)|} r^{d-1}dr
\bigg|+\bigg| \int_{\Gamma}\frac{(1+r^2)^{s/2}}{m_{k,\gamma,\tau}(\omega r)} r^{d-1}dr
\bigg| \\
&\lesssim (1+r_{k,\gamma})^{s}r_{k,\gamma}^{d-2}.
\end{align*}

Combining the cases $|\gamma_{\omega}|>\epsilon_0/4$ and $|\gamma_{\omega}|\leq \epsilon_0/4$, we obtain 
\begin{align*}
|I^{(2,s)}_{k,\gamma,\epsilon_0,\tau}(f,g)| &\lesssim  \int_{\mathbb{S}^{d-1}} \frac{(1+ r_{k,\gamma})^{s}}{r_{k,\gamma}}  |\hat{f}(\omega r_{k,\gamma})\bar{\hat{g}}(\omega r_{k,\gamma}) | r_{k,\gamma}^{d-1} d\omega \\
&\lesssim \frac{(1+ r_{k,\gamma})^{s}}{\epsilon_0 r_{k,\gamma}} \|\hat{f}\|_{H^{\eta}} \|{\hat{g}}\|_{H^{\eta}}, \quad \eta >1/2, 
\end{align*}
where we have applied  Lemma \ref{R2:n}. By Plancherel's identity, this estimate can be rewritten as
\[
|I^{(2,s)}_{k,\gamma,\epsilon_0,\tau}(f,g)| \lesssim \frac{(1+ r_{k,\gamma})^{s}}{\epsilon_0 r_{k,\gamma}} \|f\|_{L^2(\langle x\rangle^{\eta})} \|g\|_{L^2(\langle x\rangle^{\eta})},
\]
which completes the proof. 
\end{proof}

We have established the result for the case $\tau>0$. The case $\tau<0$ can be treated in an analogous manner. The only difference is that, instead of the sets $A_{\gamma,\tau}^+$ and $\partial^{\rm in}A_{\gamma,\tau}^+$, we consider the sets $A_{\gamma,\tau}^-$ and $\partial^{\rm in}A_{\gamma,\tau}^-$, defined by
\[
A_{\gamma,\tau}^- := \bigl\{ \omega \in \mathbb{S}^{d-1} : -\Theta_{\gamma}(\tau) > \gamma \cdot \omega \geq -\epsilon_0/4 \bigr\}, 
\quad
\partial^{\rm in}A_{\gamma,\tau}^- := \bigl\{ \omega \in \mathbb{S}^{d-1} : \Theta_{\gamma}(\tau) = -\gamma \cdot \omega \bigr\}.
\]

\begin{Lemma}\label{F3}
The integral $I^{(3,s)}_{k,\gamma,\epsilon_0,\tau}(f,g)$ is well-defined. Moreover, for any $f,g \in \mathcal{S}$ and $\eta>1/2$, the following estimate holds:
\[
\big| I^{(3,s)}_{k,\gamma,\epsilon_0,\tau}(f,g) \big|\lesssim
\frac{(1+r_{k,\gamma})^{s}}{\epsilon_0 r_{k,\gamma}}
\|f\|_{L^2(\langle x\rangle^{\eta})} \|g\|_{L^2(\langle x\rangle^{\eta})}.
\]
\end{Lemma}
 
\begin{proof}
We begin by observing that
\begin{align*}
I^{(3,s)}_{k,\gamma,\epsilon_0,\tau}(f,g)&\lesssim \int_{r_{k,\gamma}-4\epsilon_0}^{r_{k,\gamma}+4\epsilon_0}\!\! \int_{\mathbb{S}^{d-1}} \frac{(1+r^2)^{s/2} |\hat{f}(\omega r)\bar{\hat{g}}(\omega r) - \hat{f}(\omega r_{k,\gamma})\bar{\hat{g}}(\omega r_{k,\gamma})|}{r_{k,\gamma}|r-r_{k,\gamma}|} r^{d-1} d\omega dr \\
&\lesssim  \frac{(1+r_{k,\gamma})^s}{r_{k,\gamma}}\int_{r_{k,\gamma}-4\epsilon_0}^{r_{k,\gamma}+4\epsilon_0} \int_{\mathbb{S}^{d-1}}\frac{ |\hat{f}(\omega r)\bar{\hat{g}}(\omega r) - \hat{f}(\omega r_{k,\gamma})\bar{\hat{g}}(\omega r_{k,\gamma})|}{|r-r_{k,\gamma}|} r^{d-1} d\omega dr.
\end{align*}
Applying H\"{o}lder's inequality, we obtain
\begin{align}\label{eq:F3:1}
&\left|\int_{\mathbb{S}^{d-1}}\frac{ |\hat{f}(\omega r)\bar{\hat{g}}(\omega r) - \hat{f}(\omega r_{k,\gamma})\bar{\hat{g}}(\omega r_{k,\gamma})|}{|r-r_{k,\gamma}|} r^{d-1} d\omega\right|^2 \notag\\
&\lesssim \frac{1}{|r-r_{k,\gamma}|} \int_{\mathbb{S}^{d-1}} |\bar{\hat{g}}(\omega r)|^2 r^{d-1}d\omega \int_{\mathbb{S}^{d-1}} |\hat{f}(\omega r) - \hat{f}(\omega r_{k,\gamma})|^2 r^{d-1} d\omega \notag\\
&\quad +\frac{1}{|r-r_{k,\gamma}|} \int_{\mathbb{S}^{d-1}} |\hat{f}(\omega r_{k,\gamma})|^2 r^{d-1}d\omega \int_{\mathbb{S}^{d-1}} |\bar{\hat{g}}(\omega r) - \bar{\hat{g}}(\omega r_{k,\gamma})|^2 r^{d-1} d\omega:= J_1+ J_2. 
\end{align}

Since  $r \in [r_{k,\gamma}-4\epsilon_0, r_{k,\gamma}+4\epsilon_0 ]$ and we assume without loss of generality that $r \in [r_{k,\gamma}, r_{k,\gamma}+4\epsilon_0 ]$. Using H\"{o}lder's inequality again yields
\[
\begin{aligned}
\int_{\mathbb{S}^{d-1}} |\hat{f}(\omega r) - \hat{f}(\omega r_{k,\gamma})|^2 r^{d-1} d\omega 
=& \int_{\mathbb{S}^{d-1}} \Big| \int_{r}^{r_{k,\gamma}} \frac{d}{d \epsilon} \hat{f}(\epsilon \omega) d\epsilon \Big|^2 r^{d-1} d\omega \\
\leq& |r-r_{k,\gamma}| \int_{\mathbb{S}^{d-1}} \int_{r_{k,\gamma}}^{r_{k,\gamma}+4\epsilon_0} \Big|\frac{d}{d\epsilon} \hat{f}(\epsilon \omega) \Big|^2 d\epsilon\, r^{d-1} d\omega. 
\end{aligned}
\]
 By Fubini's theorem and Lemma~\ref{Trace:bd:1}, we obtain
\[
\begin{aligned}
\int_{\mathbb{S}^{d-1}} |\hat{f}(\omega r) - \hat{f}(\omega r_{k,\gamma})|^2 r^{d-1} d\omega 
\lesssim& |r-r_{k,\gamma}| \int_{r_{k,\gamma}}^{r_{k,\gamma}+4\epsilon_0} \!\!  \int_{\mathbb{S}^{d-1}} \sum_{|\alpha|=1}|D^{\alpha}\hat{f}(\epsilon \omega)|^2 d\omega\, r^{d-1} d\epsilon\\
\lesssim&\  |r-r_{k,\gamma}|r^{d-1} \int_{r_{k,\gamma}}^{r_{k,\gamma}+4\epsilon_0}   \epsilon^{-2}\|(\hat{f})_{\epsilon}\|_{H^{1}(\mathbb{S}^{d-1})}^2 d\epsilon\\
 \lesssim&\  |r-r_{k,\gamma}|  r^{d-1} \int_{r_{k,\gamma}}^{r_{k,\gamma}+4\epsilon_0}   \epsilon^{-2}\|(\hat{f})_{\epsilon}\|_{H^{3/2}}^2 d\epsilon. 
\end{aligned}
\]
By  Lemma \ref{R2:n}, there holds
\[
\|(\hat{f})_{\epsilon}\|_{H^{3/2}}^2 \lesssim \epsilon^{-d} \max\{\epsilon,1\}^{3}\|\hat{f}\|_{H^{3/2}}^2. 
\]
Since $\epsilon_0 \in (0,r_{k,\gamma}/4)$ and $r \in [r_{k,\gamma}-4\epsilon_0, r_{k,\gamma}+4\epsilon_0]$, there holds $\epsilon^{-d+1}r^{d-1} \lesssim_d 1 $.  Therefore, we have
\[
\begin{aligned}
\int_{\mathbb{S}^{d-1}} |\hat{f}(\omega r) - \hat{f}(\omega r_{k,\gamma})|^2 r^{d-1} d\omega 
&\lesssim  |r-r_{k,\gamma}|  \int_{r_{k,\gamma}}^{r_{k,\gamma}+4\epsilon_0}   \epsilon^{-3} \max\{\epsilon,1\}^{3} \|\hat{f}\|_{H^{3/2}}^2 d\epsilon\\
&\lesssim  \epsilon_0^{-2} |r-r_{k,\gamma}| \|\hat{f}\|_{H^{3/2}}^2. 
\end{aligned}
\]

On the other hand, we have the elementary estimate
\[
\int_{\mathbb{S}^{d-1}} |\hat{f}(\omega r) - \hat{f}(\omega r_{k,\gamma})|^2 r^{d-1} d\omega \lesssim 
\int_{\mathbb{S}^{d-1}} |\hat{f}(\omega r) |^2 r^{d-1} d\omega + r^{d-1}r_{k,\gamma}^{-d+1}  \int_{\mathbb{S}^{d-1}}
|\hat{f}(\omega r_{k,\gamma}) |^2 r_{k,\gamma}^{d-1} d\omega.  
\]
Since $r \in [r_{k,\gamma}-4\epsilon_0, r_{k,\gamma}+4\epsilon_0]$ and $\e_0 \in (0, r_{k,\g}/8)$, it follows that $r^{d-1}r_{k,\g}^{d-1}\lesssim_d 1$. Let $\e \in (0,1)$. Similar to the previous proof, we have
\[
\int_{\mathbb{S}^{d-1}} |\hat{f}(\omega r) - \hat{f}(\omega r_{k,\gamma})|^2 r^{d-1} d\omega \lesssim \epsilon_0^{-1} \|\hat{f}\|_{H^{1/2+\e}}^2. 
\]

 By interpolation, we obtain
\begin{equation}\label{eq:F3:2}
\int_{\mathbb{S}^{d-1}} |\hat{f}(\omega r) - \hat{f}(\omega r_{k,\gamma})|^2 r^{d-1} d\omega \lesssim \epsilon_0^{-1-\e} 
|r-r_{k,\gamma}|^{\epsilon} \|\hat{f}\|_{H^{1/2+\epsilon}}^2,
\end{equation}
for some possibly another $\e>0$. Substituting \eqref{eq:F3:2} in \eqref{eq:F3:1} yields 
\[
J_1 \lesssim \frac{\|\hat{f}\|_{H^{1/2+\epsilon}} \|\bar{\hat{g}}\|_{H^{1/2+\epsilon}} }{\epsilon_0^{1+\e}|r-r_{k,\gamma}|^{1-\epsilon}},\quad \epsilon>0,
\]
and the same estimate holds for $J_2$. Consequently, we have 
\begin{align*}
I^{(3,s)}_{k,\gamma,\epsilon_0,\tau}(f,g)
\lesssim  \frac{(1+r_{k,\gamma})^s}{r_{k,\gamma}}\int_{r_{k,\gamma}-\epsilon_0}^{r_{k,\gamma}+\epsilon_0}
\frac{\|\hat{f}\|_{H^{1/2+\epsilon}} \|\hat{g}\|_{H^{1/2+\epsilon}} }{\epsilon_0^{1+\e}|r-r_{k,\gamma}|^{1-\epsilon}}dr \lesssim  \frac{(1+r_{k,\gamma})^s}{\epsilon_0 r_{k,\gamma}}\|\hat{f}\|_{H^{1/2+\epsilon}} \|\hat{g}\|_{H^{1/2+\epsilon}}. 
\end{align*}
Finally, applying Plancherel's identity yields the desired estimate, which completes the proof.
\end{proof}

Therefore, the regularized Faddeev-type operator is well-defined. For notational convenience, we henceforth write $
\mathcal{H}_{k,\gamma,\tau}:= \mathcal{H}_{k,\gamma,\tau}(\,\cdot\,;\psi,\epsilon_0).$ We are now in a position to state the main theorem of this section.

\begin{Theorem}\label{thm:RF}
Let $r \in \mathbb{R}$, $s \in [0,2]$, and $\eta>1/2$. Assume that $\epsilon_0 \in \bigl(0,\min\{r_{k,\gamma}/4, 1\}\bigr)$ and
$\tau \in (-\epsilon_0^2,0)\cup(0,\epsilon_0^2)$. Then the operator $\mathcal{H}_{k,\gamma,\tau}(\,\cdot\,;\psi,\epsilon_0)$
extends to a bounded linear operator from $H^{r}(\langle x\rangle^{\eta})$ into $H^{r+s}(\langle x\rangle^{-\eta})$, and the following estimate holds:
\[
\|\mathcal{H}_{k,\gamma,\tau} f\|_{H^{r+s}(\langle x\rangle^{-\eta})}
\lesssim \frac{(1+r_{k,\gamma})^{s}}{\epsilon_0  r_{k,\gamma}}
\|f\|_{H^{r}(\langle x\rangle^{\eta})},
\]
where the implicit constant is independent of $\tau$, $\epsilon_0$, and $r_{k,\gamma}$. Moreover, for all $f \in \mathcal{S}(\mathbb{R}^d)$, the operator $\mathcal{H}_{k,\gamma,\tau}$ is a two-sided inverse of
$-(\Delta+k^2+|\gamma|^2-2\gamma\cdot\nabla+{\rm i}\tau)$ in the sense of distributions, i.e.,
\[
-(\Delta+k^2+|\gamma|^2-2\gamma\cdot\nabla+{\rm i}\tau) \mathcal{H}_{k,\gamma,\tau} f = f
= -\mathcal{H}_{k,\gamma,\tau} (\Delta+k^2+|\gamma|^2-2\gamma\cdot\nabla+{\rm i}\tau) f.
\]
\end{Theorem}

\begin{proof}
By Lemmas~\ref{F1}--\ref{F3}, we have
\[
\big\langle (I-\Delta)^{s/2}\mathcal{H}_{k,\gamma,\tau} f , g \big\rangle\lesssim
\frac{(1+r_{k,\gamma})^{s}}{\epsilon_0 r_{k,\gamma}}
\|f\|_{L^2(\langle x\rangle^{\eta})}
\|g\|_{L^2(\langle x\rangle^{\eta})}.
\]
By duality, this yields
\[
\|(I-\Delta)^{s/2}\mathcal{H}_{k,\gamma,\tau} f\|_{L^2(\langle x\rangle^{-\eta})}
\lesssim \frac{(1+r_{k,\gamma})^{s}}{\epsilon_0 r_{k,\gamma}}
\|f\|_{L^2(\langle x\rangle^{\eta})}.
\]
Since $(I-\Delta)^{r/2}$ commutes with $\mathcal{H}_{k,\gamma,\tau}$, applying Lemma~\ref{lifting} yields the first statement of the theorem.

For the second statement, let $f,g \in \mathcal{S}(\mathbb{R}^d)$. By integration by parts, Plancherel's identity, and the representation \eqref{r:F}, we obtain
\[
\int_{\mathbb{R}^d} (\Delta+k^2+|\gamma|^2-2\gamma\cdot\nabla+{\rm i}\tau)
\mathcal{H}_{k,\gamma,\tau} f \overline{g} dx=\int_{\mathbb{R}^d} f \overline{g} dx.
\]
It is clear that the set $\Theta_{\gamma}(\tau)$, consisting of directions for which the multiplier $m_{k,\gamma,\tau}$ has a pole on the positive real radial axis, has measure zero on $\mathbb{S}^{d-1}$. This property is essential to the definition of $\mathcal{H}_{k,\gamma,\tau}$ and justifies the identity above. Moreover, since $\mathcal{H}_{k,\gamma,\tau}$ commutes with the differential operator $\Delta+k^2+|\gamma|^2-2\gamma\cdot\nabla+{\rm i}\tau$, the identity also holds with the order of the two operators reversed. This completes the proof.
\end{proof}

\section{The Faddeev-type operator}\label{sec:H}

This section is devoted to providing a precise definition of the Faddeev-type operators $\mathcal{H}^{\pm}_{k,\gamma}$ via the regularization procedure introduced in Section~\ref{sec:RF}, together with the corresponding wavenumber explicit estimates.
We begin by defining the operators $\mathcal{H}^{\pm}_{k,\gamma}$.

\begin{Definition}
Let $r \in \mathbb{R}$, $s \in [0,2]$, and $\eta \in (1/2,\infty)$. We define the operator $\mathcal{H}^{+}_{k,\gamma}$ as the strong limit of $\mathcal{H}_{k,\gamma,\tau}$ as $\tau \to 0^{+}$ in $\mathcal{L}\big(H^{r}(\langle x\rangle^{\eta}),\, H^{r+s}(\langle x\rangle^{-\eta})\big)$. Similarly, we define $\mathcal{H}^{-}_{k,\gamma}$ as the strong limit of $\mathcal{H}_{k,\gamma,\tau}$ as $\tau \to 0^{-}$ in the same operator topology.
\end{Definition}

\begin{Remark}\label{r:H}
We refer to the operator $\mathcal{H}^{+}_{k} := \mathcal{H}^{+}_{k,0}$ as the inverse scattering operator corresponding to outgoing waves, and to $\mathcal{H}^{-}_{k} := \mathcal{H}^{-}_{k,0}$ as the inverse scattering operator corresponding to incoming waves. A direct consequence of Theorem~\ref{thm:F} is that $\mathcal{H}^{\pm}_{k} = -(\Delta + k^{2})^{-1}$,
where the inverse is understood in the sense of the limiting absorption principle or, equivalently, via meromorphic continuation. 
\end{Remark}

Weak convergence of the operators $\mathcal{H}_{k,0,\tau}$ as $\tau \to 0^{+}$ was established in \cite{LLM21}. In the present work, we not only consider a substantially broader class of Faddeev-type operators, but also prove strong convergence of $\mathcal{H}_{k,\gamma,\tau}^{\pm}$, as stated in the theorem below. In addition, our analysis yields wavenumber explicit estimates that are valid not only in the high frequency regime considered in \cite{LLM21}, but also for general wavenumbers.
  
\begin{Theorem}\label{thm:F}
Let $\lambda \in (0,1)$, $r \in \mathbb{R}$, $\eta>1/2$, and $s \in [0,2]$. Then the operators $\mathcal{H}^{\pm}_{k,\gamma}$ extend to bounded linear operators from $H^{r}(\langle x\rangle^{\eta})$ into $H^{r+s}(\langle x\rangle^{-\eta})$, and satisfy the estimate
\begin{equation}\label{eq:F}
\|\mathcal{H}^{\pm}_{k,\gamma} f\|_{H^{r+s}(\langle x\rangle^{-\eta},\lambda)}\lesssim
\frac{\lambda^{-2\eta}(1+r_{k,\gamma})^{s}}{\min\{r_{k,\gamma}, 1\} r_{k,\gamma}}\|f\|_{H^{r}(\langle x\rangle^{\eta},\lambda)},
\end{equation}
where the implicit constant is independent of $\lambda$ and $r_{k,\gamma}$. 
\end{Theorem}

\begin{proof}
We only consider the case $\tau>0$, since the case $\tau<0$ can be treated analogously. Let $f,g \in L^{2}(\langle x\rangle^{\eta})$ and let $0<\tau_1<\tau_2<\epsilon_0$.
Observe that
\[
\max\bigl\{ |2\gamma\cdot\xi-\tau_1|,\; |2\gamma\cdot\xi-\tau_2| \bigr\}
\geq \frac{|\tau_1-\tau_2|}{2},
\]
which implies
\begin{align*}
\left|
\frac{1}{m_{k,\gamma,\tau_1}(\xi)}-\frac{1}{m_{k,\gamma,\tau_2}(\xi)}
\right|
&=
\frac{|\tau_2-\tau_1|}
{|\xi^2-r_{k,\gamma}^2+2{\rm i}\gamma\cdot\xi-{\rm i}\tau_1|\,
 |\xi^2-r_{k,\gamma}^2+2{\rm i}\gamma\cdot\xi-{\rm i}\tau_2|} \\
&\leq
\frac{|\tau_2-\tau_1|}
{|\xi^2-r_{k,\gamma}^2|\,
\bigl((\xi^2-r_{k,\gamma}^2)^2+(2\gamma\cdot\xi-\tau_j)^2\bigr)^{1/2}} \\
&\lesssim
\frac{|\tau_2-\tau_1|}
{|\xi^2-r_{k,\gamma}^2|\,
\bigl((\xi^2-r_{k,\gamma}^2)^2+(\tau_1-\tau_2)^2\bigr)^{1/2}},
\end{align*}
where $j\in\{1,2\}$ and the implicit constant is independent of $\xi$ and $\tau_j$. Applying Young's inequality for products, we further obtain
\begin{equation}\label{eq:F:1}
\left|
\frac{1}{m_{k,\gamma,\tau_1}(\xi)}-\frac{1}{m_{k,\gamma,\tau_2}(\xi)}
\right| \lesssim \frac{|\tau_2-\tau_1|^{\epsilon}}
{|\xi^2-r_{k,\gamma}^2|^{1+\epsilon}},
\quad \forall\, \epsilon\in(0,1).
\end{equation}

Proceeding in exactly the same manner as in the proofs of Lemmas~\ref{F1} and~\ref{F3}, we obtain
\begin{equation}\label{eq:F:4}
\big|I^{(i,s)}_{k,\gamma,\epsilon_0,\tau_1}(f,g)-I^{(i,s)}_{k,\gamma,\epsilon_0,\tau_2}(f,g)
\big|\lesssim|\tau_2-\tau_1|^{\epsilon}
\frac{(1+r_{k,\gamma})^{s}}{\epsilon_0 r_{k,\gamma}}
\|\hat f\|_{H^{1/2+2\epsilon}}
\|\hat g\|_{H^{1/2+2\epsilon}},
\end{equation}
for $i=1,3$. Here, a slightly higher regularity assumption on $\hat f$ and $\hat g$ is required on the right-hand side.
It therefore remains to estimate the difference $\big|I^{(2,s)}_{k,\gamma,\epsilon_0,\tau_1}(f,g)-
I^{(2,s)}_{k,\gamma,\epsilon_0,\tau_2}(f,g)\big|.$

Let $A_{\gamma,\tau}^+$ and $\partial^{\rm in}A_{\gamma,\tau}^+$ be the sets introduced in the proof of Lemma~\ref{F2}. As in that proof, we adopt polar coordinates in \eqref{F2:m} and write
\begin{align*}
&\big|I^{(2,s)}_{k,\gamma,\epsilon_0,\tau_1}(f,g)-I^{(2,s)}_{k,\gamma,\epsilon_0,\tau_2}(f,g)\big| \\
&=\bigg|\int_{\mathbb{S}^{d-1}}\mathbbm{1}_{\partial^{\rm in}A_{\gamma,\tau}^+}(\omega)
\int_{\Gamma_0}\left(\frac{1}{m_{k,\gamma,\tau_1}(\omega r)}-\frac{1}{m_{k,\gamma,\tau_2}(\omega r)}
\right)\\
&\qquad\times(1+r^2)^{s/2}\psi(r; r_{k,\gamma},\epsilon_0)r^{d-1}dr,\hat f(\omega r_{k,\gamma})\bar{\hat g}(\omega r_{k,\gamma})d\omega
\bigg| \\
&\quad +\bigg|\int_{\mathbb{S}^{d-1}}\mathbbm{1}_{\partial^{\rm in}A_{\gamma,\tau}^{+,c}}(\omega)
\int_{\Gamma_0}\left(\frac{1}{m_{k,\gamma,\tau_1}(\omega r)}-\frac{1}{m_{k,\gamma,\tau_2}(\omega r)}\right)\\
&\qquad\times(1+r^2)^{s/2}\psi(r; r_{k,\gamma},\epsilon_0)r^{d-1}dr{\hat f}(\omega r_{k,\gamma})
{\bar{\hat g}}(\omega r_{k,\gamma})d\omega\bigg|,
\end{align*}
where $\partial^{\rm in}A_{\gamma,\tau}^{+,c}$ denotes the complement of $\partial^{\rm in}A_{\gamma,\tau}^+$ in $\mathbb{S}^{d-1}$. Since $\partial^{\rm in}A_{\gamma,\tau}^+$ has surface measure zero, the first term vanishes. Therefore, it suffices to estimate the second term.

Let $\Gamma_0$, $\Gamma_{\epsilon_0}$, and $\Gamma$ be the contours defined in the proof of Lemma~\ref{F2}, and define  $B_{\gamma,\tau}^+:=\{\omega:\ -\epsilon_0/4 < \gamma_{\omega}<\Theta_{\gamma}(\tau) \}$. Then we have
\begin{align*}
&|I^{(2,s)}_{k,\gamma,\epsilon_0,\tau_1}(f,g) - I^{(2,s)}_{k,\gamma,\epsilon_0,\tau_2}(f,g)|\\
& \leq  \int_{\mathbb{S}^{d-1}} \mathbbm{1}_{B_{\gamma,\tau}^+}(\omega) \left| \int_{\Gamma_0} \frac{1}{m_{k,\gamma,\tau_1}} - \frac{1}{m_{k,\gamma,\tau_2}} (1+r^2)^{s/2}  \psi(r; r_{k,\gamma},\epsilon_0)r^{d-1} dr \right| \ |\hat{f}(\omega r_{k,\gamma})\bar{\hat{g}}(\omega r_{k,\gamma})| d\omega\\
&\quad  + \left| \int_{\mathbb{S}^{d-1}} \mathbbm{1}_{A_{\gamma,\tau}^+}(\omega) \int_{\Gamma_0} \left(\frac{1}{m_{k,\gamma,\tau_1}} - \frac{1}{m_{k,\gamma,\tau_2}} \right)(1+r^2)^{s/2}  \psi(r; r_{k,\gamma},\epsilon_0)r^{d-1} dr \hat{f}(\omega r_{k,\gamma})\bar{\hat{g}}(\omega r_{k,\gamma}) d\omega\right| \\
& \quad + \int_{\mathbb{S}^{d-1}} \mathbbm{1}_{\{|\g_{\omega|\geq \epsilon_0/4}\}}(\omega) \int_{\Gamma_0} \left|\frac{1}{m_{k,\gamma,\tau_1}} - \frac{1}{m_{k,\gamma,\tau_2}}\right| (1+r^2)^{s/2}  \psi(r; r_{k,\gamma},\epsilon_0)r^{d-1} dr |\hat{f}(\omega r_{k,\gamma})\bar{\hat{g}}(\omega r_{k,\gamma})| d\omega\\
&:= J_1 +J_2+J_3.
\end{align*}
 
We first estimate the term $J_1$. As shown in the proof of Lemma~\ref{F2}, there are no poles inside the contour $\Gamma$. Applying Cauchy's theorem, we obtain 
\[
J_1 \lesssim \int_{\mathbb{S}^{d-1}} \mathbbm{1}_{B_{\gamma,\tau}^+}(\omega)\int_{\Gamma_{\epsilon_0}} \left|\frac{1}{m_{k,\gamma,\tau_1}} - \frac{1}{m_{k,\gamma,\tau_2}}\right| (1+r^2)^{s/2}  \psi(r; r_{k,\gamma},\epsilon_0)r^{d-1} dr |\hat{f}(\omega r_{k,\gamma})\bar{\hat{g}}(\omega r_{k,\gamma})| d\omega. 
\]
Using \eqref{F2:2}, \eqref{eq:F:1}, and Lemma \ref{R2:n}, we deduce
\begin{equation}\label{eq:F:2}
J_1 \lesssim  |\tau_2-\tau_1|^{\epsilon} \frac{(1+r_{k,\gamma})^s}{\epsilon_0^{1+\epsilon} r_{k,\gamma}^{1+\epsilon}}\|\hat{f}\|_{H^{1/2+\epsilon}} \|\hat{g}\|_{H^{1/2+\epsilon}},\quad \forall\, \epsilon\in (0,1).
\end{equation}

We next estimate the term $J_3$. By \eqref{F2:3} and Lemma~\ref{R2:n}, we obtain
\begin{equation}\label{eq:F:3}
\begin{aligned}
J_3 &\lesssim \int_{\mathbb{S}^{d-1}} \int_{\Gamma_0} \frac{|\tau_1- \tau_2| \mathbbm{1}_{\{|\gamma_{\omega|\geq \epsilon_0/4}\}}(\omega) }{|m_{k,\gamma,\tau_1} m_{k,\gamma,\tau_2}|}  (1+r^2)^{s/2}  \psi(r; r_{k,\gamma},\epsilon_0)r^{d-1} dr |\hat{f}(\omega r_{k,\gamma})\bar{\hat{g}}(\omega r_{k,\gamma})| d\omega\\
&\lesssim |\tau_2-\tau_1| \frac{(1+r_{k,\gamma})^s}{\epsilon_0^{2} r_{k,\gamma}^{2}}\|\hat{f}\|_{H^{1/2+\epsilon}} \|\hat{g}\|_{H^{1/2+\epsilon}},\quad \forall\, \epsilon\in (0,1). 
\end{aligned}
\end{equation}

We now estimate the term $J_2$. In this case, there exists a pole inside the contour $\Gamma = \Gamma_{\epsilon_0} \cup \Gamma_0$. We decompose $J_2$ as follows:
\begin{align*}
J_{2} & \lesssim  \left| \int_{\mathbb{S}^{d-1}} \mathbbm{1}_{A_{\gamma,\tau}^+}(\omega) \int_{\Gamma_{\epsilon_0}} \left(\frac{1}{m_{k,\gamma,\tau_1}} - \frac{1}{m_{k,\gamma,\tau_2}} \right)(1+r^2)^{s/2}  \psi(r; r_{k,\gamma},\epsilon_0)r^{d-1} dr \hat{f}(\omega r_{k,\gamma})\bar{\hat{g}}(\omega r_{k,\gamma}) d\omega\right| \\
&\quad + \left| \int_{\mathbb{S}^{d-1}} \mathbbm{1}_{A_{\gamma,\tau}^+}(\omega) \int_{\Gamma} \left(\frac{1}{m_{k,\gamma,\tau_1}} - \frac{1}{m_{k,\gamma,\tau_2}} \right)(1+r^2)^{s/2}  \psi(r; r_{k,\gamma},\epsilon_0)r^{d-1} dr \hat{f}(\omega r_{k,\gamma})\bar{\hat{g}}(\omega r_{k,\gamma}) d\omega\right| \\
&:= J_{2,1} + J_{2,2}.
\end{align*}
Similar to the estimate of $J_1$, it follows that 
\begin{equation}\label{eq:F:5.1}
J_{2,1} \lesssim  |\tau_2-\tau_1|^{\epsilon} \frac{(1+r_{k,\gamma})^s}{\epsilon_0^{1+\epsilon} r_{k,\gamma}^{1+\epsilon}}\|\hat{f}\|_{H^{1/2+\epsilon}} \|\hat{g}\|_{H^{1/2+\epsilon}},\quad \forall\, \epsilon\in (0,1).
\end{equation}
To estimate $J_{2,2}$, we have
\begin{equation*}
J_{2,2} \leq  \left| \int_{\mathbb{S}^{d-1}} \big(F(\omega;\tau_1)-F(\omega;\tau_2)\big) \hat{f}(\omega r_{k,\gamma})\bar{\hat{g}}(\omega r_{k,\gamma}) d\omega\right|,  
\end{equation*}
where
\[
F(\omega;\tau):= \mathbbm{1}_{A_{\gamma,\tau}^+}(\omega)\int_{\Gamma} \frac{\psi(r; r_{k,\gamma},\epsilon_0)}{m_{k,\gamma,\tau}(r,\omega)}(1+r^2)^{s/2}r^{d-1}dr.
\]
By Cauchy's integral theorem and recalling \eqref{F2:pole}, we obtain
\[
F(\omega;\tau)= -\frac{\mathbbm{1}_{A_{\gamma,\tau}^+}(\omega) \pi {\rm i}}{ (r_{k,\gamma}^2- \gamma_{\omega}^2 + {\rm i} \tau)^{1/2}} (1+r^2)^{s/2}r^{d-1},\quad r:= -{\rm i} \gamma_{\omega} + (r_{k,\gamma}^2 -\gamma_{\omega}^2 + {\rm i} \tau)^{1/2}.
\]

For any fixed $\omega \in \mathbb{S}^{d-1}$ and $\gamma \in \mathbb{R}^d$, we recall from the proof of Lemma~\ref{F2} that the possible poles are given by
\begin{align*}
r=R_{\gamma,\tau}(\gamma_{\omega}),\quad R_{\gamma,\tau}(x):= -{\rm i} x + \sqrt{r_{k,\g}^2- x^2 + {\rm i}\tau},\quad x \in(\Theta_{\gamma}(\tau),\epsilon_0/4).  
\end{align*}
It is straightforward to verify that $R_{\gamma,\tau}$ admits an explicit inverse, i.e., 
\[
x=R^{-1}_{\gamma,\tau}(r):= \frac{r_{k,\gamma}^2 + {\rm i} \tau - r^2}{2{\rm i} r },
\]
where $r$ lies inside the contour $\Gamma$. Hence, the mapping $x \mapsto R_{\gamma,\tau}(x)$ is a diffeomorphism onto its image.
As a result, $\Theta_{\gamma}(\tau)$ depends continuously on $\tau$, and $F(\omega;\tau)$ converges pointwise to $F(\omega;\tau_0)$ as $\tau \to \tau_0$ in $(0,\epsilon_0^2)$.

Let $\epsilon \in (0,1)$ be fixed. By H\"{o}lder's inequality, we obtain
\begin{align*}
 J_{2,2}&\lesssim  \left(\int_{\mathbb{S}^{d-1}} |F(\omega;\tau_1)-F(\omega;\tau_2)|^{(1+\epsilon)/\epsilon} 
 d\omega \right)^{\epsilon/(1+\epsilon)} \left(\int_{\mathbb{S}^{d-1}} |\hat{f}(\omega r_{k,\gamma})\bar{\hat{g}}(\omega r_{k,\gamma})|^{(1+\epsilon)} d\omega\right)^{1/(1+\epsilon)}\\
 &= \|F(\omega;\tau_1)-F(\omega;\tau_2) \|_{L^{(1+\epsilon)/\epsilon}(\mathbb{S}^{d-1})} \| (\hat{f})_{r_{k,\g}} \|_{L^{2(1+\epsilon)}(\mathbb{S}^{d-1})} \| (\bar{\hat{g}})_{r_{k,\g}} \|_{L^{2(1+\epsilon)}(\mathbb{S}^{d-1})}.   
 \end{align*}
By Sobolev embedding on the sphere and Lemmas~\ref{R2:n}, we have
\[
\| (\hat{f})_{r_{k,\g}} \|_{L^{2(1+\theta)}(\mathbb{S}^{d-1})} \lesssim \| (\hat{f})_{r_{k,\g}} \|_{H^{\frac{\epsilon d}{2(1+\epsilon)}}(\mathbb{S}^{d-1})} \lesssim_{k,\gamma} \| \hat{f} \|_{H^{\frac{\epsilon d}{2(1+\epsilon)}+ \frac{1}{2}}}, 
\]
and the same estimate holds for $\hat g$. Hence, we arrive at
\begin{equation}\label{eq:F:5.2}
J_{2,2} \lesssim_{k,\gamma,\e} \|F(\omega;\tau_1)-F(\omega;\tau_2) \|_{L^{\frac{(1+\epsilon)}{\epsilon}}}  \| \hat{f} \|_{H^{\frac{\epsilon d}{2(1+\epsilon)}+ \frac{1}{2}}} \| \hat{g} \|_{H^{\frac{\epsilon d}{2(1+\epsilon)}+ \frac{1}{2}}}. 
\end{equation}

By the dominated convergence theorem, for any $\epsilon>0$ there exists $\delta>0$ such that
\[
\|F(\omega;\tau_1)-F(\omega;\tau_2) \|_{L^{(1+\epsilon)/\epsilon}}< \epsilon,\quad \forall\, \tau_1,\tau_2\in (0,\delta).
\]
Combining \eqref{eq:F:4}--\eqref{eq:F:5.2}, we conclude that for any $\varepsilon>0$ there exists
$\delta>0$ such that, for all $\tau_1,\tau_2 \in (0,\delta)$,
\begin{align*}
|\langle\H_{k,\gamma,\tau_1} f - \H_{k,\g,\tau_2}f, g \rangle| &\leq \sum_{j=1}^3 |I^{(j,s)}_{k,\gamma,\epsilon_0,\tau_1}(f,g) - I^{(j,s)}_{k,\gamma,\epsilon_0,\tau_2}(f,g)|\\
&\leq\ \varepsilon  \|{f}\|_{L^{2}(\jpx^{\eta})}\|{g}\|_{L^{2}(\jpx^{\eta})},\quad \forall\, \eta >1/2.
\end{align*}
This establishes the strong convergence of $\mathcal{H}_{k,\gamma,\tau}$ as $\tau\to0^\pm$ and yields
\begin{equation*}
\|\mathcal{H}^{\pm}_{k,\gamma} f\|_{H^{r+s}(\langle x\rangle^{-\eta})}\lesssim
\frac{(1+r_{k,\gamma})^{s}}{\min\{r_{k,\gamma}, 1\} r_{k,\gamma}}\|f\|_{H^{r}(\langle x\rangle^{\eta})}.
\end{equation*}
Applying Lemma~\ref{Pre00}, we obtain
\[
\begin{aligned}
\|\mathcal{H}^{\pm}_{k,\gamma} f\|_{H^{r+s}(\langle x\rangle^{-\eta},\la)} &\lesssim\ \la^{-\eta}\|\mathcal{H}^{\pm}_{k,\gamma} f\|_{H^{r+s}(\langle x\rangle^{-\eta})}\\
&\lesssim \frac{\la^{-\eta}(1+r_{k,\gamma})^{s}}{\min\{r_{k,\gamma}, 1\} r_{k,\gamma}}\|f\|_{H^{r}(\langle x\rangle^{\eta})}
\lesssim \frac{\la^{-2\eta}(1+r_{k,\gamma})^{s}}{\min\{r_{k,\gamma}, 1\} r_{k,\gamma}}\|f\|_{H^{r}(\langle x\rangle^{\eta},\la)},
\end{aligned}
\]
which implies \eqref{eq:F}. 
\end{proof}

We consider only the one-sided limits of $\mathcal{H}_{k,\gamma,\tau}$ as $\tau\to0^\pm$, since changing the sign of $\tau$ may cause additional poles to enter the contour $\Gamma$. Consequently, the limiting operator may depend on the direction from which $\tau$ approaches zero. To illustrate this phenomenon, consider the special case $\gamma=0$. It is straightforward to verify that no poles lie inside $\Gamma$ when $\tau>0$. In contrast, when $\tau<0$, a pole may enter the contour, giving rise to an additional residue contribution. This contribution prevents convergence to the same limiting operator obtained as $\tau\to0^+$.

Theorem~\ref{thm:F} provides $L^2$-based estimates for the operators $\mathcal{H}_{k,\gamma}^{\pm}$. We now extend these estimates to an $L^p$-based framework using rescaled weighted Besov spaces. We begin with the following proposition.

\begin{Proposition}\label{P:H:Lp}
Assume that $f\in L^1(\mathbb R^d)$ satisfies $\operatorname{supp}\hat f\subseteq B(0,R_0)$ for some $R_0>0$. Let $\epsilon \in (0,1)$, $p_0 \in [1,\infty]$, and assume that $\eta > \max\{ (d +1)/2-d/p_0',1/2\}$. Then the operators $\mathcal{H}^{\pm}_{k,\gamma}$ satisfy the estimate
\[
\|\mathcal{H}^{}_{k,\gamma,\tau} f\|_{L^{p_0}(\langle x\rangle^{-\eta})}\lesssim
\max\bigl\{ r_{k,\gamma}^{-2},\, r_{k,\gamma}^{-1} \bigr\}\|f\|_{L^{2}(\langle x\rangle^{1/2+\epsilon})},
\]
where the implicit constant is independent of $\tau$ but might depend on $R_0$.
\end{Proposition}

\begin{proof}
We first consider the case $p_0\in[1,2]$. Since $d/p_0 - (d-1)/2>1/2$, we choose $\epsilon \in (0, \eta - d(1/p_0-1/2)-1/2)$. By H\"{o}lder’s inequality, we have
\[
\|\mathcal{H}^{}_{k,\gamma,\tau}f\|_{L^{p_0}(\langle x\rangle^{-\eta})}
\lesssim \|\mathcal{H}^{}_{k,\gamma,\tau}f\ \langle x\rangle^{-1/2-\epsilon}\|_{L^2} 
\|\langle x\rangle^{-\eta+1/2+\epsilon}\|_{L^{2p_0/(2-p_0)}}.
\]
Applying Theorem~\ref{thm:F} with $r=s=0$ yields
\[
\|\mathcal{H}^{}_{k,\gamma,\tau}f\|_{L^2(\langle x\rangle^{-1/2-\varepsilon})}
\lesssim \max\{r_{k,\gamma}^{-2}, r_{k,\gamma}^{-1}\}
\|f\|_{L^{2}(\langle x\rangle^{1/2+\varepsilon})}.
\]
Moreover, the weighted norm $\|\langle x\rangle^{-\eta+1/2+\varepsilon}\|_{L^{2p_0/(2-p_0)}}$ is finite provided that ${2p_0}(\eta - 1/2-\epsilon )/{(2-p_0)} >d $,  which is ensured by the condition $\eta > d/p_0-(d-1)/2$. This completes the proof for $p_0 \in [1,2]$.

We next consider the case $p_0\in(2,\infty]$, for which $\eta>1/2$. Choose $\psi\in\mathcal{S}(\mathbb{R}^d)$ such that $\hat{\psi}$ is compactly supported and $\hat{\psi}=1$ on a neighborhood of $\operatorname{supp}\hat{f}$. Since $\mathcal{H}_{k,\gamma,\tau}$ is a Fourier multiplier operator and $\hat{f}$ is compactly supported, we have $\operatorname{supp}\mathcal{F}(\mathcal{H}_{k,\gamma,\tau}f)
\subseteq \operatorname{supp}\hat{f}$. Therefore, $\mathcal{H}_{k,\gamma,\tau}f=\psi*\mathcal{H}_{k,\gamma,\tau}f.$
By \eqref{eq:ad:1}, we obtain, for $p_0\in(2,\infty)$,
\[
\begin{aligned}
\|\mathcal{H}^{}_{k,\gamma,\tau}f\|_{L^{p_0}(\langle x\rangle^{-\eta})} \lesssim&  \left(\int_{\mathbb{R}^d} \big| \int_{\mathbb{R}^d} \psi(x-y) \langle x-y \rangle^{\eta}\, \mathcal{H}^{}_{k,\gamma,\tau}f(y)\jpy^{-\eta}\, dy\big|^{p_0} dx\right)^{1/p_0}.
\end{aligned}
\]

For $p_0\in(2,\infty]$, choose $p_1\in(1,2]$ such that $1+1/p_0= 1/p_1+1/2$. Applying Young's convolution inequality to the above estimate yields
\[
\|\mathcal{H}^{}_{k,\gamma,\tau}f\|_{L^{p_0}(\langle x\rangle^{-\eta})} 
\lesssim \|\psi(x) \jpx^{\eta}\|_{L^{p_1}} \|\mathcal{H}^{}_{k,\gamma,\tau}f(x) \jpx^{-\eta}\|_{L^2}.
\]
Moreover, Theorem~\ref{thm:F} gives
\[
\|\mathcal{H}^{}_{k,\gamma,\tau}f\|_{L^{2}(\langle x\rangle^{-\eta})} 
\lesssim \max\{r_{k,\gamma}^{-2}, r_{k,\gamma}^{-1}\}
\|f\|_{L^{2}(\langle x\rangle^{1/2+\varepsilon})}.
\]
Combining these estimates completes the proof.
\end{proof}

\begin{Theorem}\label{Hsg}
Let $\lambda\in(0,1)$ and $\alpha \in [0,2]$. Let $p_1,p_2,q \in [1,\infty]$ and $r \in \mathbb{R}$ satisfy $
1/p_2-1/p_1=\alpha/d.$ Then the operator $\mathcal{H}_{k,\gamma}^{\pm}$ extends to a bounded operator from $
B^{r-2+\alpha}_{p_2,q}(\langle x\rangle^{\eta},\lambda)$ to $B^{r}_{p_1,q}(\langle x\rangle^{-\eta},\lambda)$. Moreover, for any $\eta >\max\{(d+1)/2-\alpha,1/2\}$, the following estimate holds:
\[
\|\mathcal{H}_{k,\gamma}^{\pm} f\|_{B^{r}_{p_1,q}(\langle x\rangle^{-\eta},\lambda)}
\lesssim_{\eta,p_1,p_2,q,r} \lambda^{-\eta} \max\{ r_{k,\gamma}^{\eta}, r_{k,\gamma}^{-2}\}
\|f\|_{B^{r-2+\alpha}_{p_2,q}(\langle x\rangle^{\eta},\lambda)},
\]
where the implicit constant is independent of $\lambda$, $k$, and $\gamma$.
\end{Theorem}

\begin{proof}
Let $\psi^{(-1)}$ be a smooth function such that $\operatorname{supp}\psi^{(-1)} \subseteq \{\xi\in\mathbb{R}^d: |\xi|\leq 3/2\}$ and $\psi^{(-1)} \equiv 1$ on  $\operatorname{supp}\chi$. Let $\tilde{\varphi}$ be a smooth function satisfying $
\operatorname{supp}\tilde{\varphi} \subseteq \{\xi\in\mathbb{R}^d: 1/2 \leq |\xi| \leq 3\}$ and $\tilde{\varphi} \equiv 1$ on $ \operatorname{supp}\varphi$. Define $\psi^{(n)} := \tilde{\varphi}$ for all $n\geq 0$, and introduce the dyadically rescaled cut-off functions
\[
\psi_{-1}^{(-1)}(\xi):=\psi^{(-1)}(\xi), \quad  \psi_{n}^{(n)}(\xi):=\psi^{(n)}(2^{-n}\xi), \quad n\geq 0.
\]
For $n\geq 0$, set $c_{n,k,\gamma}:=2^{-n} r_{k,\gamma},  \gamma_n:=2^{-(n-1)}\gamma$, and define $c_{-1,k,\gamma}:= r_{k,\gamma}, \gamma_{-1}:=\gamma$. Finally, for $n\geq -1$, define the multipliers $\Psi_{n,k,\gamma}$ by 
\[
\Psi_{n,k,\gamma}(\xi) := \frac{\psi^{(n)}(\xi)}{|\xi|^{2}-c_{n,k,\gamma}^{2}+{\rm i}\gamma_n\cdot\xi}.
\]

Let $N_0>4$ be a fixed integer. Then there exists at most one integer $n_0 \geq 0$ such that $c_{n_0-1,k,\gamma} \geq 16/3$ and $c_{n_0,k,\gamma} <16/3$. Moreover, for this $n_0$, we have
\[
c_{n_0+N_0,k,\gamma}^{} = 2^{-N_0} c_{n_0,k,\gamma}^{} \leq 2^{-4} c_{n_0,k,\gamma}^{}< \frac{1}{3} < \frac{3}{8}.
\]
Hence, $c_{n,k,\gamma}$ lies outside the interval $[3/8, 16/3]$ for all $n> n_0+N_0$ and $0 \leq n < n_0$. Note that the index $n_0$ depends on $(k,\gamma)$, whereas $N_0$ is independent of both $k$ and $\gamma$.

For any $f \in \mathcal{S}$, we set $g_{k,\gamma}:=\mathcal{H}^{\pm}_{k,\gamma} f$. Then
\begin{equation}\label{H:eq:00:1}
 \begin{aligned}
&\|\mathcal{H}^{\pm}_{k,\gamma}f\|_{B^{r}_{p_1,q}(\jpx^{-\eta},\lambda)}^q = \sum_{n \in  [0,n_0)} +\sum_{n\in (n_0+N_0,\infty)} 2^{nrq} \|\mathcal{F}^{-1}(\varphi_{n} \hat{g}_{k,\gamma})\|_{L^{p_1}(\jpl^{-\eta})}^q \\
&\quad + \sum_{n\in [n_0,n_0+N_0]} 2^{nrq} \|\mathcal{F}^{-1}(\varphi_{n} \hat{g}_{k,\gamma})\|_{L^{p_1}(\jpl^{-\eta})}^q + 2^{-rq} \|\mathcal{F}^{-1}(\varphi_{-1} \hat{g}_{k,\gamma})\|_{L^{p_1}(\jpl^{-\eta})}^q, 
\end{aligned}
\end{equation}
where we set $n_0=\infty$ if such an index does not exist. Moreover, for any $n\geq -1$, we have the convolution representation
 \begin{align*}
 \mathcal{F}^{-1}(\varphi_{n}\hat{g}_{k,\gamma}) =
 \mathcal{F}^{-1}\left(\frac{\psi^{(n)}_{n}}{|\xi|^{2}-k^2-|\gamma|^{2} + 2{\rm i} \gamma\cdot \xi}\right) \ast \mathcal{F}^{-1}(\varphi_{n} \hat{f}).
 \end{align*}
 
For $n \geq 0$, we have
 \begin{align*}
& \|\mathcal{F}^{-1}(\varphi_{n}\hat{g}_{k,\gamma})\|_{L^{p_1}(\jpl^{-\eta})}^{p_1} = \|\mathcal{F}^{-1}\left(\frac{\psi^{(n)}_{n}}{|\xi|^{2}-k^2-|\gamma|^{2} + 2{\rm i}\gamma\cdot \xi}\right) \ast \Delta_{n}f\|_{L^{p_1}(\jpl^{-\eta})}^{p_1}\\
 &= \int_{\mathbb{R}^d} \left| \int_{\mathbb{R}^d} \mathcal{F}^{-1}\left(\frac{\psi^{(n)}_{n}}{|\xi|^{2}-k^2-|\gamma|^{2} + 2{\rm i} \gamma\cdot \xi}\right)(y) \Delta_{n}f(x-y) dy \right|^{p_1} \jpl^{-\eta p_1} dx.
 \end{align*}
 Let $1- 1/p_0=  1/p_2-1/p_1=\alpha/d$.  By Young's convolution inequality and \eqref{eq:ad:1}, we obtain
 \begin{align*}
  &\|\mathcal{F}^{-1}(\varphi_{n}\hat{g}_{k,\gamma})\|_{L^{p_1}(\jpl^{-\eta})}\\
  & \lesssim \left( \int_{\mathbb{R}^d} \left| \int_{\mathbb{R}^d} \mathcal{F}^{-1}\left(\frac{\psi^{(n)}_{n}}{|\xi|^{2}-k^2-|\gamma|^{2} + 2{\rm i} \gamma\cdot \xi}\right)(y) \jpl^{-\eta} \Delta_{n} f(x-y)\langle \lambda(x-y) \rangle^{\eta} dy \right|^{p_1} dx\right)^{1/p_1}\\
&\lesssim 2^{-2n}  \left\| \mathcal{F}^{-1}\left(\frac{\psi^{(n)}(2^{-n} \xi)}{(2^{-n}|\xi|)^{2}-(2^{-n} r_{k,\g})^2 + 2{\rm i} 2^{-n}\g\cdot 2^{-n}\xi}\right) \right\|_{L^{p_0}(\jpl^{-\eta})} \|\Delta_{n} f\|_{L^{p_2}(\jpl^{\eta})}.
   \end{align*}
Combining the above estimates yields
\begin{equation}\label{H:eq:00}
  \|\mathcal{F}^{-1}(\varphi_{n}\hat{g}_{k,\gamma})\|_{L^{p_1}(\jpl^{-\eta})} \lesssim  2^{-2n} \|  2^{nd} \check{\Psi}_{n,k,\gamma}(2^{n} x) \|_{L^{p_0}(\jpl^{-\eta})} \|\Delta_{n} f\|_{L^{p_2}(\jpl^{\eta})}.
\end{equation}

We apply H\"{o}lder's inequality to obtain  
\[
\| 2^{nd} \check{\Psi}_{n,k,\gamma}(2^{n} x)\|_{L^{p_0}(\jpl^{-\eta})} \lesssim \|  2^{nd}\check{\Psi}_{n,k,\gamma}(2^n x) \|_{L^{p_0}} \| \jpl^{-\eta} \|_{L^{\infty}} \lesssim  2^{n\alpha}  \|\check{\Psi}_{n,k,\gamma}\|_{L^{p_0}}, 
\]
where we used the scaling relation
\[
\|2^{nd} f(2^n\cdot)\|_{L^{p_0}} = 2^{n(d-d/p_0)}\|f\|_{L^{p_0}} = 2^{n\alpha}\|f\|_{L^{p_0}},
\quad \alpha = d(1-1/p_0).
\]
When $p_0\in[1,2]$, another application of H\"{o}lder's inequality yields
 \[
  \| \check{\Psi}_{n,k,\gamma}\|_{L^{p_0}} \leq \left(\int_{\mathbb{R}^d} |\check{\Psi}_{n,k,\gamma}(x)|^2\jpx^{(d+1)(2-p_0)/p_0} dx \right)^{p_0/2}\left(\int_{\mathbb{R}^d}\jpx^{-d-1} dx\right)^{(2-p_0)/2}.
 \]
Let $d_0$ be the smallest integer such that $2d_0>(d+1)(2-p_0)/p_0$. Then $\jpx^{(d+1)(2-p_0)/p_0} \lesssim \jpx^{2d_0}$, and by Plancherel’s identity,
 \begin{align*}
  \| \check{\Psi}_{n,k,\gamma}\|_{L^{p_0}} \lesssim \left(\int_{\mathbb{R}^d} |\check{\Psi}_{n,k,\gamma}(x)|^2\jpx^{2d_0} dx \right)^{p_0/2} \lesssim  \left(\sum_{|\alpha_0|\leq d_0} \int_{\mathbb{R}^d} |D^{\alpha_0}{\Psi}_{n,k,\gamma}(x)|^2 dx\right)^{p_0/2}.
 \end{align*}
 When $p_0 \in [2,\infty)$, we apply Hausdorff--Young's inequality and obtain the same estimate with $d_0=0$.
 
 Therefore, we obtain
\begin{equation}\label{H:eq:0}
\left\|  2^{nd} \check{\Psi}_{n,k,\gamma}(2^n x)  \right\|_{L^{p_0}(\jpl^{-\eta})} \lesssim 2^{n\alpha}   \left(\sum_{|\alpha_0|\leq d_0} \int_{\mathbb{R}^d} |D^{\alpha_0}{\Psi}_{n,k,\gamma}(x)|^2 dx\right)^{p_0/2}.
\end{equation}

For the cases $n\in[0,n_0)$ or $n\in(n_0+N_0,\infty)$, we have either $c_{n,k,\gamma}\in (0,3/8)$ or $c_{n,k,\gamma}\in (16/3,\infty)$. A direct computation shows that, on $\text{supp}\tilde{\varphi}$, the denominator of
$D^{\alpha_0}\Psi_{n,k,\gamma}$ is uniformly bounded below by $1/8$. Hence, all derivatives $D^{\alpha_0}\Psi_{n,k,\gamma}$ are uniformly bounded with respect to $n$, $k$, and $\gamma$ in these regimes. Combining \eqref{H:eq:00} and \eqref{H:eq:0}, we get
\[
2^{nrq} \|\mathcal{F}^{-1}(\varphi_{n} \hat{g}_{k,\gamma})\|_{L^{p_1}(\jpl^{-\eta})}^q \lesssim   2^{n(r-2+ \alpha)q}  \|\Delta_{n} f\|_{L^{p_2}(\jpl^{\eta})}^q.
\]
Summing over $n\in[0,n_0)$ and $n\in(n_0+N_0,\infty)$ yields
\begin{align}\label{H:eq:1}
&\sum_{ n\in [0,n_0)} + \sum_{ n\in (n_0+N_0,\infty)} 2^{nrq} \|\mathcal{F}^{-1}(\varphi_{n} \hat{g}_{k,\gamma})\|_{L^{p_1}(\jpl^{-\eta})}^q\notag\\
&\quad \lesssim \sum_{ n\in [0,n_0)}+  \sum_{n\in (n_0+N_0,\infty)} 2^{n(r-2+\alpha)q} \|\Delta_{n} f\|_{L^{p_2}(\jpl^{\eta})}^q.
 \end{align}
 
For the case $n\in[n_0,n_0+N_0]$, a change of variables in \eqref{H:eq:00} yields
\begin{equation*}
2^{nr} \|\mathcal{F}^{-1}(\varphi_{n}\hat{g}_{k,\gamma})\|_{L^{p_1}(\jpl^{-\eta})} 
\lesssim  2^{n(r-2+\alpha)}  \left\|   \check{\Psi}_{n,k,\gamma}(x) \langle \lambda 2^{-n} x \rangle^{-\eta} \right\|_{L^{p_0}} \|\Delta_{n} f\|_{L^{p_2}(\jpl^{\eta})} .
\end{equation*}
For $n\in[n_0,n_0+N_0]$, we have $c_{n,k,\gamma}^2\in[3/8,16/3]$, which implies
$2^n\sim r_{k,\gamma}$. Hence, $\langle \lambda 2^{-n}x\rangle \gtrsim \min\{\lambda r_{k,\gamma}^{-1},1\}\langle x\rangle.$ It follows that 
\[
2^{nr} \|\mathcal{F}^{-1}(\varphi_{n}\hat{g}_{k,\gamma})\|_{L^{p_1}(\jpl^{-\eta})} \lesssim  2^{n(r-2+\alpha)} 
\max\{\lambda^{-\eta}r_{k,\gamma}^{\eta}, 1\}\left\|   \check{\Psi}_{n,k,\gamma}(x) \langle  x \rangle^{-\eta} \right\|_{L^{p_0}}\|\Delta_{n} f\|_{L^{p_2}(\jpl^{\eta})}. 
\]
Therefore, 
\begin{align}\label{H:eq:2}
\sum_{n=n_0}^{n_0+N_0} 2^{nrq} \|\mathcal{F}^{-1}(\varphi_{n} \hat{g}_{k,\gamma})\|_{L^{p_1}(\jpl^{-\eta})}^q &\lesssim 
\max\{\lambda^{-\eta q}r_{k,\gamma}^{\eta q}, 1\} \sup_{n \in [n_0,n_0+N_0]}\|\mathcal{H}^{\pm}_{2^{-n}k,2^{-n}\gamma}\check\psi(x) 
\|_{L^{p_0}(\langle  x \rangle^{-\eta})}^q \notag\\ 
&\quad \times\sum_{n=n_0}^{ n_0+N_0} 2^{n(r-2+\alpha)q} \|\Delta_{n} f\|_{L^{p_2}(\jpl^{\eta})}^q. 
\end{align}

For the low-frequency block $n=-1$, an argument similar to that used to establish \eqref{H:eq:00} yields
\begin{equation*}
  \|\mathcal{F}^{-1}(\varphi_{-1}\hat{g}_{k,\gamma})\|_{L^{p_1}(\jpl^{-\eta})} \lesssim \left\|\check{\Psi}_{-1,k,\gamma}(x)  \right\|_{L^{p_0}(\jpl^{-\eta})} \|\Delta_{-1} f\|_{L^{p_2}(\jpl^{\eta})}. 
\end{equation*}
Since $\lambda\in(0,1)$, we have the elementary bound $\jpl^{-\eta}\lesssim \lambda^{-\eta}\jpx^{\eta}$. Hence,
\begin{align}\label{H:eq:3}
2^{-r} \|\mathcal{F}^{-1}(\varphi_{-1}\hat{g}_{k,\gamma})\|_{L^{p_1}(\jpl^{-\eta})} \lesssim \lambda^{-\eta } \left\|\mathcal{H}^{\pm}_{k,\gamma}\check\psi^{(-1)}(x)  \right\|_{L^{p_0}(\langle  x \rangle^{-\eta})}
2^{-(r-2+\alpha)}\|\Delta_{-1} f\|_{L^{p_2}(\jpl^{\eta})}.
\end{align}

Finally, combining \eqref{H:eq:1}--\eqref{H:eq:3} and factoring out the Besov norm of $f$, we obtain
\begin{align*}
&\|\mathcal{H}^{\pm}_{k,\gamma}f\|_{B^{r}_{p_1,q}(\jpx^{-\eta},\lambda)}^q = \sum_{n \in  [0,n_0)} +\sum_{n\in (n_0+N_0,\infty)} 2^{nrq} \|\mathcal{F}^{-1}(\varphi_{n} \hat{g}_{k,\gamma})\|_{L^{p_1}(\jpl^{-\eta})}^q \\
&\quad + \sum_{n\in [n_0,n_0+N_0]} 2^{nrq} \|\mathcal{F}^{-1}(\varphi_{n} \hat{g}_{k,\gamma})\|_{L^{p_1}(\jpl^{-\eta})}^q + 2^{-rq} \|\mathcal{F}^{-1}(\varphi_{-1} \hat{g}_{k,\gamma})\|_{L^{p_1}(\jpl^{-\eta})}^q\\
&\lesssim \sum_{ n\in [0,n_0)}+  \sum_{n\in (n_0+N_0,\infty)} 2^{n(r-2+\alpha)q} \|\Delta_{n} f\|_{L^{p_2}(\jpl^{\eta})}^q + \max\{\lambda^{-\eta q}r_{k,\gamma}^{\eta q}, 1\}\\
&\quad\times \sup_{n \in [n_0,n_0+N_0]}\|\mathcal{H}^{\pm}_{2^{-n}k,2^{-n}\gamma}\psi(x) 
\|_{L^{p_0}(\langle  x \rangle^{-\eta})}^q \sum_{n=n_0}^{ n_0+N_0} 2^{n(r-2+\alpha)q} 
\|\Delta_{n} f\|_{L^{p_2}(\jpl^{\eta})}^q \\
&\quad + \lambda^{-\eta q } \|\mathcal{H}^{\pm}_{k,\gamma}\psi^{(-1)}(x)  \|_{L^{p_0}(\langle  x \rangle^{-\eta})}^q
2^{-(r-2+\alpha)q}\|\Delta_{-1} f\|_{L^{p_2}(\jpl^{\eta})}^q\\
&\lesssim \lambda^{-\eta q} \Big( \max\{r_{k,\gamma}^{\eta q}, \lambda^{\eta q}\}\sup_{n \in [n_0,n_0+N_0]}\|\mathcal{H}^{\pm}_{2^{-n}k,2^{-n}\gamma}\check\psi^{(n)}(x)  \|_{L^{p_0}(\langle  x \rangle^{-\eta})}^q \\
&\quad +  \|\mathcal{H}^{\pm}_{k,\gamma}\check\psi^{(-1)}(x)  \|_{L^{p_0}(\langle  x \rangle^{-\eta})}^q) \|f\|_{B^{r-2+\a}_{p_2,q}(\jpx^{\eta},\lambda)}^q.
\end{align*}
By Proposition~\ref{P:H:Lp} and the fact that $2^{-n}k + 2^{-n}\gamma \sim 1$ for $n \in [n_0,n_0+N_0]$, we deduce 
\[
\|\mathcal{H}^{\pm}_{k,\gamma}f\|_{B^{r}_{p_1,q}(\jpx^{-\eta},\lambda)}^q  \lesssim \lambda^{-\eta q} \max\{r_{k,\gamma}^{\eta q}, \lambda^{\eta q}, r_{k,\gamma}^{-q},r_{k,\gamma}^{-2q}\} \|f\|_{B^{r-2+\a}_{p_2,q}(\jpx^{\eta},\lambda)}^q,
\]
which completes the proof.
\end{proof}

When $\g \neq 0$, a sharper estimate for $\mathcal{H}^{\pm}_{k,\gamma}f$ can be obtained by taking into account its dependence on $\g$.

\begin{Proposition}\label{P:H:Lp:c}
Assume that $f\in L^1(\mathbb R^d)$ satisfies $\operatorname{supp}\hat f\subseteq B(0,R_0)$ for some $R_0>0$. Let $\g\ne0$, $|\tau|< \min\{|\g|^2/8,1\}$ and $ p_0\in [1,\infty]$. For any $\eta >\max\{d/p_0-d/2,0\}$, the following estimate holds:
\[
\|\mathcal H^\pm_{k,\g,\tau}f\|_{L^{p_0}(\jpx^{-\eta})} \lesssim_{d,p_0,\eta} |\g|^{-1}r_{k,\g}^{-1}\, R_0^{d/p_0'}\max\{1,R_0\}^{\eta}\|f\|_{L^1},
\]
where the implicit constant is independent of $k$, $\g$, and $\tau$.
\end{Proposition}

\begin{proof}
 We first consider the case $p_0\in(2,\infty)$. Applying the Hausdorff--Young inequality, we obtain
\begin{equation}\label{eq:P:H:Lp:c:2}
 \|\mathcal H_{k,\g,\tau}f\|_{L^{p_0}} \lesssim \left\|\frac{\hat{f}}{|\xi|^2-r_{k,\g}^2 + 2{\rm i}\g \cdot \xi + {\rm i} \tau}\right\|_{L^{p_0'}} \lesssim \|\hat{f}\|_{L^{\infty}} \int_{\mathbb{R}^d}  \frac{\mathbbm{1}_{B(0,R_0)}(\xi)}{\big| |\xi|^2-r_{k,\g}^2 + 2{\rm i}\g \cdot \xi +{\rm i}\tau \big|^{p_0'}} d\xi .
\end{equation}
Let $\beta=|\g|/r_{k,\g}$, so that $0<\beta\leq1$. By a rotation of coordinates, we may assume that $\g=|\g|e_1$. Set $2\tau_0=\tau r_{k,\g}^{-2}$ and make the change of variables $\xi=r_{k,\g}y$, where $y=(y_1,\bar y)$ with $\bar y=(y_2,\ldots,y_d)$. For any $r>0$, define the annulus $A_{r}=\{\xi:r/2<|\xi|<3r/2\}$. Then
\begin{equation}\label{eq:P:H:Lp:c:1}
\int_{A_{r_{k,\g}}}  \frac{1}{\big| |\xi|^2-r_{k,\g}^2 + 2{\rm i}\g \cdot \xi +{\rm i}\tau  \big|^{p_0'}} d\xi =  r_{k,\g}^{d-2p_0'} \int_{A_1} \frac{1}{\big||y|^2-1+2{\rm i}\beta y_1 + 2{\rm i}\tau_0 \big|^{p_0'} } dy.
\end{equation}
The denominator in the integral on the right-hand side vanishes precisely when $y_1=-\tau_0/\beta$ and $|\bar y|^2=1-(\tau_0/\beta)^2$, provided that $|\tau_0|\leq\beta$.

We now write $z=y_1$ and $\bar y=s\omega$, where $s=|\bar y|\geq0$ and $\omega\in\mathbb S^{d-2}$. Introducing the change of variables from $(s,z)$ to $(u,v)$, we set
\[
 u:=|y|^2-1=z^2+s^2-1,\quad v=2z+2\tau_0/\beta.
\]
On the region $|u|<1/4$ and $|v|<1/4$, we have $s=(1+u-(v-2\tau_0/\beta)^2/4)^{1/2}$. Provided that $|\tau_0|/\beta\leq1/2$, the quantity $s$ is uniformly bounded above and away from zero on this region. Since $|\partial(u,v)/\partial(s,z)|=4s$, the change of variables yields
\[
 dy=s^{d-2} ds dz d\omega
    =\tfrac14 s^{d-3} du dv d\omega .
\]

When $d=2$, integration over $\mathbb S^0$ amounts to summation over the two possible signs of $\bar y$. Since $0<\beta\leq1$, we have
\[
 \bigl||y|^2-1+2{\rm i}\beta y_1 + {\rm i}\tau_0 \bigr|^2  =u^2+\beta^2v^2\ge\beta^2(u^2+v^2).
\]
Using the change of variables above, we obtain, for some constant $\rho_0>0$,
 \[
 \int_{A_1} \frac{1}{\big||y|^2-1+2{\rm i}\beta y_1 + {\rm i}\tau_0\big|^{p_0'}}\, dy \lesssim  \beta^{-p_0'}\int_{u^2+v^2<\r_0}(u^2+v^2)^{-p_0'/2}\,du\,dv \lesssim_{p_0'}\beta^{-p_0'}.
\]
Substituting this estimate into \eqref{eq:P:H:Lp:c:1} yields
\[
\left( \int_{A_{r_{k,\g}}}  \frac{1}{\big| |\xi|^2-r_{k,\g}^2 + 2{\rm i}\g \cdot \xi + {\rm i}\tau \big|^{p_0'}}\, d\xi \right)^{1/p_0'} \lesssim_{d,p_0} r_{k,\g}^{d/p_0'-2}\beta^{-1}.
\]

If $A_{r_{k,\g}}$ intersects $B(0,R_0)$, then $r_{k,\g}\lesssim R_0$. Since $|\g|\leq r_{k,\g}$, we have
\[
 r_{k,\g}^{d/p_0'-2}\beta^{-1}=|\g|^{-1}r_{k,\g}^{d/p_0'-1}\lesssim R_0^{d/p_0'} |\g|^{-1}r_{k,\g}^{-1}.
\]
On the other hand, for $\xi\notin A_{r_{k,\g}}$, we have $\big| |\xi|^2-r_{k,\g}^2 \big| \gtrsim r_{k,\g}^2$, and hence
\begin{equation}\label{eq:P:H:p:c:3}
\left( \int_{B(0,R_0)}  \frac{1}{\big| |\xi|^2-r_{k,\g}^2 + 2{\rm i}\g \cdot \xi +{\rm i}\tau \big|^{p_0'}}\, d\xi \right)^{1/p_0'} \lesssim_{d,p_0} R_0^{d/p_0'}|\g|^{-1}r_{k,\g}^{-1}.
\end{equation}
Substituting this estimate into \eqref{eq:P:H:Lp:c:2}, we obtain
\begin{equation}\label{eq:P:H:Lp:c:4}
 \|\mathcal H_{k,\g,\tau}f\|_{L^{p_0}}
 \lesssim_{d,p_0,\eta}|\g|^{-1}r_{k,\g}^{-1}R_0^{d/p_0'}\|f\|_{L^1}.
\end{equation}

We next consider the case $p_0\in[1,2]$. Choose $q\in(2,\infty)$ and $h\in(1,\infty)$ such that
\[
1/p_0 =1/q+1/h,\quad  \eta >1/p_0-1/q.
\]
 By H\"older's inequality, we obtain
\[
 \|\mathcal H_{k,\g,\tau}f\|_{L^{p_0}(\jpx^{-\eta})}
 \leq \|\jpx^{-\eta}\|_{L^h}\|\mathcal H^\pm_{k,\g}f\|_{L^q}.
\]
Since $\eta h>d$, the weighted norm $\|\jpx^{-\eta}\|_{L^h}$ is finite. Therefore,
\begin{equation}\label{eq:P:H:Lp:c:5}
 \|\mathcal H_{k,\g,\tau}f\|_{L^{p_0}(\jpx^{-\eta})}
 \lesssim_{d,p_0,\eta}|\g|^{-2}R_0^{d/q'}\|f\|_{L^1}\lesssim_{d,p_0,\eta}|\g|^{-1}r_{k,\g}^{-1}R_0^{\eta+d/p_0'}\|f\|_{L^1}.
\end{equation}
Combining \eqref{eq:P:H:Lp:c:4} and \eqref{eq:P:H:Lp:c:5} completes the proof.
\end{proof}

\begin{Theorem}\label{Hsg:c}
Let $\lambda\in(0,1)$, $\alpha \in [0,2]$, and $\beta \in [0,2]$.  Let $p_1,p_2,q \in [1,\infty]$ and $r \in \mathbb{R}$ satisfy $
1/p_2-1/p_1=\alpha/d.$ Then the operator $\mathcal{H}_{k,\gamma}^{\pm}$ extends to a bounded operator from $
B^{r-2+\alpha}_{p_2,q}(\langle x\rangle^{\eta},\lambda)$ to $B^{r}_{p_1,q}(\langle x\rangle^{-\eta},\lambda)$. Moreover, for any $\eta >\max\{d/2-\alpha,0\}$, the following estimate holds:
\[
\|\mathcal{H}_{k,\gamma}^{\pm} f\|_{B^{r}_{p_1,q}(\langle x\rangle^{-\eta},\lambda)}
\lesssim_{\eta,p_1,p_2,q,r} \la^{-\eta}  |\g|^{-1}r_{k,\g}^{1-\beta}\max\{1,r_{k,\g}^{-2+\beta}\}
\|f\|_{B^{r-2+\alpha+\beta+\eta}_{p_2,q}(\langle x\rangle^{\eta},\lambda)},
\]
where the implicit constant is independent of $\lambda$, $k$, and $\gamma$.
\end{Theorem}

\begin{proof}
We follow the notation from the proof of Theorem~\ref{Hsg}. By the definition of the Besov norm,
 \begin{align*}
\|\mathcal{H}^{\pm}_{k,\gamma}f\|_{B^{r}_{p_1,q}(\jpx^{-\eta},\lambda)}^q =& \sum_{n=-1}^{\infty}  2^{nrq} \|\mathcal{F}^{-1}(\varphi_{n} \hat{g}_{k,\gamma})\|_{L^{p_1}(\jpl^{-\eta})}^q.
\end{align*}
Applying the weighted Young inequality gives
 \begin{align*}
\|\mathcal{F}^{-1}(\varphi_{n}\hat{g}_{k,\gamma})\|_{L^{p_1}(\jpl^{-\eta})} \lesssim \| \H^+_{k,\g} \check\psi_n^{(n)} \|_{L^{p_0}(\jpl^{-\eta})} \|\Delta_{n} f\|_{L^{p_2}(\jpl^{\eta})},
   \end{align*}
where the exponents satisfy $\a/d=1/p_2-1/p_1=1-1/p_0=1/p_0'$.

We first consider the dyadic blocks with $n\geq 0$. Since $\psi_n^{(n)}(\xi)=\tilde{\psi}(2^{-n}\xi)$, the Fourier transform
of $\check{\psi}_n^{(n)}$ is supported in the annulus $\{\xi\, : 2^{n-1} \leq |\xi|\leq 3\cdot 2^n\}$. If $r_{k,\g}>2^{n-2}$, Proposition~\ref{P:H:Lp:c} gives
\begin{align*}
 \| \H^+_{k,\g} \check\psi_n^{(n)} \|_{L^{p_0}(\jpx^{-\eta})} &\lesssim |\g|^{-1}r_{k,\g}^{-1}\, 2^{(d/p_0'+\eta)n} \|\tilde{\psi}(2^{-n}\xi)\|_{L^{\infty}}\\
 &\lesssim |\g|^{-1}r_{k,\g}^{-1}r_{k,\g}^{2-\beta}2^{-n(2-\beta)} 2^{n\a+n\eta} \lesssim |\g|^{-1}r_{k,\g}^{1-\beta}\, 2^{n(\a+\beta-2)+n\eta}.
 \end{align*}
Suppose now that $r_{k,\g}\leq 2^{n-2}$. On the annulus $\{\xi\in\mathbb{R}^d : 2^{n-1}\leq|\xi|\leq 3\cdot 2^n\}$, as in the proof of \eqref{eq:P:H:p:c:3}, we have
\[
\big| |\xi|^2-r_{k,\g}^2 \big| \geq 2^{n-1}- 2^{n-2} \gtrsim r_{k,\g}.
\]
Therefore,
\[
\left( \int_{\mathbb{R}^d}  \frac{\mathbbm{1}_{\{2^{n-1} \leq |\xi|\leq 3\cdot 2^n\}}}{\big| |\xi|^2-r_{k,\g}^2 + 2{\rm i}\g \cdot \xi \big|^{p_0'}}\, d\xi \right)^{1/p_0'} \lesssim_{d,p_0} 2^{n(\a-2)}\lesssim  |\g|^{-1}r_{k,\g}^{1-\beta}\, r_{k,\g}^{\beta}2^{-\beta n}\,2^{n(\a+\beta-2)}.
\]
Since $|\g| \lesssim r_{k,\g}\lesssim 2^n$, we obtain
\[
\left( \int_{\mathbb{R}^d}  \frac{\mathbbm{1}_{\{2^{n-1} \leq |\xi|\leq 3\cdot 2^n\}}}{\big| |\xi|^2-r_{k,\g}^2 + 2{\rm i}\g \cdot \xi \big|^{p_0'}}\, d\xi \right)^{1/p_0'}\lesssim_{d,p_0}|\g|^{-1}r_{k,\g}^{1-\beta}\, 2^{n(\a+\beta-2)}.
\]

For the low-frequency $n=-1$, Proposition~\ref{P:H:Lp:c} gives
\[
 \| \H^+_{k,\g} \check\psi^{(-1)} \|_{L^{p_0}(\jpx^{-\eta})} \lesssim |\g|^{-1}r_{k,\g}^{-1}.
\]
It follows from Lemma~\ref{Pre00} that
\[
\| \H^+_{k,\g} \check\psi_n^{(n)} \|_{L^{p_0}(\jpl^{-\eta})} \lesssim \la^{-\eta} \| \H^+_{k,\g} \check\psi_n^{(n)} \|_{L^{p_0}(\jpx^{-\eta})}.
\]
Combining this inequality with the preceding estimates and taking the $\ell^q$ norm over $n\geq -1$, we obtain
\[
\|\mathcal{H}^{\pm}_{k,\gamma}f\|_{B^{r}_{p_1,q}(\jpx^{-\eta},\lambda)}^q  \lesssim \la^{-\eta} |\g|^{-1}r_{k,\g}^{1-\beta}\max\{1,r_{k,\g}^{-2+\beta}\}\sum_{n=-1}^{\infty}  2^{nrq} 2^{n(\a+\beta-2+\eta) }\|\Delta_{n} f\|_{L^{p_2}(\jpl^{\eta})},
\]
which completes the proof.
\end{proof}

\begin{Remark}\label{r:Hsg}
Applying the same argument with Theorem~\ref{thm:RF} and Proposition~\ref{P:H:Lp:c}, respectively, we find that the estimates
in Theorems~\ref{Hsg} and \ref{Hsg:c} also hold for $\mathcal{H}_{k,\gamma,\tau}$, uniformly for sufficiently small
$|\tau|>0$. Moreover, as $\tau\to 0^\pm$, $\mathcal{H}_{k,\gamma,\tau}$ converges to $\mathcal{H}_{k,\gamma}^{\pm}$ in the operator norm between the corresponding Besov spaces. The proof of Proposition~\ref{P:H:Lp:c} also shows that one may take $\eta=0$ when $p_0\in(2,\infty]$. Since $\a/d=1-1/p_0$, this range corresponds to $\a>d/2$. Thus, one may take $\eta=0$ in Theorem~\ref{Hsg:c} when $\a>d/2$.
\end{Remark}

\section{Well-posedness}\label{sec:W}

We establish well-posedness for the Helmholtz equation subject to the Sommerfeld radiation condition using the unique continuation principle and the Fredholm alternative. As stated in Section~\ref{sec:intro}, we assume throughout that $V_k$ is compactly supported and that $\Im V_k$ is nonnegative. The equivalence between this radiating problem and the Lippmann--Schwinger equation is classical. For completeness, we prove it at the end of Subsection~\ref{sec:wp}.

\subsection{Functional preliminaries}

To give a rigorous meaning to the product $V_k u$ in the Helmholtz equation, we use Bony's paraproduct decomposition. For tempered
distributions $f$ and $g$, the paraproduct $f\prec g$ and the resonant product $f\circ g$ are given by
\begin{equation*}
f \prec g := \sum_{j\geq -1} S_{j-1} f\ \Delta_{j} g,\quad  f \circ g := \sum_{|i-j|\leq 1} \Delta_{i} f\ \Delta_{j} g,
\end{equation*}
whenever the corresponding series converge in $\mathcal{S}'(\mathbb{R}^d)$. Bony's decomposition defines the product by
\[
fg=f\prec g+g\prec f+f\circ g.
\]
The mapping properties of the paraproduct and the resonant product in rescaled weighted Besov spaces are collected in Appendix~\ref{sec:para}.

Throughout this section, we assume that $\la\in(0,1)$ and $k>0$. We set $V_{k,\la}:=(V_k)_\la$ and define the product $V_{k,\la}u$
by Bony's decomposition,
\begin{align*}
V_{k,\la}u:= V_{k,\la} \prec u + V_{k,\la} \succ u + V_{k,\la} \circ u.
\end{align*}
For sufficiently smooth $V_k$ and $u$, the sum agrees with the usual pointwise product, so this definition extends pointwise multiplication to the distributions considered below.

\begin{Lemma}\label{O1}
Let $p,q \in [1,\infty]$ and $r\in (1, 2)$. For any $\eta \in \mathbb{R}$ and $\epsilon>0$, the following estimate holds:
\[
\|V_{k,\lambda} f\|_{ B^{r-2}_{p,q}(\jpx^{\eta},\lambda))}
\lesssim \max\{\lambda^{r-2}, 1\}\|V_k\|_{B^{r-2+\epsilon}_{\infty,\infty}(\jpx^{2\eta})} \|f\|_{B^{r}_{p,q}(\jpx^{-\eta},\lambda)}.
\]
\end{Lemma}

\begin{proof}
By Minkowski's inequality, we have
\begin{align*}
\|V_{k,\lambda}u\|_{B^{r-2}_{p,q}(\jpx^{\eta},\lambda)} \lesssim \|V_{k,\lambda} \prec u\|_{B^{r-2}_{p,q}(\jpx^{\eta},\lambda)} + \|V_{k, \lambda} \succ u \|_{B^{r-2}_{p,q}(\jpx^{\eta},\lambda)} + \|V_{k,\lambda} \circ u\|_{B^{r-2}_{p,q}(\jpx^{\eta},\lambda)}. 
\end{align*}
For the paraproduct $V_{k,\lambda}\prec u$, we use the embedding from Lemma~\ref{Pre0} and then apply \eqref{PC1:eq:2} to obtain
\[
\|V_{k,\lambda} \prec u\|_{B^{r-2}_{p,q}(\jpx^{\eta},\lambda)} \lesssim \|V_{k,\lambda} \prec u\|_{B^{2r-2}_{p,q}(\jpx^{\eta},\lambda)} \lesssim \|V_{k,\lambda}\|_{B^{r-2}_{\infty,\infty}(\jpx^{2\eta},\lambda)} \|u\|_{B^{r}_{p,q}(\jpx^{-\eta},\lambda)} .
\]
Applying \eqref{PC1:eq:1} to the paraproduct $V_{k,\lambda}\succ u=u\prec V_{k,\lambda}$ yields
\[
\|V_{k,\lambda} \succ u \|_{B^{r-2}_{p,q}(\jpx^{\eta},\lambda)}  \lesssim  \|V_{k,\lambda}\|_{B^{r-2}_{\infty,q}(\jpx^{2\eta},\lambda)} \|u\|_{B^{r}_{p,\infty}(\jpx^{-\eta},\lambda)}.
\]
By Lemma~\ref{Pre0}, for every $\epsilon\in(0,1)$,
\[
\|u\|_{B^{r}_{p,\infty}(\jpx^{-\eta},\lambda)}\lesssim \|u\|_{B^{r}_{p,q}(\jpx^{-\eta},\lambda)},\quad \|V_{k,\lambda}\|_{B^{r-2}_{\infty,q}(\jpx^{2\eta},\lambda)} \lesssim \|V_{k,\lambda}\|_{B^{r-2+\epsilon}_{\infty,\infty}(\jpx^{2\eta},\lambda)}.
\]
These embeddings give
\[
\|V_{k,\lambda} \succ u \|_{B^{r-2}_{p,q}(\jpx^{\eta},\lambda)}  \lesssim \|V_{k,\lambda}\|_{B^{r-2+\epsilon}_{\infty,\infty}(\jpx^{2\eta},\lambda)} \|u\|_{B^{r}_{p,q}(\jpx^{-\eta},\lambda)}.
\]
For the resonant term, Lemma~\ref{Pre0} and \eqref{PC1:eq:3} yield
\[
\|V_{k,\lambda} \circ u\|_{B^{r-2}_{p,q}(\jpx^{\eta},\lambda)} \lesssim \|V_{k,\la} \circ u\|_{B^{2r-2}_{p,q}(\jpx^{\eta},\lambda)}\lesssim \|V_{k,\lambda} \|_{B^{r-2}_{\infty,\infty}(\jpx^{2\eta},\lambda)} \|u\|_{B^{r}_{p,q}(\jpx^{-\eta},\lambda)}.
\]
Combining the estimates for the two paraproducts and the resonant term, and applying Lemma~\ref{R2:n} to the two norms of $V_{k,\lambda}$ with regularity indices $r-2$ and $r-2+\epsilon$, we obtain the desired bound.
\end{proof}

\begin{Lemma}\label{O2}
Let $p\in [1,\infty]$ and $r \in (0,1)$. For any $\eta\in \mathbb{R}$, $\zeta \in (0,2-2r)$, and $\epsilon >0$, the following estimate holds:
\[
\|V_{k,\lambda} f\|_{B^{r-2+\zeta}_{(2p)',(2p)'}(\jpx^{\eta},\lambda)}\lesssim \lambda^{-d/p'} \max\{\lambda^{-r},1\}\|V_k\|_{B^{-r+\e}_{p',p'}(\jpx^{2\eta})} \|f\|_{B^{r}_{2p,2p}(\jpx^{-\eta},\lambda)}.
\]
\end{Lemma}

\begin{proof}
Let $\epsilon\in(0,r)$ and $\psi\in\mathcal{S}(\mathbb{R}^d)$. By Bony's decomposition and the triangle inequality,
\begin{align*}
\|\psi u\|_{B^{r-\e}_{p,p}(\jpx^{-2\eta},\lambda)} \lesssim \|\psi \prec u\|_{B^{r-\e}_{p,p}(\jpx^{-2\eta},\lambda)} +  \|u \prec \psi\|_{B^{r-\e}_{p,p}(\jpx^{-2\eta},\lambda)} +  \|u \circ \psi\|_{B^{r-\e}_{p,p}(\jpx^{-2\eta},\lambda)} .
\end{align*}
Since $0<\zeta<2-2r$, we have $2-r-\zeta>r>0$. Applying \eqref{PC1:eq:1} and the Besov embeddings from Lemma~\ref{Pre0}, we obtain
\begin{align*}
 \|\psi \prec u\|_{B^{r-\e}_{p,p}(\jpx^{-2\eta},\lambda)} &\lesssim \|\psi \|_{B^{2-r-\zeta}_{2p,\infty}(\jpx^{-\eta},\lambda)} \| u\|_{B^{r-\e}_{2p,p}(\jpx^{-\eta},\lambda)} \\
 &\lesssim \|\psi \|_{B^{2-r-\zeta}_{2p,2p}(\jpx^{-\eta},\lambda)} \| u\|_{B^{r}_{2p,2p}(\jpx^{-\eta},\lambda)}.
\end{align*}
The assumption $\zeta\in (0,2-2r)$ ensures that $2r-2+\zeta-\e<0$. Similarly, by \eqref{PC1:eq:2} and Lemma~\ref{Pre0},
\begin{align*}
 \|u \prec \psi\|_{B^{r-\e}_{p,p}(\jpx^{-2\eta},\lambda)}&\lesssim  \|u \|_{B^{2r-2+\zeta-\e}_{2p,2p}(\jpx^{-\eta},\lambda)}  
 \| \psi\|_{B^{2-r-\zeta}_{2p,2p}(\jpx^{-\eta},\lambda)}\\
  &\lesssim  \|u \|_{B^{r}_{2p,2p}(\jpx^{-\eta},\lambda)}  
 \| \psi\|_{B^{2-r-\zeta}_{2p,2p}(\jpx^{-\eta},\lambda)}. 
\end{align*}
Since  $\zeta\in (0,2-2r)$ , we have $2-\zeta>2r>r-\e$. For the resonant term, we first use the embedding from
Lemma~\ref{Pre0} and then apply \eqref{PC1:eq:3} to obtain
\[
 \|u \circ \psi\|_{B^{r-\e}_{p,p}(\jpx^{-2\eta},\lambda)} 
  \lesssim  \|u \circ \psi\|_{B^{2-\zeta}_{p,p}(\jpx^{-2\eta},\lambda)} 
  \lesssim \|u \|_{B^{r}_{2p,2p}(\jpx^{-\eta},\lambda)} \| \psi\|_{B^{2-r-\zeta}_{2p,2p}(\jpx^{-\eta},\lambda)}.
\]
Combining the above estimates yields
\[
\|\psi u\|_{B^{r-\e}_{p,p}(\jpx^{-2\eta},\lambda)} \lesssim \|\psi \|_{B^{2-r-\zeta}_{2p,2p}(\jpx^{-\eta},\lambda)} \| u\|_{B^{r}_{2p,2p}(\jpx^{-\eta},\lambda)}.
\]

Using Besov duality (cf. \cite[Section~2.11.2]{T83} or \cite[Proposition~2.76]{BCD11}) and the preceding estimate for $\psi u$,
we obtain, for every $\psi\in\mathcal{S}(\mathbb{R}^d)$,
\begin{align*}
\langle \psi, V_{k,\lambda}u \rangle = \langle \psi u, V_{k,\lambda} \rangle &\lesssim  \|V_{k,\lambda}\|_{B^{-r+\epsilon}_{p',p'}(\jpx^{2\eta},\lambda)} \|\psi u\|_{B^{r-\epsilon}_{p,p}(\jpx^{-2\eta},\lambda)} \\
&\lesssim  \|V_{k,\lambda}\|_{B^{-r+\epsilon}_{p',p'}(\jpx^{2\eta},\lambda)}  \|\psi\|_{B^{2-r-\zeta}_{2p,2p}(\jpx^{-\eta},\lambda)} \|u\|_{B^{r}_{2p,2p}(\jpx^{-\eta},\lambda)}.
\end{align*}
By duality, it follows that
\[
\| V_{k,\lambda}u \|_{B^{r-2+\zeta}_{(2p)',(2p)'}(\jpx^{\eta},\lambda)} \lesssim \|V_{k,\lambda}\|_{B^{-r+\epsilon}_{p',p'}(\jpx^{2\eta},\lambda)}  \|u\|_{B^{r}_{2p,2p}(\jpx^{-\eta},\lambda)}.
\]
Applying the dilation estimate in Lemma~\ref{R2:n} to the $V_{k,\lambda}$ norm completes the proof.
\end{proof}

\begin{Lemma}\label{O3}
Let $p,q \in [1,\infty]$, $r_0 \in (0,\infty)$, $\Phi \in C_{c}^{\infty}(\mathbb{R}^d)$, and set $\Phi_{\la}:=(\Phi)_{\la}$. For any $r \in (-r_0,r_0)$ and any pair of admissible weights $\rho_1$ and $\rho_2$, the following estimate holds: 
\[
\|\Phi_{\lambda} f\|_{B^{r}_{p,q}(\rho_2,\lambda) )} \lesssim  \|\Phi\|_{B^{r_0}_{\infty,\infty}(\rho_2\rho_1^{-1})}\|f\|_{B^{r}_{p,q}(\rho_1,\lambda)}.
\]
\end{Lemma}

\begin{proof}
Since $r+r_0>0$, applying \eqref{PC1:eq:1}--\eqref{PC1:eq:3}, we obtain 
\[
\|\Phi_{\lambda}f\|_{B^{r}_{p,q}(\rho_2,\lambda)} \lesssim \|\Phi_{\lambda}\|_{B^{r_0}_{\infty,\infty}(\rho_2\rho_1^{-1},\lambda)} \|f\|_{B^{r}_{p,q}(\rho_1,\lambda)}.
\]
Using Lemma~\ref{R2:n} yields
\[
\|\Phi_{\lambda}f\|_{B^{r}_{p,q}(\rho_2,\lambda)} \lesssim  \|\Phi\|_{B^{r_0}_{\infty,\infty}(\rho_2\rho_1^{-1})} \|f\|_{B^{r}_{p,q}(\rho_1,\lambda)}, 
\]
which completes the proof. 
\end{proof}

\begin{Lemma}\label{W1}
Let $\eta_0 \in (1/2,1)$, $r \in (1,2)$, $\Phi \in C_{c}^{\infty}(\mathbb{R}^d)$, and set $\Phi_{\la}:=(\Phi)_{\la}$. For any $\epsilon \in (0,2-r)$, the following estimate holds: 
\begin{align*}
\|\Phi_{\lambda} \mathcal{H}_{k\lambda,\gamma}^+ \Phi_{\lambda} V_{k,\lambda}f\|_{B^{r}_{2,2}(\jpx^{-\eta_0},\lambda)}  &\lesssim_{\g} \lambda^{r-2-2\eta_0}  \|V_k\|_{B^{r-2+\epsilon}_{\infty,\infty}(\jpx^{2\eta_0})} \|f\|_{B^{r}_{2,2}(\jpx^{-\eta_0},\lambda)},
\end{align*}
where the implicit constant is independent of $\la$.
\end{Lemma}

\begin{proof}
Applying Lemma~\ref{O3} with $r_0>|r|$ to the cutoff $\Phi_\lambda$, we obtain 
\begin{align*}
\|\Phi_{\lambda} \mathcal{H}_{k\lambda,\gamma}^+ \Phi_{\lambda} V_{k,\lambda} f\|_{B^{r}_{2,2}(\jpx^{-\eta_0},\lambda)} \lesssim \| \mathcal{H}_{k\lambda,\gamma}^+ \Phi_{\lambda} \Xi_{k,\lambda}f\|_{B^{r}_{2,2}(\jpx^{-\eta_0},\lambda)}. 
\end{align*}
Let $r_{k\lambda,\gamma} := (k^2\lambda^2 + |\gamma|^2)^{1/2}$.  By Theorem~\ref{thm:F}, we have
\begin{align}\label{W1:1}
\|\mathcal{H}_{k\lambda,\gamma}^+ \Phi_{\lambda} \Xi_{k,\lambda} f\|_{B^{r}_{2,2}(\jpx^{-\eta_0},\lambda)} 
&\lesssim \lambda^{-2\eta_0}  \frac{(1+r_{k\lambda,\gamma})^2}{\min\{r_{k\lambda,\gamma}, 1\}r_{k\lambda,\gamma}}\| \Phi_{\lambda}  \Xi_{k,\lambda} f\|_{B^{r-2}_{2,2}(\jpx^{\eta_0},\lambda)}.
\end{align}
Applying Lemmas~\ref{O1} and~\ref{O3}, we obtain that for any $\epsilon>0$,
\begin{align*}
\|\Phi_{\lambda} \mathcal{H}_{k\lambda,\gamma}^+ \Phi_{\lambda} \Xi_{k,\lambda}f\|_{B^{r}_{2,2}(\jpx^{-\eta_0},\lambda)}  &\lesssim \frac{\lambda^{-2\eta_0}(1+r_{k\lambda,\gamma})^2}{\min\{r_{k\lambda,\gamma}, 1\}r_{k\lambda,\gamma}} \|\Xi_{k,\lambda} f\|_{B^{r-2}_{2,2}(\jpx^{\eta_0},\lambda)}\\
&\lesssim \frac{\|V_k\|_{B^{r-2+\epsilon}_{\infty,\infty}(\jpx^{2\eta_0})}(1+r_{k\lambda,\gamma})^2}{\lambda^{2+2\eta_0-r-\e}\min\{r_{k\lambda,\gamma}, 1\}r_{k\lambda,\gamma}} \|  f\|_{B^{r}_{2,2}(\jpx^{-\eta_0},\lambda)}.
\end{align*}
which completes the lemma.
\end{proof}

\subsection{Unique continuation principle}

Suppose that $u$ solves the homogeneous Lippmann--Schwinger equation
\begin{equation}\label{eq:uq:cont}
u=\H^{+}_{k} V_{k} u .
\end{equation}
Applying $\Delta+k^2$ to both sides shows that $u$ satisfies 
\begin{equation}\label{eq:uq:cont:1}
(\Delta + k^2) u +  V_k u=0
\end{equation}
in the sense of distributions. Once the compact support of $u$ has been established, we use unique continuation to show that $u$ vanishes identically.

We start with a formal calculation. Substituting $x=\lambda x'$ in \eqref{eq:uq:cont:1}, we obtain
\[
\D u(\lambda x') + k^2 u(\lambda x') + V_k(\lambda x') u(\lambda x') =0,\quad x' \in \mathbb{R}^d.
\]
Define the rescaled function $u_\lambda(x') := u(\lambda x')$ and $V_{k,\lambda}:= (V_k)_{\lambda}$.  Since $\Delta (u)_{\lambda}(x')= \lambda ^2\Delta u(\lambda x')$, we arrive at the rescaled equation
\begin{equation*}
\Delta u_{\lambda}(x)+\lambda^2 k^2 u_{\lambda}(x) + \lambda^2V_{k,\lambda} (x) u_\lambda (x) =0,\quad x \in \mathbb{R}^d.
\end{equation*}
In particular, both $u_{\la}$ and $V_{k,\la}$ are compactly supported in $B(0,\la^{-1}R)$. 

Let $\tau>0$ be the attenuation parameter. Rescaling \eqref{eq:uq:cont:1} and adding ${\rm i}\tau u_\lambda$ to both sides,
we obtain 
\begin{equation*}
(\Delta +\lambda^2 k^2 +{\rm i}\tau ) u_{\la}(x)  + \lambda^2V_{k,\lambda} (x) u_\lambda (x) ={\rm i}\tau u_{\la}(x).
\end{equation*}
For $\g \neq 0$, define $g_{\la,\g}(x):=e^{\g \cdot x}u_{\la}(x)$. It is clear to note that 
\begin{equation}\label{eq:uq:cont:3}
(\Delta +\lambda^2 k^2 + \g^2 - 2\g\cdot \nabla +{\rm i}\tau ) g_{\la,\g}(x) + \lambda^2 V_{k,\lambda} (x) g_{\la,\g} (x) ={\rm i}\tau g_{\la,\g}(x).
\end{equation}

In this subsection, we assume that $V_k$ is compactly supported in $B(0,R)$. We fix a cutoff function
$\Phi_R\in C_c^\infty(\mathbb{R}^d)$ with $\Phi_R= 1$ on $B(0,R)$ and write $\Phi_{R,\lambda}:=(\Phi_R)_\lambda$ for its dilation. 

\begin{Lemma}\label{W2}
Let $\la\in (0,1)$, $r \in (0,1)$, and $p_0\in (1,d/(d-2+2r))$. For any $\e>0$, the following estimate holds: 
\begin{align*}
\|\Phi_{R,\lambda} \mathcal{H}_{k\lambda,\gamma}^+ \Phi_{R,\lambda} V_{k,\lambda}\ f\|_{B^{r}_{2p_0,2p_0}} &\lesssim_{\g,R} \lambda^{-1-d/(2p_0') -r} \|V_{k}\|_{B^{-r+\epsilon}_{p_0',p_0'}} \|\Phi_{R,\la}f\|_{B^{r}_{2p_0,2p_0}},
\end{align*}
where the implicit constant is independent of $\la$.
\end{Lemma}

\begin{proof}
 Let  $p' \in (d/2,\min\{2,d/(2-2r)\})$ and $\theta =p'/p_0'$. Since $p_0\in (1,d/(d-2+2r))$, we have $p_0' \in (d/(2-2r),\infty]$ and $\theta \in (0,1)$. Let $\eta_0$ satisfy 
 \[
 2\eta_0(1-\theta)= 1-d/(2p_0') > 1-\theta.
 \]
 It follows from $\theta =p'/p_0' >d/(2p_0')$ that we have $\eta_0>1/2$ and 
\[
\theta/(2p)+(1-\theta)/2=1/(2p_0).
\]
Since $p'>d/2$, Theorem~\ref{Hsg:c} gives
\begin{equation}\label{eq:eta:estimate}
\|\mathcal{H}_{k\lambda,\gamma}^+ \Phi_{R,\lambda} V_{k,\lambda}f\|_{B^{r}_{2p,2p}} \lesssim \  |\g|^{-1}\max\{r_{k,\g},r_{k,\g}^{-1}\} \| \Phi_{R,\lambda} V_{k,\lambda} f\|_{B^{r-2+d/p'}_{(2p)',2p}}.
\end{equation}
For any $\zeta_0>r$, applying Lemma~\ref{O3} yields 
\begin{equation*}
\begin{aligned}
\|\Phi_{R,\lambda}\mathcal{H}_{k\lambda,\gamma}^+ \Phi_{R,\lambda} V_{k,\lambda}f\|_{B^{r}_{2p,q}} 
&\lesssim\|\Phi_R\|_{B^{\zeta_0}_{\infty,\infty}}\|\mathcal{H}_{k\lambda,\gamma}^+ \Phi_{R,\lambda} V_{k,\lambda}f\|_{B^{r}_{2p,q}} \\
&\lesssim |\g|^{-1}\max\{r_{k,\g},r_{k,\g}^{-1}\} \|\Phi_R\|_{B^{\zeta_0}_{\infty,\infty}}^2 \| V_{k,\lambda} f\|_{B^{r-2+d/p'}_{(2p)',q}}.
\end{aligned}
\end{equation*}
Since $(2p)'\leq 2p$, combining the above estimates with the embedding from Lemma~\ref{Pre0} gives
\begin{align*}
\|\Phi_{R,\lambda} \mathcal{H}_{k\lambda,\gamma}^+ \Phi_{R,\lambda} V_{k,\lambda}f\|_{B^{r}_{2p,2p}}  \lesssim |\g|^{-1}\max\{r_{k,\g},r_{k,\g}^{-1}\}  \|V_{k,\lambda} f\|_{B^{r-2+d/p'}_{(2p)',(2p)'}}.
\end{align*}
Since $V_k$ is compactly supported in $B(0,R)$ and $\Phi_{R,\la}$ is compactly supported in $B(0,\la^{-1}R)$, we have $(1+R)^{-\eta_0} \leq \langle \la x \rangle^{-\eta_0} \leq 1$ for $|x|\leq \la^{-1}R$. By Lemma~\ref{Pre00}, we deduce 
\begin{equation}\label{eq:compact:equi}
\begin{aligned}
\|\Phi_{R,\lambda} \mathcal{H}_{k\lambda,\gamma}^+ \Phi_{R,\lambda} V_{k,\lambda}\ f\|_{B^{r}_{2p_0,2p_0}} &\asymp_R \|\Phi_{R,\lambda} \mathcal{H}_{k\lambda,\gamma}^+ \Phi_{R,\lambda} V_{k,\lambda}\ f\|_{B^{r}_{2p_0,2p_0}(\jpx^{-\eta_0},\lambda)}\ ,\\
\|\Phi_{R,\la}f\|_{B^{r}_{2p_0,2p_0}(\jpx^{-\eta_0},\lambda)} &\asymp_R  \|\Phi_{R,\la}f\|_{B^{r}_{2p_0,2p_0}}\ ,
\end{aligned}
\end{equation}
where the implicit constant are independent of $\la$. Hence, we have
\[
\|\Phi_{R,\lambda} \mathcal{H}_{k\lambda,\gamma}^+ \Phi_{R,\lambda} V_{k,\lambda}f\|_{B^{r}_{2p,2p}(\jpx^{-\eta_0},\lambda))}  \lesssim |\g|^{-1}\max\{r_{k,\g},r_{k,\g}^{-1}\}  \|V_{k,\lambda} f\|_{B^{r-2+d/p'}_{(2p)',(2p)'}(\jpx^{\eta_0},\lambda))}.
\]

By \eqref{W1:1}, we also have the $L^2$-based estimate 
\[
\|\Phi_{R,\lambda} \mathcal{H}_{k\lambda,\gamma}^+ \Phi_{R,\lambda}  V_{k,\lambda} f\|_{B^{r}_{2,2}(\jpx^{-\eta_0},\lambda)}  \lesssim \lambda^{-2\eta_0} \frac{(1+r_{k\lambda,\gamma})^2}{\min\{r_{k\lambda,\gamma}, 1\}r_{k\lambda,\gamma}} \| V_{k,\lambda}f\|_{B^{r-2}_{2,2}(\jpx^{\eta_0},\lambda)}.
\]
Interpolating this estimate with the preceding $L^{2p}$-based estimate by Lemmas~\ref{real:interpolation} and \ref{I:Besov}, we obtain 
\[
\begin{aligned}
\|\Phi_{R,\lambda} \mathcal{H}_{k\lambda,\gamma}^+ \Phi_{R,\lambda} V_{k,\lambda} f\|_{B^{r}_{2p_0,2p_0}(\jpx^{-\eta_0},\lambda)} &\lesssim \lambda^{-2\eta_0(1-\theta)}  \left(\frac{(1+r_{k\lambda,\gamma})^2}{\min\{r_{k\lambda,\gamma}, 1\}r_{k\lambda,\gamma}}\right)^{1-\theta}\\
&\quad \times |\g|^{-\theta} \max\{ r_{k,\gamma}^{}, r_{k,\gamma}^{-1}\}^{\theta}\, \|V_{k,\lambda} f\|_{B^{r-2+d/p_0'}_{(2p_0)',(2p_0)'}(\jpx^{\eta_0},\lambda)}. 
\end{aligned}
\] 

Since $p_0 \in [1,d/(d-2+2r))$, we have $d/p_0'<2-2r$. Applying Lemma~\ref{O2} gives 
\begin{equation}\label{eq:para:estimate}
\| V_{k,\lambda} f\|_{B^{r-2+d/p_0'}_{(2p_0)',(2p_0)'}(\jpx^{\eta_0},\lambda)} \lesssim \lambda^{-d/p_0'} \lambda^{-r} \|V_{k}\|_{B^{-r+\e}_{p_0',p_0'}(\jpx^{2\eta_0})} \| \Phi_{R,\la}f\|_{B^{r}_{2p_0,2p_0}(\jpx^{-\eta_0},\lambda)}.  
\end{equation}
Combining the above estimates, we obtain 
\begin{align*}
\|\Phi_{R,\lambda} \mathcal{H}_{k\lambda,\gamma}^+ \Phi_{R,\lambda} V_{k,\lambda}\ f\|_{B^{r}_{2p_0,2p_0}(\jpx^{-\eta_0},\lambda)} &\lesssim \lambda^{-2\eta_0(1-\theta)-d/p_0'-r} \left(\frac{(1+r_{k\lambda,\gamma})^2}{\min\{r_{k\lambda,\gamma}, 1\}r_{k\lambda,\gamma}}\right)^{1-\theta}\\
&\quad\times |\g|^{-\theta} \max\{ r_{k,\gamma}^{}, r_{k,\gamma}^{-1}\}^{\theta}\, \|V_{k}\|_{B^{-r+\e}_{p_0',p_0'}(\jpx^{2\eta_0})}\|\Phi_{R,\la}f\|_{B^{r}_{2p_0,2p_0}(\jpx^{-\eta_0},\lambda)}.
\end{align*}
By the definition of $\eta_0$, we have 
\begin{equation}\label{eq:eta:condition}
-2\eta_0(1-\theta)-d/p_0'-r= -1-d/(2p_0') -r. 
\end{equation}
Substituting \eqref{eq:eta:condition} into the above estimate yields 
\begin{align}\label{eq:weight:est}
\|\Phi_{R,\lambda} \mathcal{H}_{k\lambda,\gamma}^+ \Phi_{R,\lambda} V_{k,\lambda}\ f\|_{B^{r}_{2p_0,2p_0}(\jpx^{-\eta_0},\lambda)} \lesssim_{\g} \lambda^{-1-d/(2p_0') -r} \|V_{k}\|_{B^{-r+\e}_{p_0',p_0'}(\jpx^{2\eta_0})}\|\Phi_{R,\la}f\|_{B^{r}_{2p_0,2p_0}(\jpx^{-\eta_0},\lambda)}. 
\end{align}
Finally, we apply \eqref{eq:compact:equi} again and complete the proof. 
\end{proof}

\begin{Proposition}\label{P:uq:cont:1}
Let $r\in(0,1)$, $\e>0$, and $1<p_0<d/(d-2+2r)$, and assume that $V_k\in B^{-r+\e}_{p_0',p_0'}$ is compactly supported. If a compactly
supported distribution $u\in B^r_{2p_0,2p_0}$ satisfies \eqref{eq:uq:cont}, then $u= 0$. 
\end{Proposition}

\begin{proof}
Enlarging $R$ if necessary, we may assume that $u$ is also supported in $B(0,R)$. Set $s:=r-2+d/p_0'$ and choose
$\zeta_0>\max\{r,|s|\}$. The cutoff multiplier estimate from Lemma~\ref{O3}, followed by \eqref{eq:para:estimate} and Lemma~\ref{R2:n}, gives
\[
\begin{aligned}
\| \Phi_{R,\la}  V_{k,\la} g_{\la,\g}\|_{B^{r-2+d/p_0'}_{(2p_0)',(2p_0)'}(\jpx^{\eta},\la)} &\lesssim \|e^{-\g \cdot x} \Phi_{R,\la}\|_{B^{\zeta_0}_{\infty,\infty}(\jpx^{\eta_0-\eta},\la)} \|V_{k,\la} u_{\la}\|_{B^{r-2+d/p_0'}_{(2p_0)',(2p_0)'}(\jpx^{\eta_0},\lambda)}\\
&\lesssim \|V_{k}\|_{B^{-r+\epsilon}_{p_0',p_0'}} \|u\|_{B^{r}_{2p_0,2p_0}(\jpx^{-\eta_0})}. 
\end{aligned}
\]
Combining the above estimate with Lemma~\ref{O3}, Theorem~\ref{Hsg}, and Remark~\ref{r:Hsg}, we obtain
\begin{equation}\label{eq:P:uq:cont:1}
\|\Phi_{R,\la}\H_{k\la,\g,\tau}  \Phi_{R,\la}  V_{k,\la} g_{\la,\g}\|_{B^{r}_{2p_0,2p_0}(\jpx^{-\eta_0},\la)} \lesssim \|V_{k}\|_{B^{-r+\epsilon}_{p_0',p_0'}} \|u\|_{B^{r}_{2p_0,2p_0}(\jpx^{-\eta_0})},
\end{equation}
which implies that $\Phi_{R,\la} \H_{k\la,\g,\tau}  \Phi_{R,\la}  V_{k,\la}\, g_{\la,\g}$ belongs to $B^{r}_{2p_0,2p_0}(\jpx^{-\eta_0},\la)$. Note that  $g_{\la,\g}= e^{-\g\cdot x}\Phi_{R,\la} u_{\la}$. It follows from Lemmas~\ref{O3} and \ref{R2:n} that we have
\begin{equation}\label{eq:P:uq:cont:2}
\|g_{\la,\g}\|_{B^{r}_{2p_0,2p_0}(\jpx^{-\eta_0},\la)} \lesssim  \|e^{-\g\cdot x} \Phi_{R,\la}\|_{B^{\zeta_0}_{\infty,\infty}(\jpx^{2\eta_0},\la)}  \|u\|_{B^{r}_{2p_0,2p_0}(\jpx^{-\eta_0})}. 
\end{equation}
Since $\Phi_{R,\la}$ is compactly supported, $g_{\la,\g}$ belongs to $B^{r}_{2p_0,2p_0}(\jpx^{-\eta_0},\la)$. 

Therefore, we can apply Theorem~\ref{thm:RF} to \eqref{eq:uq:cont:3} and obtain that
\[
g_{\la,\g}= \la^2\, \Phi_{R,\la} \H_{k\la,\g,\tau}  \Phi_{R,\la}  V_{k,\la}\, g_{\la,\g} + {\rm i} \tau \Phi_{R,\la} \H_{k\la,\g,\tau} \Phi_{R,\la} g_{\la,\g}. 
\]

By Theorem~\ref{Hsg} and Remark~\ref{r:Hsg}, we may pass to the limit as $\tau\to 0^+$ in the localized equation to obtain 
\[
g_{\la,\g}= \la^2 \Phi_{R,\la}\H_{k\la,\g}^+  \Phi_{R,\la}  V_{k,\la}\, g_{\la,\g}. 
\]
Since $g_{\la,\g}$ is compactly supported, \eqref{eq:P:uq:cont:2} and the norm equivalence used in \eqref{eq:compact:equi} show that
$g_{\la,\g}\in B^r_{2p_0,2p_0}$. Applying Lemma~\ref{W2} to the identity above yields
\[
\|g_{\la,\g}\|_{B^{r}_{2p_0,2p_0}} \lesssim_{\g,R}  \lambda^{1-d/(2p_0') -r} \|V_{k}\|_{B^{-r+\epsilon}_{p_0',p_0'}} \|g_{\la,\g}\|_{B^{r}_{2p_0,2p_0}}.
\]
The assumption $p_0'>d/(2-2r)$ ensures that $1-d/(2p_0')-r>0$. Fixing $\g\neq 0$ and choosing $\la>0$ sufficiently small, we make the coefficient on the right-hand side strictly less than one. It follows that $g_{\la,\g}= 0$, which implies $u= 0$. 
\end{proof}

\begin{Lemma}\label{W3}
Let $r \in (1,2)$, $p_0 \in [1,\infty]$, and
\[
r_0= 
\begin{cases}
(d+1)/2- (d-1)/p_0, & p_0\in [2,\infty],\\
(d+1)/2- (d-1)/p_0', & p_0\in [1, 2]. 
\end{cases}
\]
For any $\epsilon>0$ and $V_k \in B^{r-2+\e}_{\infty,\infty}$, the following estimate holds: 
\[
\begin{aligned}
\|\Phi_{R,\lambda} \mathcal{H}_{k\lambda,\gamma}^+ \Phi_{R,\lambda} V_{k,\lambda}\ f\|_{B^{r}_{p_0,p_0}} \lesssim_{\g} \lambda^{r-2-r_0} \|V_k\|_{B^{r-2+\epsilon}_{\infty,\infty}}\|\Phi_{R,\la}f\|_{B^{r}_{p_0,p_0}}.
\end{aligned}
\]
\end{Lemma}

\begin{proof}
Let $\eta_0 \in (1/2, 1/2+\e/4)$ and $\eta= (d+1)/2 +\e$. Applying Theorem~\ref{Hsg}, we obtain
\begin{align*}
\|\Phi_{R,\lambda}\mathcal{H}_{k\lambda,\gamma}^+ \Phi_{R,\lambda}  V_{k,\lambda}\ f\|_{B^{r}_{p,q}(\jpx^{-\eta},\lambda)} \lesssim \lambda^{-\eta} \max\{ r_{k,\gamma}^{\eta}, r_{k,\gamma}^{-2}\}  \| \Phi_{R,\lambda} V_{k,\lambda} f\|_{B^{r-2}_{p,q}(\jpx^{\eta},\lambda)}.
\end{align*}
We first consider the endpoint cases $p=q=1$ and $p=q=\infty$. Using similar proof of Lemma~\ref{W2} and applying Lemma~\ref{O1}, we obtain
\begin{equation}\label{W2:1}
\begin{aligned}
\|\Phi_{R,\lambda} \mathcal{H}_{k\lambda,\gamma}^+ \Phi_{R,\lambda} V_{k,\lambda} f\|_{B^{r}_{p,p}(\jpx^{-\eta_0},\lambda)}  & \lesssim \lambda^{r-2-\eta+\e} \max\{ r_{k,\gamma}^{\eta}, r_{k,\gamma}^{-2}\}\\
&\quad \times \|V_k\|_{B^{r-2+\e}_{\infty,\infty}(\jpx^{2\eta_0})} \| \Phi_{R,\la} f\|_{B^{r}_{p,p}(\jpx^{-\eta_0},\lambda)}. 
\end{aligned}
\end{equation}
On the other hand, by applying \eqref{W1:1}, we have the $L^2$-based estimate
\[
\begin{aligned}
\|\Phi_{R,\lambda} \H_{k\lambda,\gamma}^+ \Phi_{R,\lambda} V_{k,\lambda} f\|_{B^{r}_{2,2}(\jpx^{-\eta_0},\lambda)}  & \lesssim \lambda^{r-2-2\eta_0+\e} \frac{(1+r_{k\lambda,\gamma})^2}{\min\{r_{k\lambda,\gamma}, 1\}r_{k\lambda,\gamma}}\\
& \times \|V_k\|_{B^{r-2+\epsilon}_{\infty,\infty}(\jpx^{2\eta_0})} \| \Phi_{R,\la} f\|_{B^{r}_{2,2}(\jpx^{-\eta_0},\lambda)}.
\end{aligned}
\]

For any $p_0 \in [2,\infty)$, we choose $p=\infty$ in \eqref{W2:1} and write $p_0=2/\theta$ for some $\theta\in [0,1]$. Applying Lemmas~\ref{real:interpolation} and~\ref{I:Besov}, we get 
\[
\begin{aligned}
\|\Phi_{R,\lambda} \mathcal{H}_{k\lambda,\gamma}^+ \Phi_{R,\lambda} V_{k,\lambda}\ f\|_{B^{r}_{p_0,p_0}(\jpx^{-\eta_0},\lambda)} &\lesssim \lambda^{r-2-\eta(1-\theta)-2\eta_0\t+\e} \left( \frac{(1+r_{k\lambda,\gamma})^2}{\min\{r_{k\lambda,\gamma}, 1\}r_{k\lambda,\gamma}} \right)^{\theta}\\
&\quad  \times \max\{ r_{k,\gamma}^{\eta}, r_{k,\gamma}^{-2}\}^{1-\theta}\  \|V_k\|_{B^{r-2+\epsilon}_{\infty,\infty}(\jpx^{2\eta_0})}\|f\|_{B^{r}_{p_0,p_0}(\jpx^{-\eta_0},\lambda)}.
\end{aligned}
\]
For any $p_0\in[1,2]$, we choose $p=1$ in \eqref{W2:1} and write $p_0'=2/\theta$ for some $\theta \in [0,1]$. We may apply Lemmas~\ref{real:interpolation} and~\ref{I:Besov} again and obtain the same estimate.

Finally, for $p_0\in[2,\infty)$, we have 
\[
-\eta(1-\theta)-2\eta_0\t + \e= -r_0+(2+\e-4\eta_0)/p_0. 
\]
For $p_0\in[1,2)$, the same identity holds with $p_0'$ in the denominator on the right-hand side. As in the proof of
\eqref{eq:compact:equi}, we use compact support to obtain the desired unweighted estimate.
\end{proof}

\begin{Proposition}\label{P:uq:cont:2}
Let $p_0\in[1,\infty]$ and $r\in(r_0,2)$, where $r_0$ is the constant defined in Lemma~\ref{W3}. Let $\e>0$ and suppose that
$V_k\in B^{r-2+\e}_{\infty,\infty}$ is compactly supported. If a compactly supported distribution $u\in B^r_{p_0,p_0}$
satisfies \eqref{eq:uq:cont}, then $u=0$.
\end{Proposition}

\begin{proof}
We may assume that $u$ is supported in $B(0,R)$. By Lemmas~\ref{O3}, \ref{O1}, and \ref{R2:n}, we have that for any $\zeta_0>r$, 
\[
\begin{aligned}
\| \Phi_{R,\la}  V_{k,\la} g_{\la,\g}\|_{B^{r-2}_{p_0,p_0}(\jpx^{\eta},\la)} &\lesssim \|e^{-\g \cdot x} \Phi_{R,\la}\|_{B^{\zeta_0}_{\infty,\infty}(\jpx^{\eta_0-\eta},\la)} \|V_{k,\la} u_{\la}\|_{B^{r-2}_{p_0,p_0}(\jpx^{\eta_0},\lambda)}\\
&\lesssim \|V_{k}\|_{B^{r-2+\epsilon}_{\infty,\infty}} \|u\|_{B^{r}_{p_0,p_0}(\jpx^{-\eta_0})}. 
\end{aligned}
\]
By Lemma~\ref{O3}, Theorem~\ref{Hsg}, and Remark~\ref{r:Hsg}, we have 
\begin{equation}\label{eq:P2:uq:cont:1}
\|\Phi_{R,\la}\H_{k\la,\g,\tau}  \Phi_{R,\la}  V_{k,\la} g_{\la,\g}\|_{B^{r}_{p_0,p_0}(\jpx^{-\eta_0},\la)} \lesssim \|V_{k}\|_{B^{r-2+\epsilon}_{\infty,\infty}} \|u\|_{B^{r}_{p_0,p_0}(\jpx^{-\eta_0})}. 
\end{equation}
Noting $g_{\la,\g}= e^{-\g\cdot x}\Phi_{R,\la} u_{\la}$, we obtain from Lemma~\ref{O3} and \ref{R2:n} that 
\begin{equation}\label{eq:P2:uq:cont:2}
\|g_{\la,\g}\|_{B^{r}_{p_0,p_0}(\jpx^{-\eta_0},\la)} \lesssim  \|e^{-\g\cdot x} \Phi_{R,\la}\|_{B^{\zeta_0}_{\infty,\infty}(\jpx^{2\eta_0},\la)}  \|u\|_{B^{r}_{p_0,p_0}(\jpx^{-\eta_0})}
\end{equation}
Similar to the proof in Proposition~\ref{P:uq:cont:1}, we apply Theorem~\ref{thm:RF} to \eqref{eq:uq:cont:3} and obtain that
\[
g_{\la,\g}= \la^2\, \Phi_{R,\la} \H_{k\la,\g,\tau}  \Phi_{R,\la}  V_{k,\la}\, g_{\la,\g} + {\rm i} \tau\,  \Phi_{R,\la} \H_{k\la,\g,\tau} \Phi_{R,\la} g_{\la,\g}. 
\]

By Theorem~\ref{Hsg} and Remark~\ref{r:Hsg}, taking $\tau \to 0^+$, we have  
\[
g_{\la,\g}= \la^2 \Phi_{R,\la}\H_{k\la,\g}^+  \Phi_{R,\la}  V_{k,\la}\, g_{\la,\g}. 
\]
Following the similar proof of \eqref{eq:compact:equi}, we deduce from \eqref{eq:P2:uq:cont:1} and \eqref{eq:P2:uq:cont:2} that both
$\Phi_{R,\la}\H_{k\la,\g,\tau}  \Phi_{R,\la}  V_{k,\la} g_{\la,\g}$ and $g_{\la,\g}$ belong to $B^r_{p_0,p_0}$. By Lemma~\ref{W2}, we have 
\[
\|g_{\la,\g}\|_{B^{r}_{p_0,p_0}} \lesssim_{\g,R}  \lambda^{r-r_0} \|V_{k}\|_{B^{r-2+\e}_{\infty,\infty}} \|g_{\la,\g}\|_{B^{r}_{p_0,p_0}}.
\]
By choosing $\e$ and $\la$ to be sufficiently small, we obtain $g_{\la,\g}=0$, which implies $u=0$.
\end{proof}

\subsection{Well-posedness}\label{sec:wp}

In this subsection, we prove Theorems~\ref{thm:main} and \ref{thm:main0}. We begin by establishing well-posedness for the Lippmann--Schwinger equation
\begin{equation}\label{eq:LS:wp}
u(x)= \int_{\mathbb{R}^d} G_{k}(x-y) V_{k} u(y) dy -  \int_{\mathbb{R}^d} G_{k}(x-y) g(y) dy.
\end{equation}

\begin{Lemma}\label{lem:wp:LP:1}
Let $r \in (0,1)$, $p_0\in [1,d/(d-2+2r))$, $\e \in (0,\infty)$, and $\eta \in (\min\{(d+1)/2-d/p_0',1/2\},\infty)$. Suppose that $V_k\in B^{-r+\e}_{p_0',p_0'}$ is compactly supported and that $g\in B^{r-2+d/p_0'}_{(2p_0)',(2p_0)'}(\jpx^\eta)$.
Then the Lippmann--Schwinger equation \eqref{eq:LS:wp} has a unique solution $u\in B^r_{2p_0,2p_0}(\jpx^{-\eta})$. 
\end{Lemma}

\begin{proof}
We first show that the homogeneous Lippmann--Schwinger equation \eqref{eq:LS:wp}, corresponding to $g=0$, has only the trivial
solution. By \eqref{eq:LS:wp}, $u$ is the convolution of the outgoing Green's function $G_k$ with the compactly supported distribution
$V_k u$. Hence $u$ satisfies the Sommerfeld radiation condition. Applying $\Delta+k^2$ to \eqref{eq:LS:wp} yields
\[
(\Delta+k^2)u=-V_k u
\]
in the sense of distributions. Let $r$ be chosen large enough so that the support of $V_k$ is contained in the ball $B(0,r)$. By Green's formula, 
\[
\Im \int_{|x|=r} u \frac{\partial}{\partial{\nu}}\bar{u} ds= \Im \int_{|x|\leq r} \left( |\nabla u|^2 + (V_k-k^2)|u|^2 \right)dx = \int_{|x|\leq r} |u|^2 \Im V_k dx \geq 0.
\]
Hence, 
\[
\int_{|x|=r} \left( \left| \frac{\partial}{\partial \nu} u\right|^2 + k^2 |u|^2 \right) ds = \int_{|x|=r} \left| \frac{\partial}{\partial \nu} u - {\rm i}k u \right|^2 ds - 2k \Im \int_{|x|=r} u \frac{\partial}{\partial{\nu}}\bar{u} ds
\]
Since $u$ satisfies the radiation condition, we obtain 
\[
\lim_{r \to \infty} \int_{|x|=r} \left( \left| \frac{\partial}{\partial \nu} u\right|^2 + k^2 |u|^2 \right) ds \leq  \lim_{r \to \infty} \int_{|x|=r} \left| \frac{\partial}{\partial \nu} u - {\rm i}k u \right|^2 ds =0.
\]
Rellich's lemma implies that $u$ vanishes outside a sufficiently large ball. By compact support and \eqref{eq:compact:equi}, we have $u\in B^r_{2p_0,2p_0}$. Proposition~\ref{P:uq:cont:1} then yields $u= 0$.

Assume that $V_k$ is compactly supported in $B(0,R)$, and let $\Phi_R\in C_c^\infty(\mathbb{R}^d)$ satisfy $\Phi_R=1$ on $B(0,R)$.
Suppose that $v\in B^r_{2p_0,2p_0}$ solves 
\begin{equation}\label{eq:LS:wp:1}
 v = \Phi_R \H^+_k \Phi_R V_k v. 
\end{equation}
Since $\Phi_R=1$ on the support of $V_k$, we have
\[
\H^{+}_k V_k  (\H^+_k V_k v)= \H^{+}_k V_k\, \Phi_R  \H^+_k \Phi_R V_k\, v = \H^{+}_k V_k v. 
\]
Thus $\H_k^+(V_kv)$ solves the homogeneous Lippmann--Schwinger equation. By the uniqueness result proved above, $\H_k^+(V_kv)=0$.
Applying $\Delta+k^2$ gives $V_kv=0$. Substituting this into \eqref{eq:LS:wp:1} yields $v=0$. Hence the homogeneous equation
\eqref{eq:LS:wp:1} has only the trivial solution. 

Let $\eta \in (\max\{(d+1)/2-d/p_0',1/2\},\infty)$ and $g\in B^r_{(2p_0)',(2p_0)'}(\jpx^\eta)$. By Theorem~\ref{Hsg},
\[
\H^+_k g \in B^{r}_{2p_0,2p_0}(\jpx^{-\eta}).
\]
Localizing with $\Phi_R$ and arguing as in the proof of \eqref{eq:compact:equi}, we obtain
\[
\Phi_R\H^+_k g \in B^{r}_{2p_0,2p_0}.
\]
By Rellich's compactness lemma, the operator $\Phi_R\H_k^+\Phi_R V_k$ is compact on $B^r_{2p_0,2p_0}$.
Since the corresponding homogeneous equation has only the trivial solution, the Fredholm alternative gives a unique solution
$v\in B^r_{2p_0,2p_0}$ to the inhomogeneous equation
\[
v = \Phi_R \H^+_k \Phi_R V_k v - \Phi_R \H^+_k g. 
\]

Set $u:=\mathcal{H}_k^+(V_kv)-\mathcal{H}_k^+g$. By \eqref{eq:para:estimate}, Theorem~\ref{Hsg}, and the assumption
$g\in B^{r-2+d/p_0'}_{(2p_0)',(2p_0)'}(\jpx^\eta)$, we have $u\in B^r_{2p_0,2p_0}(\jpx^{-\eta})$ for the given
$\eta>\min\{(d+1)/2-d/p_0',\,1/2\}$. Since $\Phi_R=1$ on the support of $V_k$, the equation for $v$ gives 
\[
\Phi_{R}u= \Phi_{R} \mathcal{H}_{k}^+ V_{k} v - \Phi_{R} \mathcal{H}_{k}^+ g = \Phi_{R} \mathcal{H}_{k}^+ V_{k} v - \Phi_{R} \mathcal{H}_{k}^+g =v.
\]
Consequently, 
\begin{align*}
\mathcal{H}_{k}^+ V_{k} u -  \mathcal{H}_{k}^+ g= \mathcal{H}_{k}^+ V_{k} \Phi_{R}\,  u - \mathcal{H}_{k}^+ g= \mathcal{H}_{k}^+ V_{k} v - \mathcal{H}_{k}^+ g = u. 
\end{align*}
Thus $u$ solves the Lippmann--Schwinger equation \eqref{eq:LS:wp}. By the uniqueness result above, it is the unique solution in
$B^r_{2p_0,2p_0}(\jpx^{-\eta})$.
\end{proof}

\begin{Lemma}\label{lem:wp:LP:2}
Let $p_0\in [1,\infty]$, $\e \in (0,\infty)$, $\eta \in ((d+1)/2,\infty)$, and $r \in (r_0,2)$, where 
\[
r_0= (d+1)/2- (d-1)/\max\{p_0,p_0'\}.
\]
 Suppose that $V_k\in B^{r-2+\e}_{\infty,\infty}$ is compactly supported and that $g\in B^{r-2}_{p_0,p_0}(\jpx^\eta)$. Then the Lippmann--Schwinger equation \eqref{eq:LS:wp} has a unique solution $u\in B^r_{p_0,p_0}(\jpx^{-\eta})$.
\end{Lemma}

\begin{proof}
Let $\eta>(d+1)/2$ and $g\in B^{r-2}_{p_0,p_0}(\jpx^\eta)$. By Theorem~\ref{Hsg}, 
\[
\H^+_k g \in B^{r}_{p_0,p_0}(\jpx^{-\eta}).
\]
The argument used in the proof of Lemma~\ref{lem:wp:LP:1} then establishes well-posedness for the Lippmann--Schwinger equation
\eqref{eq:LS:wp} in $B^r_{p_0,p_0}(\jpx^{-\eta})$.
\end{proof}

Finally, we show that the Lippmann--Schwinger equation is equivalent to the Helmholtz equation subject to the Sommerfeld radiation condition. Our argument follows \cite[Theorems~3 and~4]{LPS08} and \cite[Section~8.3]{CK98}. We include the proof for completeness.

Let $u$ be the unique solution of the rescaled Lippmann--Schwinger equation. The outgoing Green function $G_k(x-y)$ satisfies the
Sommerfeld radiation condition and is analytic in $x$ away from $x=y$. For $y$ in the support of $V_k$ or $g$, the kernel is therefore
analytic in $x$ outside the union of these supports. A direct estimate of $(\partial_{|x|}-{\rm i}k)G_k(x-y)$ in the integral representation shows that $u$ also satisfies the Sommerfeld radiation condition. Moreover, applying $\Delta+k^2$ to \eqref{eq:LS:wp} shows that $u$ satisfies the Helmholtz equation in the sense of distributions.

To prove uniqueness, suppose that $u$ satisfies the homogeneous Helmholtz equation, i.e., $g=0$, and the Sommerfeld radiation condition. We show that $u\equiv 0$.

The homogeneous Helmholtz equation gives
\[
\int_{\mathbb{R}^d} G_{k}(x-y) (\Delta+k^2)u(y) dy = -\int_{\mathbb{R}^d} G_{k}(x-y) V_{k}(y) u(y) dy.
\]
Since $(\Delta+k^2)u=-V_ku$ and $V_k$ is supported in $B(0,R)$, we may restrict the first integral to $B(0,r)$ for any $r>R$:
\[
\int_{\mathbb{R}^d} G_{k}(x-y) (\Delta+k^2)u(y) dy = \int_{B(0,r)} G_{k}(x-y) (\Delta+k^2)u(y)dy.
\] 
Under the assumptions of either Lemma~\ref{W2} or Lemma~\ref{W3}, Green's representation formula (cf.~\cite[Theorem~2.1]{CK98}) gives,
for $r>R$,
\begin{align*}
-u(x) + \int_{\partial B(0,r)} [G_{k} (x-y) \partial_{\nu} u(y) - \partial_{\nu}  G_{k} (x-y) u(y)] dS(y)\\
 =\int_{B(0,r)} G_{k}(x-y) (\Delta+k^2)u(y) dy,
\end{align*}
in the sense of distributions in $x\in B(0,r)$, where $\nu$ is the outward unit normal to $\partial B(0,r)$ at $y$. For each fixed $x$, the Sommerfeld radiation condition for $u(y)$ and $G_k(x-y)$ as $|y|\to\infty$ implies that
\[
\lim_{r \to \infty} \int_{\partial B(0,r)} [G_{k} (x-y) \partial_{\nu} u(y) - \partial_{\nu}  G_{k} (x-y) u(y)]dy=0. 
\]

Letting $r\to\infty$, we conclude that $u$ satisfies the homogeneous Lippmann--Schwinger equation \eqref{eq:LS:wp} with $g=0$.
By either Lemma~\ref{W2} or Lemma~\ref{W3}, this equation has only the trivial solution. Hence $u\equiv 0$,
which completes the proof of uniqueness.

\section{Conclusion}

In this paper, we establish well-posedness for the Helmholtz equation with rough coefficients. The results are sharp in the sense that they reach the minimal regularity threshold for defining the product of the coefficient and the solution without renormalization. We also derive estimates for a broad class of operators associated with the Helmholtz equation in an $L^p$-based framework, with explicit dependence on the wavenumber.

The analysis assumes that the coefficients are compactly supported. This assumption covers many physically relevant models and enables precise resolvent estimates. Extending the theory to coefficients with unbounded support remains an open problem. A specific direction for future work is to determine how slowly decaying rough coefficients affect well-posedness and the wavenumber dependence of the resolvent estimates.

\section*{Acknowledgement}

The work of PL is supported by the National Key R\&D Program of China (2024YFA1012300). The work of YZ is supported by Young Scientists Fund of NSFC (C) (Certificate number:12501189).

\section*{Data avaibility}
This study did not generate any new data. All analyzed data are from previously published sources cited in the reference list.

\section*{Author Contributions}
The authors made equal contributions to the paper. 

\section*{Statements and declarations}
The authors declare that they have no conflict of interest.

\appendix
\renewcommand{\appendixname}{Appendix~\Alph{section}}

\section{Bony's paraproduct}\label{sec:para}

In this section, we present paraproduct estimates in rescaled weighted Besov spaces.

\begin{Lemma}\label{Pre1:n}
Let $\lambda \in (0,\infty)$ and $p \in [1,\infty]$. For any $f \in \mathcal{S}$ and any weight $\rho \in W(\eta)$, the following estimate holds:
 \[
 \|\Delta_{j} f\|_{L^{p}(\rho,\lambda)} \lesssim_{\eta} \max\{\la^{\eta},1\}  \|f\|_{L^p(\rho,\lambda)}, 
 \]
where the implicit constant depends only on $\rho$ and is independent of $\lambda$ and $j$.
\end{Lemma}
 
 \begin{proof}
Since $\rho \in W(\eta)$, we have $\rho(\lambda x)\lesssim_{}\langle\lambda(x-y) \rangle^\eta\rho(\lambda y)$ for any $x,y\in\mathbb R^d$. Therefore, 
 \[
 \|\Delta_{j} f\|_{L^{p}(\rho,\lambda)} \lesssim \left( \int_{\mathbb{R}^d}  \left|\int_{\mathbb{R}^d} \mathcal{F}^{-1}(\varphi_{j})(x-y) \langle \lambda (x-y)\rangle^{\eta} f(y)\rho(\lambda y) dy\right|^p  dx\right)^{\frac{1}{p}}. 
 \]
Applying Young's convolution inequality yields
 \[
  \|\Delta_{j} f\|_{L^{p}(\rho,\lambda)}  \lesssim \|\mathcal{F}^{-1}(\varphi_{j})\|_{L^1(\jpx^{\eta},\lambda)} \|f\|_{L^p(\rho,\lambda)}. 
 \]
 Recalling that $\mathcal{F}^{-1}(\varphi_j)=2^{jd}\check{\varphi}(2^j\cdot)$, we compute
 \[
 \|\mathcal{F}^{-1}(\varphi_{j})\|_{L^1(\rho,\lambda)} = \int_{\mathbb{R}^d} 2^{jd}\ |\check{\varphi}(2^jx)|\langle \lambda x\rangle^{\eta}  dx = \int_{\mathbb{R}^d}  |\check{\varphi}( x)| \langle 2^{-j}\lambda x \rangle^{\eta}  dx. 
 \]
 Since $\varphi\in\mathcal{S}$, we have $\check{\varphi}\in\mathcal{S}$, and using $j \geq -1$, it follows that
 \[
  \|\mathcal{F}^{-1}(\varphi_{j})\|_{L^1(\rho,\lambda)} \lesssim  \int_{\mathbb{R}^d}  |\check{\varphi}( x)| \langle 2^{-j} \la x\rangle^{\eta} dx \lesssim_{\eta} \max\{ \la^{\eta},1\}. 
  \]
Combining the above estimates completes the proof.
 \end{proof}

\begin{Lemma}\label{Pre3:n}
Let $\mathcal{C}'$ be an annulus in $\mathbb{R}^d$, $\lambda \in (0,\infty)$, $p,q \in [1,\infty]$, and $r \in \mathbb{R}$.
Let $\{u_j\}_{j\in\mathbb{N}}$ be a sequence of smooth functions such that ${\rm supp}\,\hat{u}_{j} \subseteq 2^j \mathcal{C}'$. Then, for any weight $\rho \in W(\eta)$, the following estimate holds:
\[
\|u\|_{B^{r}_{p,q}(\rho,\lambda)} \lesssim   \max\{\la^{\eta},1\} \bigg(\sum_{j \in \mathbb{N}} 2^{jrq} \|u_{j}\|_{L^p(\rho,\lambda)}^q \bigg)^{1/q},\quad  u:=\sum_{j \in \mathbb{N}} u_{j},
\]
where the implicit constant is independent of $\lambda$.
\end{Lemma}

\begin{proof}
The argument follows the same lines as in \cite[Lemma~2.69]{BCD11}. By the spectral localization assumption, there exists an integer $N_0\ge 1$, independent of $\lambda$, such that $\Delta_{j'} u_j=0$ for $|j'-j|>N_0$.  By Minkowski's inequality, we have
\[
 2^{jr} \|\Delta_{j} u\|_{L^p(\rho,\lambda)} \lesssim \sum_{|j'-j| <N_0} 2^{(j-j')r} 2^{j'r} \|\Delta_{j}  u_{j'}\|_{L^p(\rho,\lambda)}.
\]
Applying Lemma \ref{Pre1:n} yields 
\[
 2^{jr} \|\Delta_{j} u\|_{L^p(\rho,\lambda)}  \lesssim    \max\{\la^{\eta},1\} \sum_{|j'-j| <N_0}  2^{j'r}\|u_{j'}\|_{L^p(\rho,\lambda)}.
\]
Taking the $\ell^q$ norm in $j$ and using Young’s convolution inequality for sequences completes the proof.
\end{proof}

\begin{Lemma}\label{Pre4:n}
Let $\mathcal{B}$ be a ball in $\mathbb{R}^d$, $\lambda \in (0,\infty)$, $p,q \in [1, \infty]$, and $r>0$. Let $\{u_{j}\}_{j \in \mathbb{N}}$ be a sequence of smooth functions such that ${\rm supp}(\hat{u}_{j}) \subseteq 2^j\mathcal{B}$. Then, for any weight $\rho \in W(\eta)$, the following estimate holds:
\[
\|u\|_{B^{r}_{p,q}(\rho,\lambda)} \lesssim   \max\{\la^{\eta},1\}  \bigg(\sum_{j \in \mathbb{N}} 2^{jrq} \|u_{j}\|_{L^p(\rho,\lambda)}^q \bigg)^{1/q},\quad u:=\sum_{j \in \mathbb{N}} u_{j},
\]
where the implicit constant is independent of  $\lambda$.
\end{Lemma}

\begin{proof}
The proof follows the argument of \cite[Lemma~2.84]{BCD11}. As in the proof of Lemma~\ref{Pre3:n}, there exists an integer $N_0\geq 1$, independent of $\lambda$, such that $\Delta_{j'} u_j=0$ for $j'-j>N_0$. By Lemma~\ref{Pre1:n}, we obtain
\[
2^{j'r}\|\Delta_{j'} u\|_{L^p(\rho,\lambda)} \lesssim 2^{j'r} \sum_{j'-j<N_0} \|\Delta_{j'} u_j\|_{L^p(\rho,\lambda)} \lesssim   \max\{\la^{\eta},1\} \sum_{j>j'-N_0}2^{(j'-j)r} 2^{jr}\| u_j\|_{L^p(\rho,\lambda)}.
\]
Since $r>0$, the factor $2^{(j'-j)r}$ decays exponentially as $j-j'\to\infty$. Applying Young's convolution inequality for sequences concludes the proof.
\end{proof}

The following lemma collects the basic estimates for Bony's paraproduct.

\begin{Lemma}\label{PC1}
Let $\la \in (0,\infty)$, $p,q,p_1,p_2,q_1,q_2 \in [1,\infty]$ satisfy $1/p=1/{p_1} + 1/{p_2}$, $1/q=1/{q_1} + 1/{q_2}$, and let $\rho_1, \rho_2 \in W(\eta)$ be admissible weights.
\begin{itemize}
\item[(1)] If $\alpha>0$, then
\begin{equation}\label{PC1:eq:1}
\|f \prec g\|_{B^{\beta}_{p,q}(\rho_1\rho_2,\lambda)}
\lesssim_{\alpha,\beta,\eta}  \max\{\la^{2\eta},1\}
\|f\|_{B^{\alpha}_{p_1,\infty}(\rho_1,\lambda)}
\|g\|_{B^{\beta}_{p_2,q}(\rho_2,\lambda)}.
\end{equation}

\item[(2)] If $\alpha<0$, then
\begin{equation}\label{PC1:eq:2}
\|f \prec g\|_{B^{\alpha+\beta}_{p,q}(\rho_1\rho_2,\lambda)}
\lesssim_{\alpha,\beta,\eta}  \max\{\la^{2\eta},1\}
\|f\|_{B^{\alpha}_{p_1,q_1}(\rho_1,\lambda)}
\|g\|_{B^{\beta}_{p_2,q_2}(\rho_2,\lambda)}.
\end{equation}

\item[(3)] If $\alpha+\beta>0$, then
\begin{equation}\label{PC1:eq:3}
\|f \circ g\|_{B^{\alpha+\beta}_{p,q}(\rho_1\rho_2,\lambda)}
\lesssim_{\alpha,\beta,\eta}  \max\{\la^{2\eta},1\}
\|f\|_{B^{\alpha}_{p_1,q_1}(\rho_1,\lambda)}
\|g\|_{B^{\beta}_{p_2,q_2}(\rho_2,\lambda)}.
\end{equation}
\end{itemize}
In all cases, the implicit constants are independent of the rescaling parameter $\lambda$.
\end{Lemma}

\begin{proof}
To prove \eqref{PC1:eq:1}, we first note that for each $j\geq -1$, the product $S_{j-1}f \Delta_j g$ has Fourier support contained in $2^{j}\mathcal{C}'$, where $\mathcal{C}'$ is an annulus slightly larger than $\mathcal{C}$. By Lemma~\ref{Pre3:n} and H\"{o}lder’s inequality, we obtain
\begin{align*}
\|f \prec g\|_{B^{\beta}_{p,q}(\rho_1 \rho_2,\lambda)} &\lesssim  \max\{\la^{2\eta},1\} \bigg(\sum_{j \geq -1} 2^{j \beta q} \|S_{j-1} f \Delta_{j} g\|_{L^{p}(\rho_1 \rho_2,\lambda)}^q\bigg)^{1/q}\\
& \lesssim  \max\{\la^{2\eta},1\} \sup_{j \geq -1}\|S_{j-1} f \|_{L^{p_1}(\rho_1,\lambda)} \bigg(\sum_{j \geq -1} 2^{j\beta q} \|\Delta_{j} g\|_{L^{p_2}(\rho_2,\lambda)}^q\bigg)^{1/q}.
\end{align*}
Assume $\alpha>0$. Using Minkowski's inequality and the definition of $S_{j-1}$, we have
\begin{align*}
\|S_{ j-1} f \|_{L^{p_1}(\rho_1,\lambda)} \leq  \sum_{k=-1}^{j-2}\|\Delta_{k} f \|_{L^{p_1}(\rho_1,\lambda)}
\end{align*}
Multiplying and dividing by $2^{k\alpha}$ and summing the resulting geometric series yield
\begin{align*}
 \|S_{ j-1} f \|_{L^{p_1}(\rho_1,\lambda)}\leq \sum_{k=-1}^{j-2}2^{-k\alpha } \sup_{k \geq -1} 2^{k\alpha }\|\Delta_{k-1} f \|_{L^{p_1}(\rho_1,\lambda)}^{}\lesssim \|f\|_{B^{\alpha}_{p_1,\infty}(\rho_1,\lambda)}. 
\end{align*}
Combining the above estimates, we conclude that
\[
\|f \prec g\|_{B^{\beta}_{p,q}(\rho_1 \rho_2,\lambda)} \lesssim  \max\{\la^{2\eta},1\} \|f\|_{B^{\alpha}_{p_1,\infty}(\rho_1,\lambda)} \|g\|_{B^{\beta}_{p_2,q}(\rho_2,\lambda)} .
\]

To prove the \eqref{PC1:eq:2}, applying Lemma \ref{Pre3:n} and H\"{o}lder's inequality, we obtain 
\begin{align*}
\|f \prec g \|_{B^{\alpha+ \beta}_{p,q}(\rho_1 \rho_2,\lambda)} &\lesssim   \max\{\la^{2\eta},1\} \bigg(\sum_{j \geq -1} 2^{j(\alpha+\beta) q} \|S_{j-1} f \Delta_{j} g\|_{L^{p}(\rho_1 \rho_2,\lambda)}^q\bigg)^{1/q}\\
 &\lesssim  \max\{\la^{2\eta},1\} \bigg(\sum_{j \geq -1} 2^{j\alpha q_1}\|S_{j-1} f \|_{L^{p_1}(\rho_1,\lambda)}^{q_1}\bigg)^{1/q_1} \bigg(\sum_{j \geq -1} 2^{j \beta q_2} \|\Delta_{j} g\|_{L^{p_2}(\rho_2,\lambda)}^{q_2}\bigg)^{1/{q_2}}.
 \end{align*}
Next, using Minkowski's inequality and the definition of $S_{j-1}$, we estimate
\begin{align*}
 2^{j\alpha}\|S_{ j-1} f \|_{L^{p_1}(\rho_1,\lambda)} &\leq  \sum_{k=-1}^{j-2}2^{(j-k)\alpha} 2^{k\alpha}\|\Delta_{k} f \|_{L^{p_1}(\rho_1,\lambda)}\\
 &\leq \sum_{k\geq -1} 2^{(j-k)\alpha}\mathbbm{1}_{\{j-k \geq 2\}}\ 2^{k\alpha}\|\Delta_{k} f \|_{L^{p_1}(\rho_1,\lambda)}.
\end{align*}
Since $\alpha<0$, the sequence $\{2^{\ell\alpha}\}_{\ell\ge 2}$ belongs to $\ell^1(\mathbb{N})$. Therefore, it follows from Young's convolution inequality for sequences that we get 
\begin{align*}
  \bigg(\sum_{j \geq -1} 2^{j\alpha q_1}\|S_{j-1} f \|_{L^{p_1}(\rho_1,\lambda)}^{q_1}\bigg)^{1/q_1}\lesssim \|f\|_{B^{\alpha}_{p_1,q_1}(\rho_1,\lambda)}. 
\end{align*}
Combining the above estimates yields 
\[
 \|f \prec g\|_{B^{\alpha+\beta}_{p,q}(\rho_1\rho_2,\lambda)} \lesssim
 \max\{\la^{2\eta},1\} \|f\|_{B^{\alpha}_{p_1,q_1}(\rho_1,\lambda)}
\|g\|_{B^{\beta}_{p_2,q_2}(\rho_2,\lambda)}.
\]

To prove \eqref{PC1:eq:3}, we first observe that the resonant product can be written as
\[
f\circ g = \sum_{i \geq -1} \Delta_{i} f  \Delta_{i-1} g + \Delta_{i} f  \Delta_{i} g + \Delta_{i} f  \Delta_{i+1} g=:  \sum_{i \geq -1} A_{i}(f,g).
\]
We also note that for each $i \geq -1$, the term $A_i(f,g)$ has Fourier support contained in $2^{i}\mathcal{B}'$, where $\mathcal{B}'$ is a ball slightly larger than $\mathcal{B}$. By Lemma~\ref{Pre4:n}, we have, for $\alpha+\beta>0$,
\begin{align*}
\|f \circ g\|_{B^{\alpha+ \beta}_{p,q}(\rho_1 \rho_2,\lambda)} \lesssim   \max\{\la^{2\eta},1\} \bigg(\sum_{j \geq -1}  2^{j(\alpha+\beta) q} \|A_{j}(f, g)\|_{L^{p}(\rho_1 \rho_2,\lambda)}^q\bigg)^{1/q}.
 \end{align*}
Using Minkowski's inequality and H\"{o}lder’s inequality, we estimate
\begin{align*}
2^{j(\alpha+\beta)q} \|A_{j}(f, g)\|_{L^{p}(\rho_1 \rho_2,\lambda)}^q &\lesssim 2^{j(\alpha+\beta)q} \bigg(\sum_{|\nu| \leq 1} \|\Delta_{j} f \cdot \Delta_{j-\nu} g\|_{L^{p}(\rho_1 \rho_2,\lambda)}\bigg)^q \\
&\lesssim \bigg(\sum_{|\nu|\leq 1}  2^{j\alpha}\|\Delta_{j} f \|_{L^{p_1}(\rho_1,\lambda)}  2^{(j-\nu)\beta} \|\Delta_{j-\nu} g\|_{L^{p_2}(\rho_2,\lambda)}\bigg)^q\\
&\lesssim_q \sum_{|\nu|\leq 1} \bigg(  2^{j\alpha}\|\Delta_{j} f \|_{L^{p_1}(\rho_1,\lambda)}  2^{(j-\nu)\beta} \|\Delta_{j-\nu} g\|_{L^{p_2}(\rho_2,\lambda)}\bigg)^q.
\end{align*}
Summing over $j\geq -1$ and applying H\"{o}lder’s inequality for sequences, we obtain
\begin{align*}
\|f \circ g\|_{B^{\alpha+ \beta}_{p,q}(\rho_1 \rho_2,\lambda)}& \lesssim  \max\{\la^{2\eta},1\} \sum_{|\nu|\leq 1}  \Bigg(\sum_{j \geq -1}  \bigg(  2^{j\alpha}\|\Delta_{j} f \|_{L^{p_1}(\rho_1,\lambda)}  2^{(j-\nu)\beta} \|\Delta_{j-\nu} g\|_{L^{p_2}(\rho_2,\lambda)}\bigg)^q\Bigg)^{1/q}\\
&\lesssim  \max\{\la^{2\eta},1\}\sum_{|\nu|\leq 1}  \bigg(\sum_{j \geq -1}   2^{j\alpha q_1}\|\Delta_{j} f \|_{L^{p_1}(\rho_1,\lambda)}^{q_1}\bigg)^{1/q_1} \bigg(\sum_{j \geq -1}  2^{(j-\nu)\beta q_2} \|\Delta_{j-\nu} g\|_{L^{p_2}(\rho_2,\lambda)}^{q_2}\bigg)^{1/q_2}.
\end{align*}
Since shifts by $\nu\in\{-1,0,1\}$ do not affect the Besov norms, we conclude 
\[
 \|f \circ g\|_{B^{\alpha+\beta}_{p,q}(\rho_1\rho_2,\lambda)}
\lesssim  \max\{\la^{2\eta},1\} \|f\|_{B^{\alpha}_{p_1,q_1}(\rho_1,\lambda)} \|g\|_{B^{\beta}_{p_2,q_2}(\rho_2,\lambda)}, 
\]
which completes the proof. 
\end{proof}

\section{Proof of Lemma \ref{R2:n}}\label{p:R1}

We first assume  that $\lambda=2^{-N_0}$ for some integer $N_0\geq 0$. We start with the case $k\geq 0$. By definition of the Littlewood--Paley projection $\Delta_{k}$, we have
\begin{align}\label{R1:eq:0}
\|\Delta_{k} (f)_{\lambda}\|_{L^p(\rho,\lambda)}= \left\|\int_{\mathbb{R}^d} \check{\varphi}_{k}(x-y) f(\lambda y) dy \right\|_{L^p(\rho, \lambda)}= \lambda^{-d}\left\|  \int_{\mathbb{R}^d}  \check{\varphi}_{k}\left(\lambda^{-1}(\lambda x -y)\right) f(y) dy  \right\|_{L^p(\rho, \lambda)}. 
\end{align}
 Then, there holds
\[
\mathcal{F}(\check{\varphi}_{k}(2^{N_0} \cdot))=2^{-N_0d} \varphi_{k}(2^{-N_0} \xi)= 2^{-N_0d} \varphi(2^{-(k+N_0)} \xi)= 2^{-N_0d} \varphi_{k}(2^{-N_0} \xi)= 2^{-N_0d} \varphi_{k+N_0} (\xi),
\]
which implies 
\[
\check{\varphi}_{k}(\la^{-1} x)=\check{\varphi}_{k}(2^{N_0} x)= 2^{-N_0d} \check{\varphi}_{k+N_0}(x)= \lambda^d \check{\varphi}_{k+N_0}(x).
\]
It is clear to note that 
\[
 \Delta_{k} (f)_{\lambda}(x)= \int_{\mathbb{R}^d}  \check{\varphi}_{k+N_0}\left(\lambda x -y\right) f(y) dy =\Delta_{k+N_0} f(\lambda x).
\]
Therefore,
\begin{align*}
\sum_{k\geq 0} 2^{krq}\|\Delta_{k} (f)_{\lambda}( \cdot)\|_{L^p(\rho, \lambda)}^q &= \sum_{k \geq 0} 2^{krq} \left\| \Delta_{k+N_0} f(\lambda \cdot) \right\|_{L^p(\rho, \lambda)}^q\\
&= \sum_{k\geq 0} \lambda^{rq} 2^{(k+N_0)rq} \|\Delta_{k+N_0} f(\lambda \cdot)\|_{L^p(\rho, \lambda)}^q \\
&= \lambda^{rq} \sum_{k \geq 0}  2^{(k+N_0)rq} \lambda^{-dq/p} \|\Delta_{k+N_0} f\|_{L^p(\rho)}^q.
\end{align*}
Combining the above estimates yields 
\begin{equation}\label{R1:eq:1}
\sum_{k=0}^{\infty} 2^{krq}\|\Delta_{k} (f)_{\lambda}( \cdot)\|_{L^p(\rho,{\lambda})}^q =  \lambda^{-dq/p+rq}  \sum_{k=N_0}^{\infty}  2^{krq} \|\Delta_{k} f( x)\|_{L^p(\rho)}^q.
\end{equation}

For the case $k=-1$, we have
\begin{align*}
\|\Delta_{-1} (f)_{\lambda}\|_{L^p(\rho,\lambda)} =
\left\| \lambda^{-d} \int_{\mathbb{R}^d}  \check{\chi}\left(\lambda^{-1}(\lambda x -y)\right) f(y) dy  \right\|_{L^p(\rho,\lambda)}. 
\end{align*}
Assume that $\lambda =2^{-N_0} $ with $N_0 \geq 0$. Then $\mathcal{F}(\check{\chi}(2^{N_0}\cdot))= 2^{-N_0d} \chi_{N_0}(\cdot)$,
which implies 
\[
\check{\chi}(2^{N_0} x)= 2^{-N_0d}\ \mathcal{F}^{-1}(\chi_{N_0})(x).
\]
Hence, 
\begin{align*}
\|\Delta_{-1} (f)_{\lambda}\|_{L^p(\rho, \lambda)} = \left\|  \int_{\mathbb{R}^d}  \mathcal{F}^{-1}(\chi_{N_0})\left(\lambda x -y\right) f(y) dy  \right\|_{L^p(\rho, \lambda)}.
\end{align*}
By Minkowski's inequality and H\"{o}lder’s inequality for sequences, we obtain
\begin{align*}
\|\Delta_{-1} (f)_{\lambda}\|_{L^p(\rho,{\lambda})} &\leq  \sum_{k=-1}^{N_0-1} \|\Delta_{k} f(\lambda x)\|_{L^p(\rho,{\lambda})} \leq  \sum_{k=-1}^{N_0-1}\lambda^{-d/p} \|\Delta_{k} f\|_{L^p(\rho)}  \\
&\leq\ \lambda^{-d/p}  \left(\sum_{k=-1}^{N_0-1} 2^{krq}\|\Delta_{k} f(x)\|_{L^p(\rho)}^q\right)^{1/q}  \left(\sum_{k=-1}^{N_0-1} 2^{-krq/(q-1)}\right)^{(q-1)/q}.
\end{align*}
We now distinguish two cases to estimate $\|\Delta_{-1} (f)_{\lambda}\|_{L^p(\rho, \lambda)}$ more precisely.
 
When $r < 0$, we have 
\[
\sum_{k=-1}^{N_0-1} 2^{-rkq/(q-1)} \leq \sum_{k=-1}^{\infty} 2^{-krq/(q-1)} \lesssim_{q,r} 1.
\]
Therefore, combining this with the previous estimate yields
\begin{equation}\label{R1:eq:2.1}
2^{-rq}\|\Delta_{-1} (f)_{\lambda}\|_{L^p(\rho, \lambda)}^q \lesssim_{q,r} \lambda^{-dq/p}\sum_{k=-1}^{N_0-1} 2^{rkq}\|\Delta_{k} f(x)\|_{L^p(\rho)}^q.
\end{equation}

When $r<0 $, we have
\[
\sum_{k=-1}^{N_0-1} 2^{-rkq/(q-1)} \lesssim_{q,r}  2^{-N_0 r q/(q-1)} \lesssim_{q,r} \lambda^{rq/(q-1)}.
\]
Hence,  
\begin{equation}\label{R1:eq:2.2}
2^{-rq}\|\Delta_{-1} (f)_{\lambda}\|_{L^p(\rho, \lambda)}^q \lesssim_{q,r} \lambda^{-dq/p}  \lambda^{rq} \sum_{k=-1}^{N_0-1} 2^{rkq}\|\Delta_{k} f(x)\|_{L^p(\rho)}^q .
\end{equation}

When $r=0$, we have
\[
\sum_{k=-1}^{N_0-1} 2^{-rkq/(q-1)} \lesssim_{q,r}  N_0 \lesssim_{q,r} 1+|\ln(\la)|.
\]
Hence,  
\begin{equation}\label{R1:eq:2.3}
2^{-rq}\|\Delta_{-1} (f)_{\lambda}\|_{L^p(\rho, \lambda)}^q \lesssim_{q,r} \lambda^{-dq/p} (1+|\ln(\la)|)^{q-1} \sum_{k=-1}^{N_0-1} 2^{rkq}\|\Delta_{k} f(x)\|_{L^p(\rho)}^q .
\end{equation}

Combining \eqref{R1:eq:1}--\eqref{R1:eq:2.3}, we conclude
\begin{equation}\label{eq:main:1}
\|(f)_{\lambda}\|_{B^{r}_{p,q}(\rho, \lambda)} \lesssim_{q,r} \lambda^{-d/p}C_{\la,r,q} \|f\|_{B^{r}_{p,q}(\rho)}\ ,
\end{equation}
where 
\[
\begin{aligned}
C_{\la,r,q}:= 
\begin{cases}
\la^r,\quad &r<0,\\
(1+\ln(\la))^{1/q'},\quad &r=0,\\
1,\quad &r>0.
\end{cases}
\end{aligned}
\]

Next, we assume that $\lambda = 2^{N_0}$ for some integer $N_0 \geq 0$. By the argument in the proof of \eqref{R1:eq:0}, we have
\begin{equation}\label{R2:eq:1}
\sum_{k=N_0}^{\infty} 2^{krq}\|\Delta_{k} f(\lambda \cdot)\|_{L^p(\rho,{\lambda})}^q =  \lambda^{-dq/p+rq}  \sum_{k=0}^{\infty}  2^{krq} \|\Delta_{k} f( x)\|_{L^p(\rho)}^q.
\end{equation}
It remains to estimate the contribution from the low-frequency terms $k=-1,0,1,...,N_0-1$. Similar to the proof of \eqref{R1:eq:1}, we compute
\[
\mathcal{F}(\check{\varphi}_{k}(\lambda^{-1} \cdot))=\mathcal{F}(\check{\varphi}_{k}(2^{-N_0} \cdot))=2^{N_0d} \varphi_{k}(2^{N_0} \xi)= 2^{N_0d} \varphi(2^{N_0-k} \xi),
\]
which implies that 
\begin{align}\label{R2:eq:1.1}
\check{\varphi}_{k}(\lambda^{-1} x)= 2^{N_0d} \mathcal{F}^{-1}(\varphi)_{2^{-k+N_0}}(x).
\end{align}

For $k=0,1,\ldots,N_0-1$, we apply \eqref{R1:eq:0} and \eqref{R2:eq:1.1} to obtain
\begin{align*}
\|\Delta_{k} (f)_{\lambda}\|_{L^p(\rho,\lambda)}&=\left\|  \int_{\mathbb{R}^d}   \IF(\varphi)_{2^{-k+N_0}}(\lambda x -y) f(y) dy  \right\|_{L^p(\rho, \lambda)}\\
&=\lambda^{-d/p} \left\|  \int_{\mathbb{R}^d}   \IF(\varphi)_{2^{-k+N_0}}( x -y) f(y) dy  \right\|_{L^p(\rho)}\\
&= \lambda^{-d/p} \left\| \IF\big((\varphi)_{2^{-k+N_0}} \hat{f}\big)   \right\|_{L^p(\rho)}. 
\end{align*}
Since $2^{-k+N_0}\geq 2$ for $k=0,1,\ldots,N_0-1$, we have $(\varphi)_{2^{-k+N_0}}=(\varphi)_{2^{-k+N_0}} \chi_0$, and hence 
\[
\|\Delta_{k} (f)_{\lambda}\|_{L^p(\rho,\lambda)}=\lambda^{-d/p} \left\| \IF\big((\varphi)_{2^{-k+N_0}}\ \chi_0 \hat{f}\big)   \right\|_{L^p(\rho)}.
\]
Using the fact that $\rho$ is of type $W(\eta)$, we obtain
\begin{align*}
\|\Delta_{k} (f)_{\lambda}\|_{L^p(\rho,\lambda)} \lesssim& \lambda^{-d/p} \left\|  \int_{\mathbb{R}^d}  \IF(\varphi)_{2^{-k+N_0}}( x -y) \langle x-y \rangle^{\eta}\ S_1f(y) \rho(y) dy  \right\|_{L^p}.
\end{align*}
By Young's convolution inequality, we arrive at
\begin{align*}
\|\Delta_{k} (f)_{\lambda}\|_{L^p(\rho,\lambda)}\lesssim& \lambda^{-d/p} \|\IF(\varphi)_{2^{-k+N_0}}( x ) \langle x \rangle^{\eta} \|_{L^1} \|f\|_{B^{r}_{p,q}(\r)}.
\end{align*}

Let $C_0:=\{x:\ |x|\leq 1\}$ and $C_j:=\{x:\ 2^{j-1}\leq |x|\leq 2^{j}\}$ for $j\geq 1$. By H\"{o}lder's inequality, we have
\[
\begin{aligned}
 \|\IF(\varphi)_{2^{-k+N_0}}( x ) \langle x \rangle^{\eta} \|_{L^1}&=\sum_{j\geq 0}\int_{C_j} |\IF(\varphi)_{2^{-k+N_0}}( x ) |\langle x \rangle^{\eta} dx\\
 &\leq\sum_{j\geq 0} \bigg(\int_{C_j} |\F(\varphi)_{2^{-k+N_0}}( x ) |^2 dx \bigg)^{1/2} \bigg(\int_{C_j} \langle x \rangle^{2\eta} dx\bigg)^{1/2}.
 \end{aligned}
\]
By a change of variables, it is clear to note that
\[\int_{C_j} \langle x \rangle^{2\eta} dx \lesssim  \int_{0\leq|x'|\leq 1} \langle 2^j x' \rangle^{2\eta} d(2^j x') \lesssim  2^{2j\eta+jd} \int_{0\leq|x'|\leq 1} \langle x' \rangle^{2\eta} dx'.
\]
Hence,
\[
\begin{aligned}
\|\mathcal{F}^{-1}(\varphi)_{2^{-k+N_0}}( x ) \langle x \rangle^{\eta} \|_{L^1} &\lesssim \sum_{j\geq 0} 2^{(k-N_0)d}\bigg(\int_{C_j} |\hat\varphi(2^{k-N_0} x ) |^2 dx \bigg)^{1/2} 2^{j(d/2+\eta)} \\
& \lesssim \sum_{j\geq 0} 2^{(N_0-k)\eta} \bigg(\int_{2^{k-N_0} C_j} |\hat{\varphi}( x ) |^2 dx \bigg)^{1/2} 2^{(j+k-N_0)(d/2+\eta)}\\
 & \lesssim 2^{(N_0-k)\eta} \Bigg(\sum_{j\geq 0}\int_{2^{k-N_0} C_j}   (1+x^2)^{d/2+\eta+\epsilon}|\hat{\varphi}( x ) |^2 dx\Bigg)^{1/2}\\
&\quad \times \Bigg(\sum_{j=0}^{N_0-k}2^{(j+k-N_0)(d+2\eta)} + \sum_{j> N_0-k}  2^{-2(j+k-N_0)\epsilon} \Bigg)^{1/2}\\
  & \lesssim 2^{(N_0-k)\eta} \left(\int_{\mathbb{R}^d}   (1+x^2)^{d/2+\eta+\e}|\hat{\varphi}( x ) |^2 dx\right)^{1/2} ,
\end{aligned}
\]
for any $\epsilon \in (0,1)$. Therefore, for $k=0,1,\ldots,N_0-1$, we have proved that
\begin{equation}\label{R2:eq:2}
2^{kr} \|\Delta_{k} (f)_{\lambda}\|_{L^p(\rho,\lambda)} \lesssim 2^{k(r-\eta)}\lambda^{\eta-d/p}\|\varphi\|_{H^{d/2+\eta+\epsilon}} \|f\|_{B^{r}_{p,q}(\rho)},\quad k=0,1,..,N_0-1.
\end{equation}

Moreover, by a similar argument, one can show that the estimate \eqref{R2:eq:2} also holds for $k=-1$. 
Combining \eqref{R2:eq:1} and \eqref{R2:eq:2}, we obtain
\begin{equation}\label{eq:main:2}
\|(f)_{\lambda}\|_{B^{r}_{p,q}(\rho,\lambda)} =\left(\sum_{k=-1}^{\infty} 2^{krq}\|\Delta_{k} f(\lambda \cdot)\|_{L^p(\rho,{\lambda})}^q \right)^{1/q} \lesssim \lambda^{-d/p} C_{\la,r,\eta,q} \|f\|_{B^{r}_{p,q}(\rho)}, 
\end{equation}
where 
\[
\begin{aligned}
C_{\la,r,q,\eta}:= 
\begin{cases}
\la^{\eta},\quad &r<\eta,\\
\la^{\eta}(1+\ln(\la))^{1/q},\quad &r=\eta,\\
\la^{r},\quad &r>\eta.
\end{cases}
\end{aligned}
\]

Combining \eqref{eq:main:1} and \eqref{eq:main:2}, we obtain the result for all $\la=2^{N_0}$, $N_0 \in \mathbb{Z}$. Finally, for general rescaling parameter $\lambda\in (0,\infty)$, we apply the standard reduction argument in \cite[Remark 2.19]{BCD11}, which completes the proof.

\end{document}